\documentclass[preprint,3p,times]{elsarticle}
\usepackage[toc,page,title,titletoc,header]{appendix}
\usepackage{amsmath}
\usepackage{amssymb}
\usepackage{amsthm}
\usepackage{mathrsfs}
\usepackage{graphicx}
\usepackage{epsfig}
\usepackage{tikz,pgfplots,tikz-3dplot}
\usepackage{epstopdf}
\usetikzlibrary{fit}
\biboptions{sort&compress}

\newtheorem{thm}{Theorem}[section]
\newtheorem{lemma}[thm]{Lemma}
\newtheorem{rem}[thm]{Remark}
\newtheorem{myDef}{Definition}[section]

\newcommand{\nn}{\nonumber}

\def\epsilon{\varepsilon} 

\begin{document}
\begin{frontmatter}
\title{
%
Structure-preserving finite element methods 
for computing dynamics of rotating Bose–Einstein condensate
}

\author[1]{Meng Li}
\address[1]{School of Mathematics and Statistics, Zhengzhou University, Henan Academy of Big Data, 
Zhengzhou 450001, China}
\ead{limeng@zzu.edu.cn}
\author[2]{Junjun Wang}
\address[2]{School of Mathematics and Statistics, Pingdingshan University, Pingdingshan 467000, China}
\author[2]{Zhen Guan}
\author[3]{Zhijie Du}
\address[3]{School of Science, Wuhan University of Technology, Wuhan 430070, China}

\begin{abstract}
This work is concerned with the construction and analysis of structure-preserving Galerkin methods for computing the dynamics of rotating
Bose–Einstein condensate (BEC) based on
the Gross-Pitaevskii equation with angular momentum rotation. Due to the presence of the rotation term, constructing finite element methods (FEMs) that preserve both mass and energy remains an unresolved issue, particularly in the context of nonconforming FEMs. 
Furthermore, in comparison to existing works, we provide a comprehensive convergence analysis, offering a thorough demonstration of the methods' optimal and high-order convergence properties.
Finally, extensive numerical
results are presented to check the theoretical analysis of the structure-preserving numerical method for rotating BEC, and the quantized vortex lattice's behavior is scrutinized through a series of numerical tests.
\end{abstract} 


\end{frontmatter}

\setcounter{equation}{0}
\section{Introduction} \label{sec:intro}
The phenomenon of Bose–Einstein condensate (BEC) occurs when a sparse gas of a particular type of bosons is confined by a potential and subsequently cooled to extremely low temperatures approaching the absolute minimum of $0$ Kelvin. Since the pioneering experimental achievement of generating a quantized vortex in a gaseous BEC \cite{matthews1999vortices,madison2000vortex,abo2001observation,raman2001vortex}, 
there has been noteworthy advancement in both experimental and theoretical fronts in the field of research \cite{aftalion2001vortices,penckwitt2002nucleation,adhikari2002effect,coddington2004experimental,bao2005dynamics,bao2005ground,bao2006dynamics,zhang2007dynamics,bao2009generalized,Bao2013Kinetic,besse2017high,henning2017finite,bao2019ground,da2023vortex,chen2023second}. 
The topic of this work is to construct and analyze the structure-preserving finite element methods (FEMs) for a special case of BEC in a rotational framework. One notable characteristic of a BEC is its superfluid behavior.  
It is necessary to confirm the formation of vortices with quantized circulation for the purpose of  
distinguishing a superfluid from a normal fluid at the quantum level, which in experimental setups can be induced by rotating the condensate that can be achieved by applying laser beams to the magnetic trap to create a stirring potential. 
If the rotational speed is sufficiently high, these vortices become detectable (as indicated in Ref. \cite{abo2001observation}). 
Notably from the analytical proof in \cite{seiringer2002gross}, the equilibrium velocity of the BEC no longer aligns with that of solid body rotation, and the breaking of rotational symmetry can be observed. 
The frequency of rotation significantly influences the number of vortices in a BEC, yet insufficient rotational speeds lead to a lack of vortices, and overly rapid speeds can cause the BEC's destruction by overpowering centrifugal forces. Refer to the related analytical and numerical results in \cite{aftalion2011non,bao2005ground,correggi2013inhomogeneous,bao2006dynamics,Bao2013Kinetic} and references therein.

This paper is devoted to the Galerkin approximations to the Gross-Pitaevskii equation (GPE) with an angular momentum rotation term in two/three dimensions for modeling a rotating BEC \cite{pitaevskii2016bose,bao2005ground}
\begin{align}\label{eqn:model}
i\partial_tu(\mathbf x, t)=\left[-\frac{1}{2}\nabla^2+V(\mathbf x)-\Omega L_z+\beta|u(\mathbf x, t)|^2\right]	u(\mathbf x, t),
\qquad \mathbf x\in U\subset \mathbb R^d,\quad t>0,
\end{align}
with the homogeneous Dirichlet boundary condition 
\begin{align}\label{eqn:model:dirichlet}
	u(\mathbf x, t)=0,\qquad \mathbf x\in\Gamma=\partial U,\quad t\geq 0,
\end{align}
and the initial condition 
\begin{align}\label{eqn:model:initial}
	u(\mathbf x, 0)=u^0(\mathbf x),\qquad \mathbf x\in U.
\end{align}
Here $\mathbf x=(x, y)$ in two dimensions (2D) and $\mathbf x=(x, y, z)$ in three dimensions (3D), $t$ is time, $u:=u(\mathbf x, t)$ is the complex-valued wave function, 
$\Omega$ denotes a dimensionless constant corresponding to the angular speed of the laser beam in experiments, and $u^0(\mathbf x)$ is a given complex-valued function.
The parameter $\beta$ is a dimensionless constant characterizing the interaction between particles in the rotating BEC that can be positive for repulsive interaction and negative for attractive interaction. $V(\mathbf x)$ is a real-valued function corresponding to the external trap potential that can confines the BEC by adjusting $V(\mathbf x)$ to some trap frequencies. Usually in numerical experiments,  $V(\mathbf x)$ can be selected as a harmonic potential, i.e., a quadratic polynomial: 
\begin{equation}\label{eqn:V}
V(\mathbf x)= 
\left\{
\begin{aligned}
	&(\gamma_x^2x^2+\gamma_y^2{\color{red}y^2})/2,&d=2,\\
	&(\gamma_x^2x^2+\gamma_y^2{\color{red}y^2}+\gamma_z^2{\color{red}z^2})/2,&d=3,
\end{aligned}
\right.
\end{equation}
with the constants $\gamma_x$, $\gamma_y$ and $\gamma_z$.
But in real world, $V(\mathbf x)$ may be a possibly rough/discontinuous potential. 
$L_z$ denotes the $z$-component of the angular momentum:
\begin{align}\label{eqn:def-Lz}
L_z=-i(x\partial_y-y\partial_x)=-i\partial_\theta, 	
\end{align}
where $(r, \theta)$ and  $(r, \theta, z)$ denote the polar coordinates in 2D and the cylindrical coordinates in 3D, respectively. 
For the derivation,
well-posedness and dynamical properties of the GPE \eqref{eqn:model} with  (i.e., $\Omega\neq 0$)  and
without (i.e., $\Omega= 0$) an angular momentum rotation term, we can refer to \cite{cazenave2003semilinear,hao2007global,lieb2006derivation}
and the references therein.
Obviously, the system \eqref{eqn:model}-\eqref{eqn:model:initial} converses the total mass 
\begin{equation}\label{eqn:mass-conservation}
N(u(\cdot, t)):=\int_U|u(\mathbf x, t)|^2d\mathbf x	\equiv N(u(\cdot, 0)) =N(u^0),\qquad t\geq 0,
\end{equation}
and the total energy 
\begin{equation}\label{eqn:energy-conservation}
E(u(\cdot, t)):=\int_U\left[\frac{1}{2}|\nabla u(\mathbf x, t)|^2+V(\mathbf x)|u(\mathbf x, t)|^2+\frac{\beta}{2}|u(\mathbf x, t)|^4-\Omega\bar u(\mathbf x, t)L_zu(\mathbf x, t)\right]d\mathbf x	\equiv E(u(\cdot, 0)) =E(u^0),\qquad t\geq 0,
\end{equation}
where $\bar u(\mathbf x, t)$ represents the conjugate of $u(\mathbf x, t)$.

As we all know, conservative schemes consistently outperform nonconservative ones. The crucial factor behind this lies in their ability to preserve certain invariant properties, allowing them to capture intricate details of physical processes. Viewed from this perspective, discrete schemes that maintain the invariant properties of the original continuous model serve as a criterion for assessing the success of numerical simulations. 
Compared with other spatial discretization methods, the usage of FEMs possesses at least the following two advantages: one is the ability to compute the GPE with a rough/discontinuous potential $V(\mathbf x)$, as seen in \cite{williams1998achieving,zapata1998josephson,nikolic2013dipolar,henning2017crank}; the other is that the FEMs offer an additional advantage by effortlessly integrating with mesh adaptivity, a feature that can be beneficial in effectively addressing localized vortices in BEC. 
There exist some well-known numerical works focusing on the conservative FEMs of the particular case of \eqref{eqn:model},  without the angular momentum rotation term \cite{sanz1984methods,akrivis1991fully,tourigny1991optimal}. 
However, the coupling conditions between the time step size
and the spatial mesh width cannot be removed in these references. Still for the above particular model, Henning and Peterseim \cite{henning2017crank} studied the unconditional error analysis for its modified Crank-Nicolson FEM, which has both the mass and energy conservations, but the theoretical analysis is still limited to the repulsive interaction case (i.e., $\beta>0$). The only finite element (FE) work concerning the model \eqref{eqn:model} with the angular momentum rotation term 
can be documented in \cite{henning2017finite} by Henning and Malqvist. But the scheme is only mass-conserving. The detailed theoretical analysis of the classical mass- and energy-conserving scheme with Sanz-Serna in time, for the model \eqref{eqn:model}, remains an unexplored issue and constitutes a key contribution in this paper.

Furthermore, of greater significance, constructing a conservative scheme for the nonconforming FEM poses a substantial challenge. Due to the presence of the rotation term and since the FE space is not included in the continuous Sobolev space $H_0^1(\Omega)$, the fully discrete scheme based on the Sanz-Serna temporal method cannot be proved to be conservative in the senses of both the discrete mass and energy. In this work, we introduce an innovative stabilizing term into the non-conservative scheme. This term does not affect the convergence rate, and the resulting numerical scheme successfully preserves both mass and energy conservation properties.
 This conservation-adjusting technique is the first application in developing the conservative scheme for the nonconforming FEM, which is another important contribution of this work. 
Additionally, this study offers a comprehensive proof of the boundedness of numerical solutions, utilizing the Gagliardo-Nirenberg inequality and the Cl{\'e}ment interpolation. 

The other main contribution of this paper is the establishment of unconditional error estimates for the proposed conservative schemes. These estimates hold without the existence of any time-space step coupling condition and include optimal $L^2$ and $H^1$ estimates, along with high-accuracy convergence rates in the $H^1$-norm (while maintaining computational complexity within reasonable limits). The theoretical analytical approach draws inspiration from the existing technique \cite{henning2017crank} for the nonlinear Schr{\"o}dinger equation. However, a distinct aspect is our ability to handle cases involving attractive interactions (i.e., $\beta<0$). Additionally, the inclusion of the rotation term introduces significant complexity to the convergence analysis, and the nonconforming case exhibits notable differences from the literature \cite{henning2017crank}. 
Furthermore, different from \cite{henning2017crank,henning2017finite}, we also achieve the high-order convergence rates in space for both numerical schemes.  
This contribution stands out as one of the key highlights of this work.

{\bf{Outline.}} 
In Section \ref{sec2}, we introduce the discretized schemes, encompassing both the time-discrete method and the fully discrete method. Detailed proofs are provided for the discrete mass and energy conservations inherent in the proposed schemes. Additionally, we establish the boundedness of the fully discrete solutions in the case of repulsive interaction, leveraging the Gagliardo–Nirenberg inequality and the Cl{\'e}ment interpolation technique. 
Section \ref{sec3} presents the main convergence results of the fully discrete conforming and nonconforming schemes. Section \ref{sec4} is devoted to the proof of the convergence results. Some numerical experiments are provided in Section \ref{sec5} to check the theoretical results, and the tests also include the vortex lattice dynamics in rotating BEC by our numerical schemes. Some conclusions are drawn in Section \ref{sec6}. Finally, the brief convergence analysis for the nonconforming FEM is provided in Appendix. 

\textbf{Notations.} 
We denote the standard notations for Sobolev spaces $W^{s,p}(U)$ with 
$|\cdot|_{W^{s,p}}$ and $\|\cdot\|_{W^{s,p}}$ being its seminorm and norm, respectively. For the case of $p=2$, we denote the notations $H^s(U)=W^{s,2}(U)$ and 
$H_0^1(U):=\{v\in H^1(U): v|_{\partial\Omega}=0\}$. When $s=0$, $L^p(U)=W^{0,p}(U)$.
For a finite time interval $J:=[0, T]$ with a positive constant $T$, 
we define the Bochner space for a strongly measurable function $\phi$ by  
$L^p(J, U):=\{\phi: J\rightarrow X;  \|\phi\|_{L^p(J, U)}<\infty\}$ with 
\begin{equation*}
\|\phi\|_{L^p(J, U)}:= 
\left\{
\begin{aligned}
	&\left[\int_0^T\|\phi(t)\|_{L^p(U)}^p\text{d}t\right]^{\frac1p},~~~~~~~~~~~~~~&\text{if}~p<+\infty,\\
	&\inf\left\{C: \|\phi\|_{L^\infty(U)}\leq C~~\text{a.e.}~\text{on}~J\right\},~~~~~~~~~~~~~~&\text{if}~p=+\infty.
\end{aligned}
\right.
\end{equation*}

\section{Discretized schemes}\label{sec2}
For the sake of simplicity in presentation, we only discuss the methods in 2D. The theoretical results in 3D and the corresponding analytical techniques are straightforward and similar to those in 2D.

\subsection{Time-discrete method}
Consider a family of admissible partitions of the time interval $J$ denoted by $\{I_n; n\in \mathbb N, 0\leq n\leq N-1\}$, where 
$I_n:=(t_{n}, t_{n+1}]$ with $0=t_0<t_1<\ldots<t_N=T$ and $\tau_n=|I_n|$. Additionally, we assume that the partitions are quasi-uniform, i.e. 
$\tau:=\max_{1\leq n\leq N}\{\tau_n\}\leq C_q \min_{1\leq n\leq N}\{\tau_n\}$ for any partitions, where $C_q$ is a positive constant independent of the discretization. Subsequently, we consider the time-discrete Crank-Nicolson approximation for the GPE \eqref{eqn:model}-\eqref{eqn:model:initial}. 

\begin{myDef}
	(Time-discrete method for GPE) Let $u_\tau^0:=u^0$. Then for $n\geq 1$, we define the following time-discrete system, which is to find $u_\tau^{n+1}\in H_0^1(U)$, $0\leq n\leq N-1$ such that 
	\begin{align}\label{eqn:timediscrete}
		iD_\tau u^{n+\frac{1}{2}}_\tau=-\frac{1}{2}\Delta \hat u_\tau^{n+\frac{1}{2}}+V\hat u_\tau^{n+\frac{1}{2}}-\Omega L_z\hat u_\tau^{n+\frac{1}{2}}
+\beta\frac{|u_\tau^{n}|^2+|u_\tau^{n+1}|^2}{2}	\hat u_\tau^{n+\frac{1}{2}},
	\end{align}
where $D_\tau u^{n+\frac{1}{2}}_\tau:=(u_\tau^{n+1}-u_\tau^{n})/{\tau_n}$ and $\hat u_\tau^{n+\frac{1}{2}}:=(u_\tau^n+u_\tau^{n+1})/2$.
\end{myDef}

The well-posedness and convergence of the system \eqref{eqn:timediscrete} will be shown in Section \ref{sec4}. Furthermore, we can prove that the time-discrete scheme \eqref{eqn:timediscrete} keeps the mass and energy conservations as follows. 
\begin{thm}\label{thm-timediscrete-conservation}
	The time-discrete system \eqref{eqn:timediscrete} is conservative in the senses of the total mass and energy:
	\begin{equation}\label{eqn-timediscrete-conservation}
	 N(u_\tau^n)=N(u^0),\qquad	E(u_\tau^n)=E(u^0),\qquad 0\leq n\leq N.
	\end{equation}
\end{thm}
\begin{proof}
Taking the inner product of \eqref{eqn:timediscrete} with $\hat u_\tau^{n+\frac{1}{2}}$, and selecting the imaginary part of the resulting, 
the mass conservation is directly obtained. 
Furthermore, taking the inner product of \eqref{eqn:timediscrete} with $D_\tau u_\tau^{n+\frac{1}{2}}$, and selecting the real part of the resulting, one obtains 
	\begin{align}\label{eqn-timediscrete-conservation-proof1}
	&\frac{1}{4\tau_n}\left[\int_U(\nabla u_\tau^{n+1})^2d\mathbf x-\int_U(\nabla u_\tau^{n})^2d\mathbf x\right]	
	+\frac{1}{2\tau_n}\left[\int_U V(u_\tau^{n+1})^2d\mathbf x-\int_UV( u_\tau^{n})^2d\mathbf x\right]
	-\Omega Re(L_z\hat u_\tau^{n+\frac{1}{2}}, D_\tau u_\tau^{n+\frac{1}{2}})	\nonumber\\
	&+\frac{\beta}{4\tau_n}\left[\int_U(u_\tau^{n+1})^4d\mathbf x-\int_U(u_\tau^{n})^4d\mathbf x\right]=0.
	\end{align}
For the third term of \eqref{eqn-timediscrete-conservation-proof1}, there holds  
\begin{align}\label{eqn-timediscrete-conservation-proof2}
&	Re(L_z\hat u_\tau^{n+\frac{1}{2}}, D_\tau u_\tau^{n+\frac{1}{2}})=\frac{1}{2\tau_n}Im\left([x\partial_y-y\partial_x](u_\tau^{n+1}+u_\tau^{n}), u_\tau^{n+1}-u_\tau^{n}\right)\nonumber\\
&~~~=\frac{1}{2\tau_n}\left[Im\left([x\partial_y-y\partial_x]u_\tau^{n+1}, u_\tau^{n+1}\right)-Im\left([x\partial_y-y\partial_x]u_\tau^{n}, u_\tau^{n}\right)\right]\nn\\
&~~~~~~~+\frac{1}{2\tau_n}\left[Im\left([x\partial_y-y\partial_x]u_\tau^{n}, u_\tau^{n+1}\right)-Im\left([x\partial_y-y\partial_x]u_\tau^{n+1}, u_\tau^{n}\right)\right]\nonumber\\
&~~~=:\frac{1}{2\tau_n}\left[Im\left([x\partial_y-y\partial_x]u_\tau^{n+1}, u_\tau^{n+1}\right)-Im\left([x\partial_y-y\partial_x]u_\tau^{n}, u_\tau^{n}\right)\right] + \frac{1}{2\tau_n}A. 
\end{align}
By using integration by parts, we have 
\begin{align}\label{eqn-timediscrete-conservation-proof3}
	A =-Im\left(u_\tau^{n}, [x\partial_y-y\partial_x]u_\tau^{n+1}\right)-Im\left([x\partial_y-y\partial_x]u_\tau^{n+1}, u_\tau^{n}\right)=0.
\end{align}
Substituting \eqref{eqn-timediscrete-conservation-proof2} and \eqref{eqn-timediscrete-conservation-proof3} into \eqref{eqn-timediscrete-conservation-proof1} gives that 
\begin{align}\label{eqn-timediscrete-conservation-proof4}
	&\frac{1}{2}\int_U(\nabla u_\tau^{n+1})^2d\mathbf x	
	+\int_U V(u_\tau^{n+1})^2d\mathbf x
	-
	\Omega Im\left([x\partial_y-y\partial_x]u_\tau^{n+1}, u_\tau^{n+1}\right)+\frac{\beta}{2}\int_U(u_\tau^{n+1})^4d\mathbf x\nonumber\\
	&=\frac{1}{2}\int_U(\nabla u_\tau^{n})^2d\mathbf x	
	+\int_U V(u_\tau^{n})^2d\mathbf x
	-
	\Omega Im\left([x\partial_y-y\partial_x]u_\tau^{n}, u_\tau^{n}\right)+\frac{\beta}{2}\int_U(u_\tau^{n})^4d\mathbf x.
	\end{align}
Since $Re([x\partial_y-y\partial_x]u_\tau^{m}, u_\tau^{m})
	=0$ for any $m\in\mathbb N$, we can remove ``Im" in \eqref{eqn-timediscrete-conservation-proof4}. Hence, the energy conservation is obtained. Therefore, we have completed the proof. 
\end{proof}

\subsection{Fully discrete method}
In this work, we build a unified framework of theoretical results in the conforming and nonconforming FEs. 
In what follows, we present a brief introduction of two types of elements. 
Assume that $U$ is a rectangle in $(x,y)$ plane with edges parallel to the coordinate axes. Let $\mathscr{T}_{h}$  be a regular rectangular subdivision of $U$, and set  $K\in\mathscr{T}_{h}$ and $h=\max\limits_{K}diam(K)$. Let us consider 
\begin{itemize}
	\item Conforming FE space:
	\begin{align}
		V_{h}^C:=\left\{v\in H_0^1(U); v|_{K}\in Q_{11}(K), \forall K\in \mathscr{T}_{h}\right\},
	\end{align}
where $Q_{11}(K)=\operatorname{span}\{1, x, y, xy\}$. 
	\item Nonconforming FE space ($EQ_1^{rot}$):
 \begin{align}
 	V_{h}^{NC}:=\left\{v_h\in L^2(U); v_h|_{K}\in \operatorname{span}\{1, x, y, x^2, y^2\}, \int_F[v_h]ds=0, F\subset\partial K, \forall K\in \mathscr{T}_{h}\right\},
 \end{align}
where 
$[v_h]$ means the jump of $v_h$ across the edge $F$ if $F$ is an internal edge, and $v_h$ itself if $F$ is a boundary edge.  
\end{itemize}

\textbf{Weak formulation}: given an initial data $u^0\in H_0^1(U):=H_0^1(U; \mathbb C)$, find a wave function $u\in L^\infty(J, H_0^1(U))$ with $\partial_tu\in L^\infty(J, H^{-1}(U))$, such that 
\begin{align}\label{eqn:model-vf}
i\left(\partial_tu(\mathbf x, t), \omega\right)=\frac{1}{2}\left(\nabla u(\cdot, t), \nabla\omega\right)+\left(Vu(\cdot, t), \omega\right)-\Omega\left(L_zu(\cdot, t), \omega\right)
+\beta\left(|u(\mathbf x, t)|^2	u(\cdot, t), \omega\right),
\qquad \forall \omega\in H_0^1(U).
\end{align}
We refer to Refs. \cite{cazenave2003semilinear,henning2017finite} for the well-posedness of the system.
Moreover, it is obvious that the system \eqref{eqn:model-vf} conserves the total mass and energy defined in \eqref{eqn:mass-conservation} and \eqref{eqn:energy-conservation}.

For convenience, we denote $V_h$ as a unified FE space, with $V_h=V_h^C$ and $V_h=V_h^{NC}$ for the conforming and nonconforming cases respectively. 
 We define the inner product and corresponding norms piecewisely, given by
\begin{equation*}
(\phi, \psi)_h:=\sum_K\int_K	\phi\cdot\bar \psi \text{d}\mathbf x,
\quad \|\phi\|_{1,h}:=\left(\sum_K|\phi|_{1, K}^2\right)^{1/2},
\end{equation*}
and 
\begin{align*}
\left\langle\phi, \psi\right\rangle:=\sum_K\int_{\partial K}(\phi\cdot\bar\psi)(\mathbf x\cdot\mathbf n^\perp)\text{d}s \quad \text{with}
\quad \mathbf n^\perp=(n_y, -n_x). 
\end{align*}

Define the Ritz projection $R_h: H^1_0(U)\rightarrow V_h$, such that 
\begin{align}\label{eqn:ritz}
	(\nabla v-\nabla R_hv, \nabla \omega_h)_h=0,\qquad \text{for~all~}\omega_h\in V_h.
\end{align}
There exists a generic $h$-independent constant $C_{R_h}$, such that 
\begin{align}
\label{eqn:ritz_error}
 	\|v-R_hv\|_{L^2}+h\|v-R_hv\|_{1,h}\leq C_{R_h}h^s\|v\|_{H^s},\qquad s=1, 2, \quad\text{for~all~}v\in H_0^1(U)\cap H^2(U).
\end{align}
Then we define $I_h$ as the corresponding interpolation operator, such that 
\begin{equation}\label{eqn:Ih_error}
\|u-I_h u\|_{L^2}\leq C_{I_h}h^2\|u\|_{H^2},\qquad \text{and}\qquad 
(\nabla(u-I_h u), \nabla  v_h)_h= 
\left\{
\begin{aligned}
	O(h^2)\|u\|_{H^3}\|v_h\|_{1,h},\qquad&\text{for}~v_h\in V_h^C,\\
	0,~~~~~~~\qquad&\text{for}~v_h\in V_h^{NC}.
\end{aligned}
\right.
\end{equation}

With above preparation, we introduce the following new fully discrete Crank-Nicolson method, which is different from the existing works \cite{bao2013optimal,henning2017finite}. 

\begin{myDef}\label{Def:fully}
	(Fully discrete method for GPE) Let $u_{h,\tau}^0$ be a suitable interpolation of $u^0$. Then for $n\geq 1$, we define the following time-discrete system, which is to find $u_{h,\tau}^{n+1}\in V_h$, $0\leq n\leq N-1$ such that 
	\begin{align}\label{eqn:fullydiscrete}
		i\left(D_\tau u_{h,\tau}^{n+\frac{1}{2}}, \omega_h\right)=\frac{1}{2}\left(\nabla \hat u_{h,\tau}^{n+\frac{1}{2}}, \nabla\omega_h\right)_h+\left(V\hat u_{h,\tau}^{n+\frac{1}{2}}, \omega_h\right)-\Omega \left(L_z\hat u_{h,\tau}^{n+\frac{1}{2}}, \omega_h\right)_h
+\beta\left(\frac{|u_{h, \tau}^{n}|^2+|u_{h, \tau}^{n+1}|^2}{2}	\hat u_{h, \tau}^{n+\frac{1}{2}}, \omega_h\right)+\left\langle S^{n+1}, \omega_h\right\rangle,
	\end{align}
for any $\omega_h\in V_h$, where $\left\langle S^{n+1}, \omega_h\right\rangle$ is defined by 
\begin{equation}\label{eqn:fullydiscrete-stability}
\left\langle S^{n+1}, \omega_h\right\rangle:=
\begin{cases}
0,\qquad &\text{for conforming case,}\\
-i\frac{\Omega}{2}Re\left\langle \hat u_{h,\tau}^{n+\frac{1}{2}}, \omega_h\right\rangle+\frac{\Omega}{2}Im\left\langle u_{h,\tau}^{n+1}, \omega_h\right\rangle,&\text{for nonconforming case.}\\
\end{cases}
\end{equation}
\end{myDef}

We will supply the well-posedness and convergence of the system \eqref{eqn:fullydiscrete} in Section \ref{sec5}. In the following theorem, we prove that the fully discrete scheme \eqref{eqn:fullydiscrete} keeps the total mass and energy conservations. 
\begin{thm}\label{thm-fullydiscrete-conservation} (Mass and energy conservations)
	The fully discrete system \eqref{eqn:fullydiscrete} is conservative in the senses of the total mass and energy:
	\begin{equation}\label{eqn-fullydiscrete-conservation}
	 N(u_{h,\tau}^n)=N(u_{h,\tau}^0),\qquad	E_h^n=E_h^0,\qquad 0\leq n\leq N,
	\end{equation}
	where the discrete energy $E_h^n$ is defined by	\begin{equation}\label{eqn-discrete-energy}
		E_h^n:=\frac{1}{2}\|u_{h,\tau}^n\|_{1,h}^2
		+\left(Vu_{h,\tau}^n, u_{h,\tau}^n\right)
		+\frac{\beta}{2}\|u_{h,\tau}^n\|_{L^4}^4\
		-\Omega Re\left(L_z u_{h,\tau}^{n}, u_{h,\tau}^{n}\right)_h. 		
	\end{equation}

\end{thm}
\begin{proof}
The results are directly obtained for the conforming case, and thus we only show some proof details for the nonconforming case. Setting $\omega_h=\hat u_{h,\tau}^{n+\frac{1}{2}}$ in \eqref{eqn:fullydiscrete}, and selecting the imaginary part, we have 
\begin{align}\label{eqn-fullydiscrete-conservation-proof1}
 \frac{\|u_{h,\tau}^{n+1}\|^2_{L^2}-\|u_{h,\tau}^{n}\|^2_{L^2}}{2\tau_n}=	-\Omega Im\left(L_z\hat u_{h,\tau}^{n+\frac{1}{2}}, \hat u_{h,\tau}^{n+\frac{1}{2}}\right)_h
+Im\left\langle S^{n+1}, \hat u_{h,\tau}^{n+\frac{1}{2}}\right\rangle.
\end{align}
For the first term at the right-hand side of \eqref{eqn-fullydiscrete-conservation-proof1}, by virtue of the integration by parts, it holds 
\begin{align}\label{eqn-fullydiscrete-conservation-proof2}
	&-\Omega Im\left(L_z\hat u_{h,\tau}^{n+\frac{1}{2}}, \hat u_{h,\tau}^{n+\frac{1}{2}}\right)_h
	=\Omega Re\left(\left[x\partial_y-y\partial_x\right]\hat u_{h,\tau}^{n+\frac{1}{2}}, \hat u_{h,\tau}^{n+\frac{1}{2}}\right)_h\nonumber\\
&	=\frac{\Omega}{2}\left[\left(\left[x\partial_y-y\partial_x\right]\hat u_{h,\tau}^{n+\frac{1}{2}}, \hat u_{h,\tau}^{n+\frac{1}{2}}\right)_h
	+
	\left(\hat u_{h,\tau}^{n+\frac{1}{2}}, \left[x\partial_y-y\partial_x\right]\hat u_{h,\tau}^{n+\frac{1}{2}}\right)_h\right]
	=\frac{\Omega}{2}\left\langle\hat u_{h,\tau}^{n+\frac{1}{2}}, \hat u_{h,\tau}^{n+\frac{1}{2}}\right\rangle. 
\end{align}
Moreover, for the second term at the right-hand side of \eqref{eqn-fullydiscrete-conservation-proof1}, we have 	
\begin{align}\label{eqn-fullydiscrete-conservation-proof3}
Im\left\langle S^{n+1}, \hat u_{h,\tau}^{n+\frac{1}{2}}\right\rangle=-\frac{\Omega}{2}\left\langle \hat u_{h,\tau}^{n+\frac{1}{2}}, \hat u_{h,\tau}^{n+\frac{1}{2}}\right\rangle.	
\end{align}
Substituting \eqref{eqn-fullydiscrete-conservation-proof2} and 
\eqref{eqn-fullydiscrete-conservation-proof3} into \eqref{eqn-fullydiscrete-conservation-proof1}, we derive the mass conservation. 

Setting $\omega_h=D_\tau u_{h,\tau}^{n+\frac{1}{2}}$ in \eqref{eqn:fullydiscrete}, and selecting the real part of the resulting, we have 
\begin{align}\label{eqn-fullydiscrete-conservation-proof4}
 \mathscr D + \mathscr E =0,	
\end{align}
where 
\begin{align*}
\mathscr D := \frac{1}{2}Re\left(\nabla \hat u_{h,\tau}^{n+\frac{1}{2}}, \nabla D_\tau u_{h,\tau}^{n+\frac{1}{2}}\right)_h+Re\left(V\hat u_{h,\tau}^{n+\frac{1}{2}}, D_\tau u_{h,\tau}^{n+\frac{1}{2}}\right)+\beta Re\left(\frac{|u_\tau^{n}|^2+|u_\tau^{n+1}|^2}{2}	\hat u_\tau^{n+\frac{1}{2}}, D_\tau u_{h,\tau}^{n+\frac{1}{2}}\right)
\end{align*}
and 
\begin{align*}
\mathscr E := 	-\Omega Re\left(L_z\hat u_{h,\tau}^{n+\frac{1}{2}}, D_\tau u_{h,\tau}^{n+\frac{1}{2}}\right)_h+Re\left\langle S^{n+1}, D_\tau u_{h,\tau}^{n+\frac{1}{2}}\right\rangle.
\end{align*}
The term $\mathscr D$ is equivalent to 
\begin{align}\label{eqn-fullydiscrete-conservation-proof5}
\mathscr D=\frac{\|u_{h,\tau}^{n+1}\|_{1,h}^2-\|u_{h,\tau}^{n}\|_{1,h}^2}{4\tau_n}	
+\frac{(Vu_{h,\tau}^{n+1}, u_{h,\tau}^{n+1}) - (Vu_{h,\tau}^{n}, u_{h,\tau}^{n})}{2\tau_n}
+\frac{\beta}{4\tau_n}\left(\|u_{h,\tau}^{n+1}\|_{L^4}^4-\|u_{h,\tau}^{n}\|_{L^4}^4\right).
\end{align}
The term $\mathscr E$ can be rewritten as 
\begin{align}\label{eqn-fullydiscrete-conservation-proof6}
	\mathscr E&=-\frac{\Omega}{2\tau_n}\left[Re\left(L_z u_{h,\tau}^{n+1}, u_{h,\tau}^{n+1}\right)_h
	-Re\left(L_z u_{h,\tau}^{n}, u_{h,\tau}^{n}\right)_h
	\right]
	+\frac{\Omega}{2\tau_n}\left[Re\left(L_z u_{h,\tau}^{n+1}, u_{h,\tau}^{n}\right)_h
	-Re\left(L_z u_{h,\tau}^{n}, u_{h,\tau}^{n+1}\right)_h
	\right]+Re\left\langle S^{n+1}, D_\tau u_{h,\tau}^{n+\frac{1}{2}}\right\rangle\nonumber\\
	&=:-\frac{\Omega}{2\tau_n}\left[Re\left(L_z u_{h,\tau}^{n+1}, u_{h,\tau}^{n+1}\right)_h
	-Re\left(L_z u_{h,\tau}^{n}, u_{h,\tau}^{n}\right)_h 
	\right]+ \mathscr E_1.
\end{align}
For the term $\mathscr E_1$ in \eqref{eqn-fullydiscrete-conservation-proof6}, using integration by parts, we have 
\begin{align}\label{eqn-fullydiscrete-conservation-proof7}
\mathscr E_1&=\frac{\Omega}{2\tau_n}\left[Re\left(L_z u_{h,\tau}^{n+1}, u_{h,\tau}^{n}\right)_h
	-Re\left(L_z u_{h,\tau}^{n}, u_{h,\tau}^{n+1}\right)_h
	\right]+Re\left\langle S^{n+1}, D_\tau u_{h,\tau}^{n+\frac{1}{2}}\right\rangle\nonumber\\
	&=-\frac{\Omega}{2\tau_n}Im\left\{\left([x\partial_y-y\partial_x] u_{h,\tau}^{n+1}, u_{h,\tau}^{n}\right)_h
	-\left([x\partial_y-y\partial_x] u_{h,\tau}^{n}, u_{h,\tau}^{n+1}\right)_h+\left\langle u_{h,\tau}^{n+1}, u_{h,\tau}^{n}\right\rangle\right\}\nonumber\\
	&=\frac{\Omega}{2\tau_n}Im\left\{\left( u_{h,\tau}^{n+1}, [x\partial_y-y\partial_x]u_{h,\tau}^{n}\right)_h
	+\left([x\partial_y-y\partial_x] u_{h,\tau}^{n}, u_{h,\tau}^{n+1}\right)_h\right\}=0.
\end{align}
Substituting \eqref{eqn-fullydiscrete-conservation-proof5}, \eqref{eqn-fullydiscrete-conservation-proof6} and \eqref{eqn-fullydiscrete-conservation-proof7} into \eqref{eqn-fullydiscrete-conservation-proof4} obtains the energy conservation. 
\end{proof}

In order to obtain the boundedness of the numerical solution, we should first prove 
the following lemma. 
\begin{lemma}\label{lem:gn}
	For any $v_h\in V_h$, no matter for the conforming or nonconforming FE spaces, there holds  
	\begin{align}\label{eqn:gn_inequality}
	\|v_h\|_{L^4}\leq C_{gn}\|v_h\|_{L^2}^{\frac{1}{2}}\|v_h\|_{1,h}^{\frac{1}{2}}. 
	\end{align}
 
\end{lemma}
\begin{proof}
For the case of the conforming FE space, we can obtain \eqref{eqn:gn_inequality} by directly using the Gagliardo–Nirenberg inequality \cite{gagliardo1961proprieta,fila2019gagliardo}. When $V_h=V_h^{NC}$ that means $V_h\nsubset H_0^1(U)$, the Gagliardo-Nirenberg inequality is not applicable straightforwardly. We introduce the Cl{\'e}ment interpolation $\wedge_hv_h\in H_0^1(U)$ on $\mathscr{T}_{h}$, where $v_h\in V_h^{NC}$. It holds that 
\begin{align}\label{eqn:clement}
\|v_h-\wedge_hv_h\|_{L^2}+h	\|v_h-\wedge_hv_h\|_{1,h}\leq Ch\|v_h\|_{1,h}. 
\end{align}
Then, by virtue of the inverse inequality and \eqref{eqn:clement}, and for $\wedge_hv_h\in H_0^1(U)$, we have 
\begin{align*}
\|v_h\|_{L^4}&\leq \|v_h-\wedge_hv_h\|_{L^4}
+ \|\wedge_hv_h\|_{L^4}\nn\\
& \leq Ch^{-\frac{1}{2}}\|v_h-\wedge_hv_h\|_{L^2}+ \|\wedge_hv_h\|_{L^4}\nn\\
&=Ch^{-\frac{1}{2}}\sqrt{\|v_h-\wedge_hv_h\|_{L^2}}\sqrt{\|v_h-\wedge_hv_h\|_{L^2}}+ \|\wedge_hv_h\|_{L^4}\nn\\
&\leq Ch^{-\frac{1}{2}}\sqrt{h\|v_h\|_{1,h}}\sqrt{\|v_h\|}+C\|v_h\|_{L^2}^{\frac{1}{2}}\|v_h\|_{1,h}^{\frac{1}{2}}\nn\\
& \leq C_{gn}\|v_h\|_{L^2}^{\frac{1}{2}}\|v_h\|_{1,h}^{\frac{1}{2}}. 
\end{align*}
We have completed the proof. 
\end{proof}

\begin{thm}\label{thm-boundedness}
	(Boundedness) Assume that one of the following conditions holds, 
	\begin{itemize}
	 \item [(a)] $\beta\geq 0$, 
	 \item 	[(b)] $\beta< 0$ and $\frac{1}{4}+\frac{\beta(C_{gn})^2}{2}\|u_{h, \tau}^0\|_{L^2}^2\geq (C_0)^{-1}>0$. 
	\end{itemize}
Then, it follows that  
\begin{align}\label{eqn:boundedness}
	\|u_{h,\tau}^n\|_{L^2}\leq M_1, \qquad 
	\|u_{h,\tau}^n\|_{1,h}\leq M_2,\qquad 
	1\leq n\leq N,
\end{align}
where $M_1$ and $M_2$ are positive constants. 
\end{thm}

\begin{proof} 
The $L^2$-norm boundedness of $u_{h,\tau}^n$ can be directly derived by the mass conservation \eqref{eqn-fullydiscrete-conservation}. To prove the boundedness of $\|u_{h,\tau}^n\|_{1,h}$, we consider the following two cases. 

	If $\beta\geq 0$, from the conservation property \eqref{eqn-fullydiscrete-conservation}
, we can directly obtain 
\begin{align}\label{eqn_b_pf1}
\frac{1}{2}\|u_{h,\tau}^n\|_{1,h}^2
	+\left(Vu_{h,\tau}^n, u_{h,\tau}^n\right)
&\leq E_h^0-\Omega Im\left(x\partial_yu_{h,\tau}^n-y\partial_xu_{h,\tau}^n, u_{h,\tau}^n\right)\nn\\
&\leq |E_h^0|+C|\Omega| \left(\|\partial_yu_{h,\tau}^n\|_{L^2}\|u_{h,\tau}^n\|_{L^2}+\|\partial_xu_{h,\tau}^n\|_{L^2}\|u_{h,\tau}^n\|_{L^2}\right)
\nn\\
&\leq |E_h^0|+C|\Omega|\|u_{h,\tau}^n\|_{h}\|u_{h,\tau}^n\|_{L^2}
=|E^0|+C|\Omega|\|u_{h,\tau}^n\|_{h}\|u_{h,\tau}^0\|_{L^2}
\nn\\
&\leq |E_h^0|+\frac{1}{4}\|u_{h,\tau}^n\|_{1,h}^2+C\|u_{h,\tau}^0\|_{L^2}^2,
\end{align}
which implies the boundedness of $\|u_{h,\tau}^n\|_{1,h}$. 

If $\beta<0$, from Lemma \ref{lem:gn} and similar as \eqref{eqn_b_pf1}, we have 
\begin{align}
\frac{1}{2}\|u_{h,\tau}^n\|_{1,h}^2
	+\left(Vu_{h,\tau}^n, u_{h,\tau}^n\right)
&= E_h^0-\frac{\beta}{2}\|u_{h, \tau}^n\|_{L^4}^4-\Omega Im\left(x\partial_yu_{h,\tau}^n-y\partial_xu_{h,\tau}^n, u_{h,\tau}^n\right)\nn\\
&\leq |E_h^0|-\frac{\beta(C_{gn})^2}{2}\|u_{h, \tau}^n\|_{L^2}^2\|u_{h, \tau}^n\|_{h}^2
+C|\Omega| \left(\|\partial_yu_{h,\tau}^n\|_{L^2}\|u_{h,\tau}^n\|_{L^2}+\|\partial_xu_{h,\tau}^n\|_{L^2}\|u_{h,\tau}^n\|_{L^2}\right)
\nn\\
&\leq |E_h^0|-\frac{\beta(C_{gn})^2}{2}\|u_{h, \tau}^0\|_{L^2}^2\|u_{h, \tau}^n\|_{h}^2+C|\Omega|\|u_{h,\tau}^n\|_{h}\|u_{h,\tau}^0\|_{L^2}
\nn\\
&\leq |E_h^0|+\left(\frac{1}{4}-\frac{\beta(C_{gn})^2}{2}\|u_{h, \tau}^0\|_{L^2}^2\right)\|u_{h,\tau}^n\|_{1,h}^2+C\|u_{h,\tau}^0\|_{L^2}^2,
\end{align}
which further gives that 
\begin{align}\label{eqn_b_pf2}
	\left(\frac{1}{4}+\frac{\beta(C_{gn})^2}{2}\|u_{h, \tau}^0\|_{L^2}^2\right)\|u_{h,\tau}^n\|_{1,h}^2\leq |E_h^0|+C\|u_{h,\tau}^0\|_{L^2}^2.
\end{align}
Thus from \eqref{eqn_b_pf2}, if the condition (b) holds, we can derive 
\begin{align}
	\|u_{h,\tau}^n\|_{1,h}^2\leq C_0|E_h^0|+CC_0\|u_{h,\tau}^0\|_{L^2}^2,
\end{align}
which means the boundedness of $u_{h,\tau}^n$ in the sense of piecewise $H^1$-norm. 
\end{proof}

\begin{rem}
If $\Omega=0$, the model \eqref{eqn:model} simplifies to the classical NLS. Even in this special case, we have expanded upon the findings presented in Ref. \cite{henning2017crank}. Notably, \cite{henning2017crank} exclusively addresses the scenario of $\beta\geq 0$ and within the confines of conforming FEM. From this point of view, our extension covers a broader range of cases. 
Under specific conditions on the initial values, the boundedness of the numerical solutions $u_{h,\tau}^n$ can be directly derived from Theorem \ref{thm-boundedness}. This allows us to apply the methods from \cite{henning2017crank} to analyze unconditional convergence for any values of $\beta$. 
However, it's worth mentioning that investigating the unconditional convergence of the scheme using the time-space error splitting technique (in \cite{henning2017crank}) may become unnecessary if the numerical solution is bounded. This technique, while effective, can be somewhat cumbersome compared to more traditional methods."

\end{rem}

\begin{rem}
 Compared with Refs. \cite{akrivis1991fully,henning2017crank}, we in this work consider a more complex rotating Gross-Pitaevskii equation. By cleverly constructing numerical schemes, we are able to ensure both the mass and energy conservations for \textbf{conforming and nonconforming} FEMs. In fact, the presence of the rotation term poses inherent challenges to the conservation properties of nonconforming FEMs. Traditional approaches fail to achieve both the mass and energy conservation in this case. 
\end{rem}

\begin{rem}
	In what follows, we will provide a more comprehensive framework for the unconditional error analysis of the FE schemes, removing any restrictions on the parameter $\beta$ and eliminating the need for specific requirements on the initial values.  In this context, the term "unconditional" takes on a more general connotation, as we relax the continuity requirements on the FE spaces, impose no limitations on the parameter and initial value, and also remove the classical restrictions of the time-space mesh ratio. 
	\end{rem}
%

\section{Main results}\label{sec3}
In this section, we present the main convergence results, including optimal and high-order error estimates, and the proof will be given in the subsequent section.
We in this work assume that the solution to \eqref{eqn:model} exists and satisfies that 
\begin{align}\label{eqn:reg}
	&\|u^0\|_{H^2} + \|u\|_{L^\infty((0, T); H^2)}
	+ \|u_t\|_{L^2((0, T); H^2)}
	+ \|u_{tt}\|_{L^2((0, T); H^2)}+ \|u_{ttt}\|_{L^2((0, T); H^1)}\leq C_u.
\end{align}
In the subsequent analysis, for the sake of simplicity, we assume that $\tau_n=\tau$.

\begin{thm}\label{thm:main}
Let $u_{h,\tau}^n\in V_h$ be the solution of the fully discrete scheme \eqref{eqn:fullydiscrete}. Then, under the regularity assumption \eqref{eqn:reg}, there hold 
\begin{itemize}
\item [(a)] (the $L^\infty$-norm boundedness)
\begin{align}\label{eqn:L_infty}
	\sup_{0\leq n\leq N}\|u_{h,\tau}^n\|_{L^\infty}\leq M, 
\end{align}
\item [(b)] (the optimal $L^2$-norm error estimate)
\begin{align}\label{eqn:convergenceL2}
	\sup_{0\leq n\leq N}\|u^n-u_{h,\tau}^n\|_{L^2}
	\leq C(h^2+\tau^2),
\end{align}
\item [(c)] (the optimal $H^1$-norm error estimate)
\begin{align}\label{eqn:convergenceH1}
	\sup_{0\leq n\leq N}\|u^n-u_{h,\tau}^n\|_{1,h}\leq C(h+\tau^2),
\end{align}
\item [(d)] (the high-order $H^1$-norm error estimates)
\begin{align}
	&\sup_{0\leq n\leq N}\|I_hu^n-u_{h,\tau}^n\|_{1,h}\leq C(h^2+\tau^2),\label{eqn:supercloseH1}\\
	&\sup_{0\leq n\leq N}\|u^n-I_{2h}u_{h,\tau}^n\|_{1,h}\leq C(h^2+\tau^2)\label{eqn:superconvergenceH1},
\end{align}
\end{itemize}
where $C>0$ is a constant independent of $h$ and $\tau$. 
\end{thm}

The proof of Theorem \ref{thm:main} will be given in the following section. 

\section{Error analysis for the FEM}\label{sec4}
In this section, we provide an elaborate proof of Theorem \ref{thm:main}.

\subsection{Error analysis for the time-discrete method}
In this subsection, our focus is on investigating the convergence and boundedness of the solution in the context of the time-discrete method \eqref{eqn:timediscrete}. To this end, we first consider a truncated auxiliary problem of the time-discrete system. 

\subsubsection{A truncated auxiliary problem}

Under the regularity assumption \eqref{eqn:reg}, we define 
\begin{equation}
\label{eqn:timebound}
K_0:=\sup_{M\in \mathbb N}\left\{\max_{0\leq m\leq M}\|u^m\|_{L^\infty}\right\}+1.
\end{equation}
The following truncated function plays a pivotal role in convergence analysis.
\begin{myDef} (Truncated function)
	Define the truncated function as 
	\begin{align}
		\label{eqn:cutoff}
\mu_A(s)=s\cdot \chi(s\cdot K_0^{-2}),
	\end{align}
where 
\begin{equation}\label{eqn:cutoffx}
  \chi(x)=\left\{
   \begin{aligned}
   &\,0,  &|x|\in [2, +\infty), \\
   &\exp\left(1+\frac{1}{(|x|-1)^2-1}\right),  &|x|\in [1, 2), \\
   &1,  &|x|\in [0, 1). 
   \end{aligned}
   \right.
  \end{equation}
\end{myDef}
The truncated function $\mu_A(\cdot)$ exhibits the following characteristics.
\begin{lemma}
The truncated function $\mu_A(s)$ belongs to $C^\infty(\mathbb{R})$, and there hold 
\begin{subequations}
\label{eqn:cutoff_pro}
\begin{align}
&\|\mu_A(w)\|_{L^\infty}\leq C_M, \quad  \left|\mu_A(w_1)-\mu_A(w_2)\right|\leq C_\mu|w_1-w_2|, \label{eqn:cutoff_pro_a}\\
&\left|\mu_A(|w_1|^2)-\mu_A(|w_2|^2)\right|\leq C_\mu|w_1-w_2|,\label{eqn:cutoff_pro_b}
\end{align}
\end{subequations}
where $C_M$ and $C_\mu$ are positive and bounded constants.   
\end{lemma}
\begin{proof}
 Based on the simple elementary properties of functions, we can directly obtain the conclusion of this lemma.  
\end{proof}

\begin{rem}
Let $r=s\cdot K_0^{-2}$ in \eqref{eqn:cutoff}, then 
	\begin{align*}
	\mu_A=K_0^2r\cdot \chi(r)=:K_0^2\,\hat{\mu}_A(r).	
	\end{align*}
To study the properties \eqref{eqn:cutoff_pro} of the 
function $\mu_A(s)$, we only need to consider the function $\hat{\mu}_A(r)$.  For giving a vivid description, we plot the figures of $\hat{\mu}_A(r)$, $d\hat{\mu}_A(r)/dr$ and $d^2\hat{\mu}_A(r)/dr^2$ in Figure \ref{figure:1}. As depicted in Figure \ref{figure:1}, we can clearly observe the boundedness of the function $\hat{\mu}_A(r)$ and its first-order and second-order derivatives. 

\begin{figure}[!htp]
\centering
\includegraphics[width=0.32\textwidth]{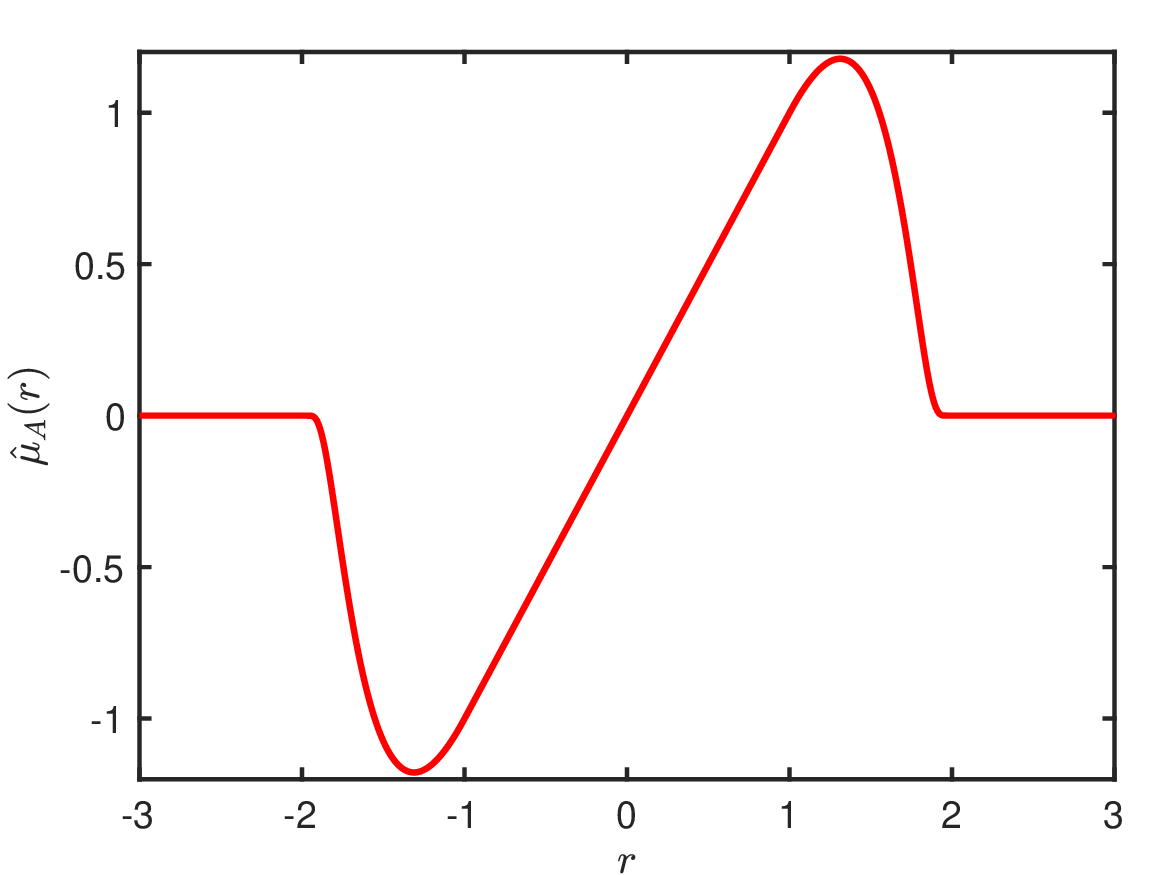}
\includegraphics[width=0.32\textwidth]{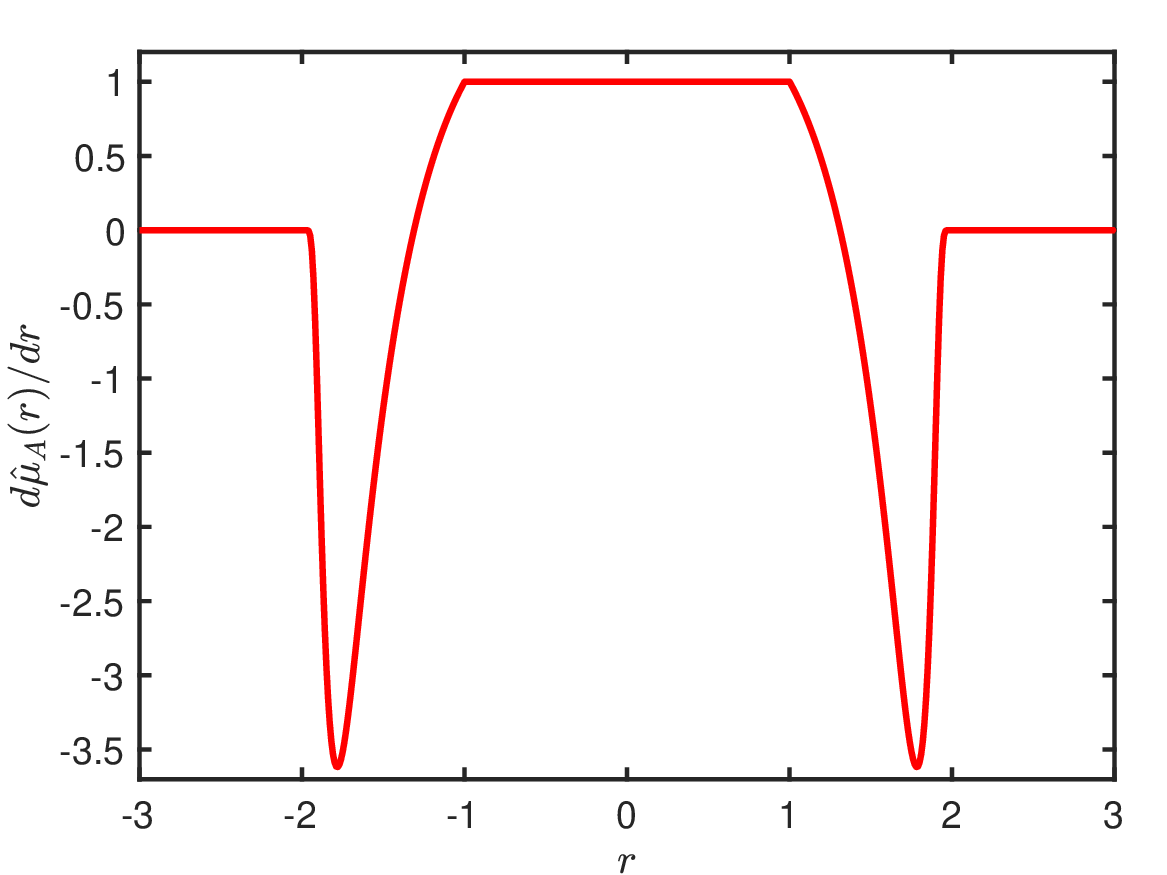}
\includegraphics[width=0.32\textwidth]{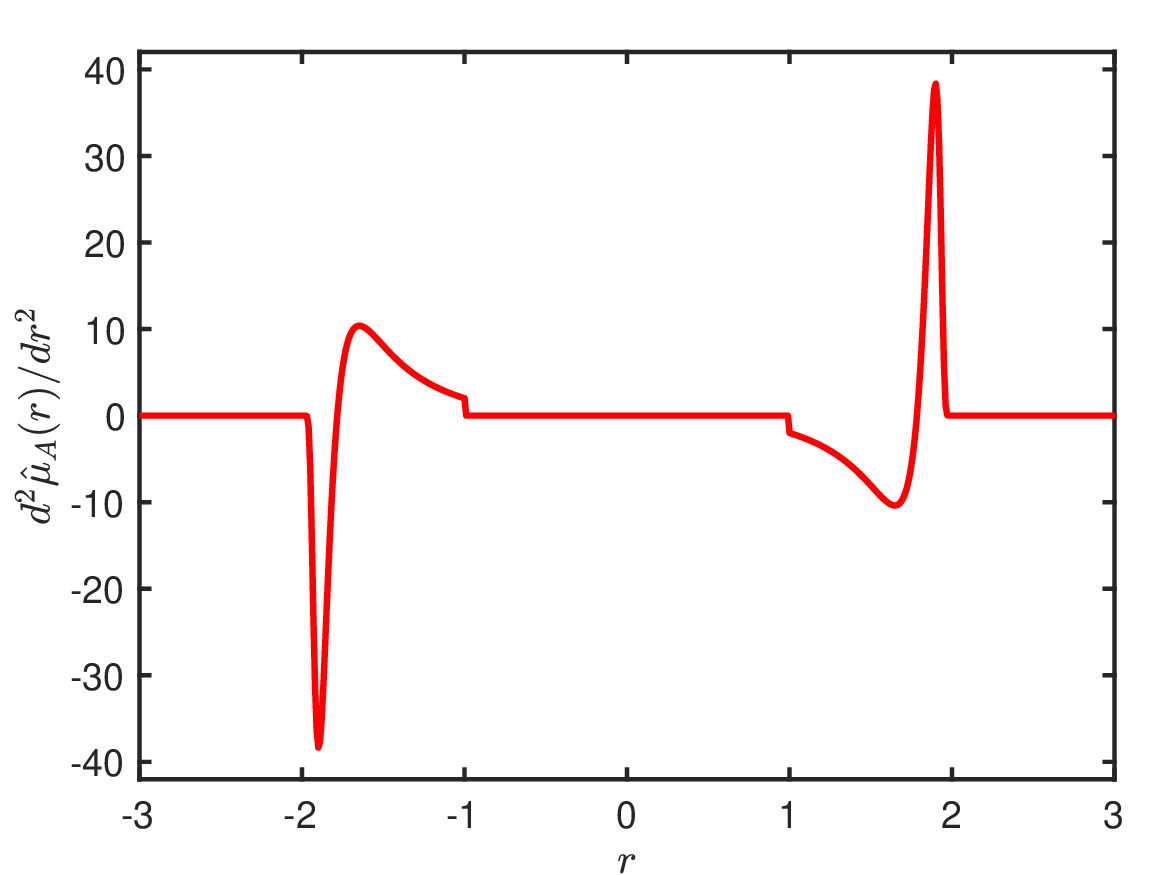}
\caption{Figures of the functions $\hat{\mu}_A(r)$, $d\hat{\mu}_A(r)/dr$ and $d^2\hat{\mu}_A(r)/dr^2$}
\label{figure:1}
\end{figure}

\end{rem}

By using the truncated function $\mu_A(\cdot)$, we introduce the following truncated time-discrete scheme.\begin{myDef} (Time-discrete method with truncation) Let $u_\tau^{T, 0}:=u^0$. Then for 
$n\geq 1$, we define the truncated time-discrete method $u^{T, n+1}_\tau\in H_0^1(U)$ as the solution to 
	\begin{align}\label{eqn:timediscrete_truncate}
		iD_\tau u^{T, n+\frac{1}{2}}_\tau=-\frac{1}{2}\Delta \hat u_\tau^{T, n+\frac{1}{2}}+V\hat u_\tau^{T, n+\frac{1}{2}}-\Omega L_z\hat u_\tau^{T, n+\frac{1}{2}}
+\beta\frac{\mu_A(|u_\tau^{T, n}|^2)+\mu_A(|u_\tau^{T, n+1}|^2)}{2}	\hat u_\tau^{T, n+\frac{1}{2}}. 
	\end{align}
\end{myDef}

\subsubsection{Existence of truncated time-discrete method}

To begin our investigation into the properties of solutions to the truncated system  \eqref{eqn:timediscrete_truncate}, it is crucial to demonstrate that there is at least one solution. To accomplish this, we can rely on the following Brouwer fixed point theorem. 
\begin{lemma}\label{lem:Brouwer} (Brouwer Fixed Point Theorem)
Denote $\overline{S_1(0)} = \{\boldsymbol{\alpha}\in \mathbb C^N; |\boldsymbol{\alpha}|\leq 1\}$ as an unit disk in $\mathbb C^N$. Then, every continuous function $g:\mathbb C^N\rightarrow \mathbb C^N$ with $Re(g(\boldsymbol\alpha),  \boldsymbol\alpha)\geq 0$ with $\boldsymbol\alpha\in \partial S_1(0)$ has a zero in $\overline{S_1(0)}$, i.e. a point $\boldsymbol{\alpha}_0\in \overline{S_1(0)}$ such that $g(\boldsymbol{\alpha}_0)=0$. 
\end{lemma}

In addition, we will utilize the following lemma, which corresponds to a specific scenario of the Vitali convergence theorem \cite{leoni2017first}. 
\begin{lemma}\label{lem:Vit}
(Vitali Convergence Theorem)
A sequence $(g_k)_{k\in \mathbb N}\subset L^2(U)$ convergences strongly to the function $g\in L^2(U)$ if and only if 
\begin{itemize}
\item [(1)] $(g_k)_{k\in \mathbb N}$	 locally converges to $g$ in measure;
\item [(2)] $(g_k)_{k\in \mathbb N}$ is $2$-equi-integrable, which means that for every $\epsilon>0$ there exists a $\delta_\epsilon>0$ such that $\|g_k\|_{L^2(S)}<\epsilon$ for any measurable subset $S\subset U$ with the measure 
$\mu(S)\leq \delta_\epsilon$.
\end{itemize}
	
\end{lemma}

We  prove 
the truncated time-discrete system 
\eqref{eqn:timediscrete_truncate} has at least one solution in $H_0^1(U)$. 
The system \eqref{eqn:timediscrete_truncate} is indeed equivalent to  
\begin{align}
\hat u_\tau^{T, n+\frac{1}{2}}- u_\tau^{T, n}	
-\frac{\Omega\tau}{2}
\left[x\partial_y-y\partial_x\right]\hat u_\tau^{T, n+\frac{1}{2}}
+\tau i\Phi(\hat u_\tau^{T, n+\frac{1}{2}})=0,
\end{align}
where $\Phi(\hat u_\tau^{T, n+\frac{1}{2}})=\Phi(w)|_{w=\hat u_\tau^{T, n+\frac{1}{2}}}$ with 
\begin{align*}
	\Phi(w):=
	-\frac{1}{4}\Delta w+\frac{1}{2}Vw
	+\frac{\beta}{4}
	\left[\mu_A(|u_\tau^{T,n}|^2)+\mu_A(|2w-u_\tau^{T,n}|^2)\right]w.
\end{align*}

\begin{lemma}\label{lem:ex_trun_timediscrete}
	There exists at least one solution $u_\tau^{M, n}\in H_0^1(U)$, $n\geq 1$ for the truncated system \eqref{eqn:timediscrete_truncate}. 
\end{lemma}

\begin{proof}
We only need to prove the existence of the solution to 
\begin{align}\label{eqn:t_trun_ex_pf1} 
	\mathcal G(w) := w - u_\tau^{T, n}	
-\frac{\Omega\tau}{2}
\left[x\partial_y-y\partial_x\right]w
+\tau i\Phi(w)=0,\quad w\in H_0^1(U).
\end{align}	
The proof will be built in the following three steps. \\
\noindent \textbf{Step 1.} We first show the existence in a finite-dimensional subspace 
\begin{align}\label{eqn:t_trun_ex_pf2}
\mathscr X_N:=\left\{\phi_m; m\in \mathbb N	\right\}, 
\end{align}
which denotes a countable basis of the space $H_0^1(U)$. For given $u_{h,\tau}^{T, n}\in H_0^1(U)$, we want to check whether there exists $X_N\in \mathscr X_N$, such that 
\begin{align}\label{eqn:t_trun_ex_pf3}
	\mathcal G(X_N) = X_N - u_\tau^{T, n}	
-\frac{\Omega\tau}{2}
\left[x\partial_y-y\partial_x\right]X_N
+\tau i\Phi(X_N)=0.
\end{align}
Taking the inner product of \eqref{eqn:t_trun_ex_pf3} with $X_N$ gives that 
\begin{align}\label{eqn:t_trun_ex_pf4}
\left(\mathcal G(X_N), 	X_N\right)
=\|X_N\|_{L^2}^2-(u_\tau^{T, n}, X_N)
-\frac{\Omega\tau}{2}
\left(\left[x\partial_y-y\partial_x\right]X_N, X_N\right)
+\tau i(\Phi(X_N), X_N)=0.
\end{align}
It is obvious that 
\begin{align}\label{eqn:t_trun_ex_pf5}
	Im(\Phi(X_N), X_N)=0.
\end{align}
Additionally, since 
\begin{align}\label{eqn:t_trun_ex_pf6}
\left(\left[x\partial_y-y\partial_x\right]X_N, X_N\right)=
-\left(X_N, \left[x\partial_y-y\partial_x\right]X_N\right), 
\end{align}	
we have 
\begin{align}\label{eqn:t_trun_ex_pf7}
Re\left(\left[x\partial_y-y\partial_x\right]X_N, X_N\right)=0.	
\end{align}
Extracting the real part of \eqref{eqn:t_trun_ex_pf4}, and using 
 \eqref{eqn:t_trun_ex_pf5} and \eqref{eqn:t_trun_ex_pf7}, we get 
\begin{align}\label{eqn:t_trun_ex_pf8}
	Re\left(\mathcal G(X_N), 	X_N\right)
	=\|X_N\|_{L^2}^2-Re(u_\tau^{T, n}, X_N)\geq \|X_N\|_{L^2}^2-\|u_\tau^{T, n}\|_{L^2}\|X_N\|_{L^2}
	=\left(\|X_N\|_{L^2}-\|u_\tau^{T, 0}\|_{L^2}\right)\|X_N\|_{L^2},
\end{align}
where the last equality is due to the mass conservation that can be derived similar as \eqref{thm-timediscrete-conservation}. Take sufficiently large $\|X_N\|_{L^2}$ such $Re\left(\mathcal G(X_N), 	X_N\right)\geq 0$, then 
the existence of the solutions to \eqref{eqn:t_trun_ex_pf3}, provided that $u_\tau^{T, n}$ exists. 
Therefore, the solution $Z_N=2X_N-u_\tau^{T, n}$ exists, and $Z_N\in \mathscr X_N$ is indeed the solution to 
\begin{align}\label{eqn:timediscrete_countable}
		i\frac{Z_N-u_\tau^{T, n}}{\tau_n}=-\frac{\Delta Z_N+\Delta u_\tau^{T, n}}{4}+V\frac{Z_N+u_\tau^{T, n}}{2}-\Omega \frac{L_z Z_N+L_z u_\tau^{T, n}}{2}
+\beta\frac{\mu_A(|u_\tau^{T, n}|^2)+\mu_A(|Z_N|^2)}{2}	\frac{Z_N+u_\tau^{T, n}}{2}. 
	\end{align}

\noindent \textbf{Step 2.} We build the estimate of 
\eqref{eqn:timediscrete_countable} and further demonstrate the boundedness of $Z_N$ under the $H^2$-norm. The corresponding results are shown as follows: 
\begin{align}\label{eqn:t_trun_ex_pf9}
	\left\|Z_N-u_\tau^{T, n+1}\right\|_{H^1}+\tau^2\left\|Z_N-u_\tau^{T, n+1}\right\|_{H^2}\leq C_Z\tau^2, 
\end{align}
where $C_Z>0$ is a constant independent of $\tau$. This implies the boundedness of $Z_N$ in the $H^2$-norm, which, in turn, guarantees its boundedness in the $L^\infty$-norm by employing the Sobolev embedding inequality. The estimate \eqref{eqn:t_trun_ex_pf9} follows a similar approach as the error analysis of the time-discrete system in the next subsection, obviating the need to present its proof here.

\noindent \textbf{Step 3.} 
Thanks to the boundedness of of $Z_N$ in $H^2$-norm, 
and applying the Vitali’s theorem shown in Lemma \ref{lem:Vit}, we can derive iteratively the existence of a solution $u_\tau^{T, n+1}\in H_0^1(U)$ to the truncated system \eqref{eqn:timediscrete_truncate}. Here, the nonlinear part in \eqref{eqn:t_trun_ex_pf9} strongly converges to the corresponding nonlinear term in \eqref{eqn:timediscrete_truncate} using Vitali’s theorem, as demonstrated in a similar proof in \cite{henning2017crank}.
\end{proof}

\subsubsection{Uniform $L^\infty$-boundedness of the truncated approximation}

Define 
$e_\tau^{T, n}=u^n-u_\tau^{T, n}$ for $n\geq 0$. Then, from \eqref{eqn:timediscrete_truncate} and \eqref{eqn:model}, we have 
\begin{align}\label{eqn:uniform_pf1}
iD_\tau e_\tau^{T, n+\frac{1}{2}} = -\frac{1}{2}\Delta \hat e_\tau^{T, n+\frac{1}{2}}+V\hat e_\tau^{T, n+\frac{1}{2}}-\Omega L_z\hat e_\tau^{T, n+\frac{1}{2}}
+\mathcal N_\tau^{T, n+\frac{1}{2}}+\mathcal R_\tau^{T, n+\frac{1}{2}},
\end{align}
where 
\begin{align}\label{eqn:uniform_pf1a1}
\mathcal N_\tau^{T, n+\frac{1}{2}}=\beta\left[\frac{|u^{n}|^2+|u^{n+1}|^2}{2}	\hat u^{n+\frac{1}{2}}	
-\frac{\mu_A(|u_\tau^{T, n}|^2)+\mu_A(|u_\tau^{T, n+1}|^2)}{2}	\hat u_\tau^{T, n+\frac{1}{2}}\right],	
\end{align}
and 
\begin{align}\label{eqn:uniform_pf1a2}
\mathcal R_\tau^{T, n+\frac{1}{2}}=&
i\left[D_\tau u^{n+\frac12}-\partial_tu^{n+\frac12}\right]
+\frac12\left[\Delta\hat u^{n+\frac12}-\Delta u^{n+\frac12}\right]
+\Omega\left[ L_z\hat u^{n+\frac12}-L_z u^{n+\frac12}\right]\nn\\
&+\beta\left[|u^{n+\frac12}|^2u^{n+\frac12}-
\frac{|u^n|^2+|u^{n+1}|^2}{2}\hat u^{n+\frac12}\right]. 
\end{align}
Based on the regularity assumption \eqref{eqn:reg}, we readily derive
\begin{align}\label{eqn:uniform_pf1a3}
|\mathcal R_\tau^{T, n+\frac{1}{2}}|_{H^2}\leq C_R\tau^2,
\end{align}
where $C_R$ is a positive constant independent of $\tau$. Moreover, utilizing the definition of the truncated function in \eqref{eqn:cutoff}, we obtain 
\begin{align}\label{eqn:uniform_pf2}
\mathcal N_\tau^{T, n+\frac{1}{2}}&=\beta\left[\frac{|u^{n}|^2+|u^{n+1}|^2}{2}	\hat u^{n+\frac{1}{2}}	
-\frac{\mu_A(|u_\tau^{T, n}|^2)+\mu_A(|u_\tau^{T, n+1}|^2)}{2}	\hat u_\tau^{T, n+\frac{1}{2}}\right]\nonumber\\
&=\beta\left[\frac{\mu_A(|u^{n}|^2)+\mu_A(|u^{n+1}|^2)}{2}	\hat u^{n+\frac{1}{2}}	
-\frac{\mu_A(|u_\tau^{T, n}|^2)+\mu_A(|u_\tau^{T, n+1}|^2)}{2}	\hat u_\tau^{T, n+\frac{1}{2}}\right]\nn\\
&=
\beta\left[\frac{\mu_A(|u^{n}|^2)+\mu_A(|u^{n+1}|^2)}{2}	
-\frac{\mu_A(|u_\tau^{T, n}|^2)+\mu_A(|u_\tau^{T, n+1}|^2)}{2}\right]\hat u^{n+\frac{1}{2}}	+
\beta\frac{\mu_A(|u_\tau^{T, n}|^2)+\mu_A(|u_\tau^{T, n+1}|^2)}{2}	\hat e^{n+\frac{1}{2}}.
\end{align}
Then, using \eqref{eqn:cutoff_pro} in \eqref{eqn:uniform_pf2} gives 
\begin{align*}
\mathcal N_\tau^{T, n+\frac{1}{2}}\leq 
\frac{|\beta|(K_0-1)}{2}	C_\mu\left(|e_\tau^{T,n}|+|e_\tau^{T,n+1}|\right)
+\frac{|\beta|}{2}C_M\left(|e_\tau^{T,n}|+|e_\tau^{T,n+1}|\right),
\end{align*}
which implies that 
\begin{align}\label{eqn:uniform_pf3}
	\|\mathcal N_\tau^{T, n+\frac{1}{2}}\|_{H^s}
	\leq C_{\mathcal N, \tau}\left(\|e_\tau^{T,n}\|_{H^s}+\|e_\tau^{T,n+1}\|_{H^s}\right),\qquad s=0, 1, 2, 
\end{align}
where $C_{\mathcal N, \tau}$ is a positive constant independent of $h$ and $\tau$.

Taking the inner product of the error equation \eqref{eqn:uniform_pf1} with $\hat e_\tau^{T, n+\frac{1}{2}}$ and $\Delta \hat e_\tau^{T, n+\frac{1}{2}}$ respectively, and selecting the imaginary part of the resulting, we can easily derive 
\begin{align}\label{eqn:uniform_pf4}
\|e_\tau^{T, n+1}\|_{H^1}\leq C_{E, \tau}\tau^2, 
\end{align}
provided that $\tau\leq \tau_1$ with a positive constant $\tau_1$. Then, it follows that 
\begin{align}\label{eqn:uniform_pf5}
\|D_\tau e_\tau^{T, n+\frac{1}{2}}\|_{L^2}\leq 2C_{E, \tau}\tau.
\end{align}
Here $C_{E, \tau}$ denotes a positive constant independent of $\tau$ that will be used herein and hereafter, and may be different in different places.
Then from \eqref{eqn:uniform_pf1}, we get 
$\|\Delta \hat e_\tau^{T, n+\frac{1}{2}}\|_{L^2}\leq C_{E, \tau}\tau$, 
which further implies that 
\begin{align}\label{eqn:uniform_pf6}
	\|\Delta e_\tau^{T, n+\frac{1}{2}}\|_{L^2}\leq C_{E, \tau}.
\end{align}
Therefore, by using \eqref{eqn:uniform_pf4} and \eqref{eqn:uniform_pf6}, we get 
\begin{align}\label{eqn:uniform_pf7}
	\|e_\tau^{T, n+1}\|_{H^2}\leq C_{E, \tau}.
\end{align} 
By virtue of the Gagliardo–Nirenberg inequality, we obtain 
\begin{align}\label{eqn:uniform_pf8}
	\|e_\tau^{T, n+1}\|_{L^\infty}\leq 
	C\|e_\tau^{T, n+1}\|_{H^2}^{\frac{3}{4}}\|e_\tau^{T, n+1}\|_{L^2}^{\frac{1}{4}}+C\|e_\tau^{T, n+1}\|_{L^2}
	\leq CC_{E, \tau}\tau^{\frac{1}{2}}\leq \tau^{\frac{1}{4}}.
\end{align}
From \eqref{eqn:uniform_pf7} and \eqref{eqn:uniform_pf8}, if $\tau\leq 1$, it follows that 
\begin{align}
& \|u_\tau^{T, n+1}\|_{H^2}\leq \|e_\tau^{T, n+1}\|_{H^2}+\|u^{n+1}\|_{H^2}
\leq C_{E, \tau}+C_u\leq C_{u, \tau},\\
& \|u_\tau^{T, n+1}\|_{L^\infty}\leq 
\|e_\tau^{T, n+1}\|_{L^\infty} + 
\|u^{n+1}\|_{L^\infty}\leq K_0,
\end{align}
with a positive constant $C_{u, \tau}$.

The proof for the uniform $L^\infty$-boundedness of the truncated approximation is now concluded. Utilizing the definition of the function $\mu_A(\cdot)$ in \eqref{eqn:cutoff}, we establish the relationship:
\begin{align}
\mu_A(u_\tau^{T, n})= 	u_\tau^{T, n},\quad 
\text{for~all~}n\geq 0.
\end{align}
This implies that the time-truncated system \eqref{eqn:timediscrete_truncate} is indeed equivalent to the time-discrete method \eqref{eqn:timediscrete}. Consequently, the $L^\infty$-norm boundedness of $u_\tau^{n}$ naturally follows.

\subsubsection{Convergence and boundedness of the time-discrete method}
From above analysis, we have derived some theoretical results of the time-discrete method, which are given in the following lemma. 
\begin{lemma}\label{lem:time_error1}
	Assume $u$ be the solution of the model \eqref{eqn:model} satisfying the regularized condition \eqref{eqn:reg}. Then there exists 
	$\tau_0=\min\{\tau_1, 1\}>0$ such that when $\tau\leq \tau_0$, the time-discrete system \eqref{eqn:timediscrete} has a unique solution $u_\tau^n$, $n=1,\ldots,N$, satisfying 
	\begin{align}\label{eqn:time_error1a}
	\sup_{0\leq n\leq N}\|u_\tau^n\|_{H^2}\leq C_{u, \tau},\qquad
	 \sup_{0\leq n\leq N}\|u_\tau^n\|_{L^\infty}\leq K_0,
	\end{align}
and the error estimate
\begin{align}\label{eqn:time_error1b}
\sup_{0\leq n\leq N}\left\|u^n-u_\tau^n\right\|_{H^1}+\tau^2\sup_{0\leq n\leq N}\left\|u^n-u_\tau^n\right\|_{H^2}\leq C_{E, \tau}\tau^2,
\qquad 	
\end{align}
holds true, where $C_{u, \tau}$, $K_0$ and $C_{E, \tau}$ are positive constants independent of $\tau$. 
\end{lemma}


We can also enhance the rate of convergence of $\sup_{0\leq n\leq N}\left\|u_\tau^n-u^n\right\|_{H^2}$ to $O(\tau^2)$. In the next subsection, we will demonstrate the necessity of achieving a higher-order convergence rate.
\begin{lemma}\label{lem:time_error2}
	Under the conditions of Lemma \ref{lem:time_error1}, there holds
\begin{align}\label{eqn:time_error2}
\sup_{0\leq n\leq N}\left\|u^n-u_\tau^n\right\|_{H^2}\leq C_{E, \tau}\tau^2,
\qquad 	
\end{align}
provided $\tau\leq \tau_0^*$ with a positive constant $\tau_0^*$, where 
$C_{E, \tau}$ denotes a positive constant independent of $\tau$. 
\end{lemma}
\begin{proof}
Define 
$e_\tau^{n}=u^n-u_\tau^{n}$ for $n\geq 0$.
From \eqref{eqn:timediscrete}  and \eqref{eqn:model}, we have 
\begin{align}\label{eqn:time_error2_pf1}
iD_\tau e_\tau^{n+\frac{1}{2}} = -\frac{1}{2}\Delta \hat e_\tau^{n+\frac{1}{2}}+V\hat e_\tau^{n+\frac{1}{2}}-\Omega L_z\hat e_\tau^{n+\frac{1}{2}}
+\mathcal N_\tau^{n+\frac{1}{2}}+\mathcal R_\tau^{n+\frac{1}{2}},
\end{align}
where $R_\tau^{n+\frac{1}{2}}=R_\tau^{T, n+\frac{1}{2}}$ and 
\begin{align}\label{eqn:time_error2_pf2}
\mathcal N_\tau^{n+\frac{1}{2}}=\beta\left[\frac{|u^{n}|^2+|u^{n+1}|^2}{2}	\hat u^{n+\frac{1}{2}}	
-\frac{|u_\tau^{n}|^2+|u_\tau^{n+1}|^2}{2}	\hat u_\tau^{n+\frac{1}{2}}\right]. 
\end{align}
Taking the inner product of \eqref{eqn:time_error2_pf1} with $D_\tau \Delta e_\tau^{n+\frac{1}{2}}$, and then selecting the real parts, we get 
\begin{align}\label{eqn:time_error2_pf3}
\frac{\|\Delta e_\tau^{n+1}\|_{L^2}^2-\|\Delta e_\tau^{n}\|_{L^2}^2}{4\tau}	
=Re\left(V\hat{e}_\tau^{n+\frac{1}{2}}, D_\tau\Delta e_\tau^{n+\frac{1}{2}}\right)
-\Omega Re\left(L_z\hat{e}_\tau^{n+\frac{1}{2}}, D_\tau\Delta e_\tau^{n+\frac{1}{2}}\right)
+Re\left(\mathcal N_\tau^{n+\frac{1}{2}}, D_\tau\Delta e_\tau^{n+\frac{1}{2}}\right)+Re\left(\mathcal R_\tau^{n+\frac{1}{2}}, D_\tau\Delta e_\tau^{n+\frac{1}{2}}\right). 
\end{align}
We next analyze the four terms on the right-hand side of \eqref{eqn:time_error2_pf2}. Firstly, we have 
\begin{align}
&\left(V\hat{e}_\tau^{n+\frac{1}{2}}, D_\tau\Delta e_\tau^{n+\frac{1}{2}}\right)
=\frac{\left(V\hat{e}_\tau^{n+\frac{1}{2}}, \Delta e_\tau^{n+1}\right)-\left(V\hat{e}_\tau^{n-\frac{1}{2}}, \Delta e_\tau^{n}\right)}{\tau}
-\frac{1}{2}\left(V\left[D_\tau e_\tau^{n+\frac{1}{2}}+D_\tau e_\tau^{n-\frac{1}{2}}\right], \Delta e_\tau^n\right),\label{eqn:time_error2_pf4}\\
&\left(\mathcal N_\tau^{n+\frac{1}{2}}, D_\tau\Delta e_\tau^{n+\frac{1}{2}}\right)=\frac{(\mathcal N_\tau^{n+\frac{1}{2}}, \Delta e_\tau^{n+1})-(\mathcal N_\tau^{n-\frac{1}{2}}, \Delta e_\tau^{n})}{\tau}
  	-\left(\frac{\mathcal N_\tau^{n+\frac{1}{2}}-\mathcal N_\tau^{n-\frac{1}{2}}}{\tau}, \Delta e_\tau^n\right),\label{eqn:time_error2_pf5}\\
  	&\left(\mathcal R_\tau^{n+\frac{1}{2}}, D_\tau\Delta e_\tau^{n+\frac{1}{2}}\right)=\frac{(\mathcal R_\tau^{n+\frac{1}{2}}, \Delta e_\tau^{n+1})-(\mathcal R_\tau^{n-\frac{1}{2}}, \Delta e_\tau^{n})}{\tau}
  	-\left(\frac{\mathcal R_\tau^{n+\frac{1}{2}}-\mathcal R_\tau^{n-\frac{1}{2}}}{\tau}, \Delta e_\tau^n\right).\label{eqn:time_error2_pf6}
\end{align}
Then, similar as \eqref{eqn-timediscrete-conservation-proof2}, we have 
\begin{align}\label{eqn:time_error2_pf7}
-\Omega Re\left(L_z\hat{e}_\tau^{n+\frac{1}{2}}, D_\tau\Delta e_\tau^{n+\frac{1}{2}}\right)
=\frac{\Omega}{2\tau}\left(\left[x\partial y-y\partial x\right]\nabla e_\tau^{n+1}, \nabla e_\tau^{n+1}\right)-
\frac{\Omega}{2\tau}\left(\left[x\partial y-y\partial x\right]\nabla e_\tau^{n}, \nabla e_\tau^{n}\right).
\end{align}
For convenience, we denote
\begin{align}\label{eqn:time_error2_pf8}
	\frac{\mathcal N_\tau^{n+\frac{1}{2}}-\mathcal N_\tau^{n-\frac{1}{2}}}{\tau}=
	\frac{\beta}{4}*\left[\frac{K_1^n-K_1^{n-1}}{\tau}
	+
	\frac{K_2^n-K_2^{n-1}}{\tau}
	+
	\frac{K_3^n-K_3^{n-1}}{\tau}
	+
	\frac{K_4^n-K_4^{n-1}}{\tau}
	\right],
\end{align}
where $K_1^n:= |u^{n}|^2u^n-|u_\tau^{n}|^2u_\tau^n, K_2^n:= |u^{n}|^2u^{n+1}-|u_\tau^{n}|^2u_\tau^{n+1}, K_3^n:= |u^{n+1}|^2u^n-|u_\tau^{n+1}|^2u_\tau^n$ and $K_4^n:= |u^{n+1}|^2u^{n+1}-|u_\tau^{n+1}|^2u_\tau^{n+1}$. We further have 
\begin{align}\label{eqn:time_error2_pf9}
&\frac{K_1^n-K_1^{n-1}}{\tau}
=\frac{\left[|u^n|^2e_\tau^n+(u^n\bar{e}_\tau^n+e_\tau^n\bar{u}_\tau^n)u_\tau^n\right]-\left[|u^{n-1}|^2e_\tau^{n-1}+(u^{n-1}\bar{e}_\tau^{n-1}+e_\tau^{n-1}\bar{u}_\tau^{n-1})u_\tau^{n-1}\right]}{\tau}	\nn\\
&=\frac{|u^n|^2e_\tau^n-|u^{n-1}|^2e_\tau^{n-1}}{\tau}
+\frac{u^n\bar{e}_\tau^nu_\tau^n-u^{n-1}\bar{e}_\tau^{n-1}u_\tau^{n-1}}{\tau}
+\frac{e_\tau^n|u_\tau^{n}|^2-e_\tau^{n-1}|u_\tau^{n-1}|^2}{\tau}\nonumber\\
&=\left(|u^n|^2\frac{e_\tau^n-e_\tau^{n-1}}{\tau}
+\frac{|u^n|^2-|u^{n-1}|^2}{\tau}e_\tau^{n-1}\right)
+\left(\frac{u^n-u^{n-1}}{\tau}\bar{e}_\tau^nu_\tau^n+u^{n-1}\frac{\bar{e}_\tau^n-\bar{e}_\tau^{n-1}}{\tau}u_\tau^n+u^{n-1}\bar{e}_\tau^{n-1}\frac{u_\tau^n-u_\tau^{n-1}}{\tau}\right)\nonumber\\
&~~~~~~+\left(\frac{e_\tau^n-e_\tau^{n-1}}{\tau}(u_\tau^n)^2+e_\tau^{n-1}\frac{\bar{u}_\tau^n-\bar{u}_\tau^{n-1}}{\tau}u_\tau^n+e_\tau^{n-1}\bar{u}_\tau^{n-1}\frac{u_\tau^n-u_\tau^{n-1}}{\tau}\right).
\end{align}
Then, by using \eqref{eqn:reg} and the results given in Lemma \ref{lem:time_error1}, we derive
\begin{align}\label{eqn:time_error2_pf10}
\left(\frac{K_1^n-K_1^{n-1}}{\tau}, \Delta e_\tau^n\right)	
&\leq C\left(\left\|\frac{e_\tau^n-e_\tau^{n-1}}{\tau}\right\|_{L^2}\left\|\Delta e_\tau^n\right\|_{L^2}
+\left\|e_\tau^{n}\right\|_{L^2}\left\|\Delta e_\tau^n\right\|_{L^2}
+\left\|e_\tau^{n-1}\right\|_{L^2}\left\|\Delta e_\tau^n\right\|_{L^2}
+\left\|e_\tau^{n-1}\right\|_{L^4}
\left\|\frac{u_\tau^n-u_\tau^{n-1}}{\tau}\right\|_{L^4}\left\|\Delta e_\tau^n\right\|_{L^2}\right)\nn\\
&\leq C\left(\left\|\frac{e_\tau^n-e_\tau^{n-1}}{\tau}\right\|_{L^2}^2+\left\|\Delta e_\tau^n\right\|_{L^2}^2+\tau^4\right).
\end{align}
Similarly, we have 
\begin{align}\label{eqn:time_error2_pf11}
& \left(\frac{K_2^n-K_2^{n-1}}{\tau}+\frac{K_3^n-K_3^{n-1}}{\tau}+\frac{K_4^n-K_4^{n-1}}{\tau}, \Delta e_\tau^n\right)\leq 	
C\left(\left\|\frac{e_\tau^{n+1}-e_\tau^{n}}{\tau}\right\|_{L^2}^2+\left\|\frac{e_\tau^n-e_\tau^{n-1}}{\tau}\right\|_{L^2}^2+\left\|\Delta e_\tau^n\right\|_{L^2}^2+\tau^4\right).
\end{align}
Hence, from \eqref{eqn:time_error2_pf10} and \eqref{eqn:time_error2_pf11}, we get 
\begin{align}\label{eqn:time_error2_pf12}
	\left(\frac{\mathcal N_\tau^{n+\frac{1}{2}}-\mathcal N_\tau^{n-\frac{1}{2}}}{\tau}, \Delta e_\tau^n\right)\leq C_{\mathcal N}\left(\left\|\frac{e_\tau^{n+1}-e_\tau^{n}}{\tau}\right\|_{L^2}^2+\left\|\frac{e_\tau^n-e_\tau^{n-1}}{\tau}\right\|_{L^2}^2+\left\|\Delta e_\tau^n\right\|_{L^2}^2+\tau^4\right).
\end{align}

In addition, using Taylor formulation, we have 
\begin{align*}
	\left\|\frac{\mathcal R_\tau^{n+\frac{1}{2}}-\mathcal R_\tau^{n-\frac{1}{2}}}{\tau}\right\|_{L^2}\leq C_R\tau^2,
\end{align*}
which implies that 
\begin{align}\label{eqn:time_error2_pf13}
\left(\frac{\mathcal R_\tau^{n+\frac{1}{2}}-\mathcal R_\tau^{n-\frac{1}{2}}}{\tau}, \Delta e_\tau^n\right)\leq C	\left(\left\|\Delta e_\tau^n\right\|_{L^2}^2+\tau^4\right).
\end{align}
Therefore, using \eqref{eqn:time_error2_pf4}-\eqref{eqn:time_error2_pf13} in \eqref{eqn:time_error2_pf3} gives 
\begin{align}\label{eqn:time_error2_pf14}
\frac{\|\Delta e_\tau^{n+1}\|_{L^2}^2-\|\Delta e_\tau^{n}\|_{L^2}^2}{4\tau}	\leq \frac{J^{n+1}-J^n}{\tau}	+ C\left(\left\|\frac{e_\tau^{n+1}-e_\tau^{n}}{\tau}\right\|_{L^2}^2+\left\|\frac{e_\tau^n-e_\tau^{n-1}}{\tau}\right\|_{L^2}^2+\left\|\Delta e_\tau^n\right\|_{L^2}^2+\tau^4\right),
\end{align}
where 
\begin{align*}
	J^n:=\left(V\hat{e}_\tau^{n-\frac{1}{2}}, \Delta e_\tau^{n}\right)
	+
	\frac{\Omega}{2}\left(\left[x\partial y-y\partial x\right]\nabla e_\tau^{n}, \nabla e_\tau^{n}\right)
	+(\mathcal N_\tau^{n-\frac{1}{2}}, \Delta e_\tau^{n})
	+(\mathcal R_\tau^{n-\frac{1}{2}}, \Delta e_\tau^{n}).
\end{align*}
Replacing $n$ by $m$ in \eqref{eqn:time_error2_pf14}, and then summing the inequality from $2$ to $n$, we get 
\begin{align}\label{eqn:time_error2_pf15}
\|\Delta e_\tau^{n+1}\|_{L^2}^2\leq 
\|\Delta e_\tau^{0}\|_{L^2}^2+	J^{n+1}
+C\tau\sum_{m=0}^n\left\|\frac{e_\tau^{m+1}-e_\tau^{m}}{\tau}\right\|_{L^2}^2+C\tau \sum_{m=0}^n\|\Delta e_\tau^{m}\|_{L^2}^2+C\tau^4.
\end{align}
By virtue of Young's inequality, we have 
\begin{align}\label{eqn:time_error2_pf16}
	J^{n+1}\leq \epsilon \|\Delta e_\tau^n\|^2_{L^2}+C_\epsilon\|\mathcal N_\tau^{n-\frac{1}{2}}\|_{L^2}^2
	+C_\epsilon\|\mathcal R_\tau^{n-\frac{1}{2}}\|_{L^2}^2+C_\epsilon \tau^4.
\end{align}
Since $\mathcal N_\tau^{n-\frac{1}{2}}=\mathcal N_\tau^{T, n-\frac{1}{2}}$ that satisfies \eqref{eqn:uniform_pf3}, and $\mathcal R_\tau^{n-\frac{1}{2}}=\mathcal R_\tau^{T, n-\frac{1}{2}}$ that satisfies \eqref{eqn:uniform_pf1a3}, we obtain from \eqref{eqn:time_error2_pf16} and the known results given in Lemma \ref{lem:time_error1} that 
\begin{align}\label{eqn:time_error2_pf17}
	J^{n+1}\leq \epsilon \|\Delta e_\tau^n\|^2_{L^2}+C_\epsilon+C_\epsilon \tau^4.
\end{align}
Selecting sufficient small $\epsilon$, then combining \eqref{eqn:time_error2_pf15} and 
\eqref{eqn:time_error2_pf17} shows 
\begin{align}\label{eqn:time_error2_pf18}
	\|\Delta e_\tau^{n+1}\|_{L^2}^2\leq C\tau\sum_{m=0}^n\left\|\frac{e_\tau^{m+1}-e_\tau^{m}}{\tau}\right\|_{L^2}^2
+C\tau \sum_{m=0}^n\|\Delta e_\tau^{m}\|_{L^2}^2+C\tau^4.
\end{align}

Taking the difference of \eqref{eqn:time_error2_pf1} between two consecutive steps, we obtain 
	\begin{align}\label{eqn:time_error2_pf19}
	 	i\frac{e_\tau^{n+1}-2e_\tau^{n}+e_\tau^{n-1}}{\tau^2} = -\frac{\Delta e_\tau^{n+1}-\Delta e_\tau^{n-1}}{4\tau}+
	 	\frac{V e_\tau^{n+1}-V e_\tau^{n-1}}{2\tau}
	 	-\Omega \frac{L_ze_\tau^{n+1}-L_ze_\tau^{n-1}}{2\tau}
+\frac{\mathcal N_\tau^{n+\frac{1}{2}}-\mathcal N_\tau^{n-\frac{1}{2}}}{\tau}+\frac{\mathcal R_\tau^{T, n+\frac{1}{2}}-\mathcal R_\tau^{T, n-\frac{1}{2}}}{\tau}.
	\end{align}
For simplicity, we denote $e_\tau^{n+1}-e_\tau^n=\tau H_\tau^{n+1}$. Then we rewrite \eqref{eqn:time_error2_pf19} as 
\begin{align}\label{eqn:time_error2_pf20}
		iD_\tau H^{n+\frac{1}{2}}_\tau=-\frac{1}{2}\Delta \hat H_\tau^{n+\frac{1}{2}}+V\hat H_\tau^{n+\frac{1}{2}}-\Omega L_z\hat H_\tau^{n+\frac{1}{2}}
+\frac{\mathcal N_\tau^{n+\frac{1}{2}}-\mathcal N_\tau^{n-\frac{1}{2}}}{\tau}+\frac{\mathcal R_\tau^{T, n+\frac{1}{2}}-\mathcal R_\tau^{T, n-\frac{1}{2}}}{\tau}.
	\end{align}
Taking the inner product of \eqref{eqn:time_error2_pf20} with $\hat H_\tau^{T, n+\frac{1}{2}}$, and selecting the imaginary part of the resulting, 
we have 
\begin{align}\label{eqn:time_error2_pf21}
		\frac{\|H_\tau^{n+1}\|_{L^2}^2-\|H_\tau^{n}\|_{L^2}^2}{2\tau}=
\left(\frac{\mathcal N_\tau^{n+\frac{1}{2}}-\mathcal N_\tau^{n-\frac{1}{2}}}{\tau}, \hat H_\tau^{T, n+\frac{1}{2}}\right)+\left(\frac{\mathcal R_\tau^{T, n+\frac{1}{2}}-\mathcal R_\tau^{T, n-\frac{1}{2}}}{\tau}, \hat H_\tau^{T, n+\frac{1}{2}}\right).
	\end{align}
	Similar as \eqref{eqn:time_error2_pf12} and \eqref{eqn:time_error2_pf13}, we can obtain 
	\begin{align}\label{eqn:time_error2_pf22}
	\left(\frac{\mathcal N_\tau^{n+\frac{1}{2}}-\mathcal N_\tau^{n-\frac{1}{2}}}{\tau}, \hat H_\tau^{T, n+\frac{1}{2}}\right)+\left(\frac{\mathcal R_\tau^{T, n+\frac{1}{2}}-\mathcal R_\tau^{T, n-\frac{1}{2}}}{\tau}, \hat H_\tau^{T, n+\frac{1}{2}}\right)\leq 	
	C\left(\left\|H_\tau^{n+1}\right\|_{L^2}^2+\left\|H^n_\tau\right\|_{L^2}^2\right)+C\tau^4.
	\end{align}
From \eqref{eqn:time_error2_pf21} and \eqref{eqn:time_error2_pf22}, and by using discrete Gronwall's inequality, there exists a positive constant $\tau_2$ such that when $\tau\leq \tau_2$,
\begin{align}\label{eqn:time_error2_pf23}
 \|H_\tau^{n+1}\|_{L^2}=\left\|\frac{e_\tau^{n+1}-e_\tau^n}{\tau}\right\|_{L^2}\leq C_{H}\tau^2,
\end{align}
where $C_H>0$ is a constant independent of $\tau$.

Substituting \eqref{eqn:time_error2_pf23} into \eqref{eqn:time_error2_pf18}, it is sufficient to prove that 
\begin{align}
	\|\Delta e_\tau^{n+1}\|_{L^2}\leq C_{E, \tau}\tau^2,
\end{align}
provided that $\tau<\tau_3$ with a positive constant $\tau_3$. Setting $\tau_0^*=\min\{\tau_0, \tau_2, \tau_3\}$, then if $\tau\leq \tau_0^*$, and thanks to the given results in Lemma \ref{lem:time_error1}, we finally obtain the high-order convergence result \eqref{eqn:time_error2}.
\end{proof}

\begin{rem}
The idea of the proof is to first establish the $L^\infty$-norm boundedness of the solution to the truncated time-discrete method. Subsequently, we discard the truncated function and exclusively investigate the $H^2$-norm convergence of the time-discrete system. 
 In addition, although in the proof it is also existing a mild condition that $\tau$ should be sufficiently small, we can also remove this condition by using the conservation properties given in Theorem \ref{thm-timediscrete-conservation}. But For simplicity, we do not intend to list its proof in the paper. 
\end{rem}


\subsection{Spatial error analysis}
In this subsection, we delve into the optimal and high-order error analysis of the fully discrete method outlined in Definition \ref{Def:fully}. As for the time-discrete method, we also take a detour over its auxiliary system. For simplicity, our primary focus remains on the conforming FEM, while it's noteworthy that a similar convergence analysis can be carried out for the nonconforming FEM, and a concise proof of the nonconforming case is provided in the Appendix. 
\subsubsection{Fully discrete method with truncation}
\begin{myDef}
	(Fully discrete method with truncation) Let $u_{h,\tau}^{T, 0}$ be a suitable interpolation of $u^0$. Then for $n\geq 1$, we define the following truncated fully discrete system, which is to find $u_{h,\tau}^{T, n+1}\in V_h^{C}$, $0\leq n\leq N-1$ such that 
	\begin{align}\label{eqn:conformingfullydiscrete_truncation}
		i\left(D_\tau u_{h,\tau}^{T, n+\frac{1}{2}}, \omega_h\right)=\frac{1}{2}\left(\nabla \hat u_{h,\tau}^{T, n+\frac{1}{2}}, \nabla\omega_h\right)+\left(V\hat u_{h,\tau}^{T, n+\frac{1}{2}}, \omega_h\right)-\Omega \left(L_z\hat u_{h,\tau}^{T, n+\frac{1}{2}}, \omega_h\right)
+\beta\left(\frac{\mu_A(|u_{h, \tau}^{T, n}|^2)+\mu_A(|u_{h, \tau}^{T, n+1}|^2)}{2}	\hat u_{h, \tau}^{T, n+\frac{1}{2}}, \omega_h\right),
	\end{align}
for any $\omega_h\in V_h^C$. 
\end{myDef}

The existence of the solution to the truncated system \eqref{eqn:conformingfullydiscrete_truncation} is covered similar as the proof of truncated time-discrete method. Before obtaining the $L^\infty$-norm boundedness of $u_{h, \tau}^{n}$ for $n\geq 1$, we should first consider the truncated system \eqref{eqn:conformingfullydiscrete_truncation}.

\begin{lemma}\label{lem:spacetruncated_error}
	Suppose the regularity condition \eqref{eqn:reg} holds, and denote $u_{h, \tau}^{T, n}\in V_h$ as a solution of the truncated fully discrete system \eqref{eqn:conformingfullydiscrete_truncation}. Then there exists a positive constant $\tau_0^{**}$, such that when $\tau\leq \tau_0^{**}$, it holds 
	\begin{align}\label{eqn:spacetruncated_error}
	\sup_{0\leq n\leq N}\left\|R_hu_{\tau}^{n}-u_{h, \tau}^{T, n}\right\|	_{L^2}\leq C_{E, h}(h^2+h\tau), 
	\end{align}
where $C_{E, h}$ is a positive constant independent of $h$ and $\tau$. 

\end{lemma}

\begin{proof}
For convenience, we denote 
\begin{align*}
	e_{h,\tau}^{T, n}=u_\tau^n-u_{h, \tau}^{T, n}=\left(u_\tau^n-R_hu_\tau^n\right)+\left(R_hu_\tau^n-u_{h, \tau}^{T, n}\right)=:\rho_{h,\tau}^{T,n}+\theta_{h,\tau}^{T,n}.
\end{align*}
	Subtracting \eqref{eqn:conformingfullydiscrete_truncation} from the variational  formulation of the time-discrete method \eqref{eqn:timediscrete} gives 
	\begin{align}\label{eqn:spacetruncated_error_pf1}
		i\left(D_\tau e_{h,\tau}^{T, n+\frac{1}{2}}, \omega_h\right)=\frac{1}{2}\left(\nabla \hat e_{h,\tau}^{T, n+\frac{1}{2}}, \nabla\omega_h\right)+\left(V\hat e_{h,\tau}^{T, n+\frac{1}{2}}, \omega_h\right)-\Omega \left(L_z\hat e_{h,\tau}^{T, n+\frac{1}{2}}, \omega_h\right)
+\left(\mathcal N_{h,\tau}^{T, n+\frac{1}{2}}, \omega_h\right),\qquad \forall \omega_h\in V_h^C,
	\end{align}
where
\begin{align}\label{eqn:spacetruncated_error_pf2}
\mathcal N_{h,\tau}^{T, n+\frac{1}{2}}=\beta\left[\frac{|u_\tau^{n}|^2+|u_\tau^{n+1}|^2}{2}	\hat u_\tau^{n+\frac{1}{2}}-\frac{\mu_A(|u_{h, \tau}^{T, n}|^2)+\mu_A(|u_{h, \tau}^{T, n+1}|^2)}{2}	\hat u_{h, \tau}^{T, n+\frac{1}{2}}\right].	
\end{align}
Then, by using the definition of the projection $R_h$, \eqref{eqn:spacetruncated_error_pf1} is equivalent to the following one 
\begin{align}\label{eqn:spacetruncated_error_pf3}
		i\left(D_\tau \theta_{h,\tau}^{T, n+\frac{1}{2}}, \omega_h\right)=
		&-i\left(D_\tau \rho_{h,\tau}^{T, n+\frac{1}{2}}, \omega_h\right)+\frac{1}{2}\left(\nabla \hat \theta_{h,\tau}^{T, n+\frac{1}{2}}, \nabla\omega_h\right)+\left(V\hat \rho_{h,\tau}^{T, n+\frac{1}{2}}, \omega_h\right)+\left(V\hat \theta_{h,\tau}^{T, n+\frac{1}{2}}, \omega_h\right)\nn\\
		&-\Omega \left(L_z\hat \rho_{h,\tau}^{T, n+\frac{1}{2}}, \omega_h\right)
		-\Omega \left(L_z\hat \theta_{h,\tau}^{T, n+\frac{1}{2}}, \omega_h\right)
+\left(\mathcal N_{h,\tau}^{T, n+\frac{1}{2}}, \omega_h\right),\qquad \forall \omega_h\in V_h^C.
	\end{align}
Denoting $\omega_h=\hat \theta_{h,\tau}^{T, n+\frac{1}{2}}$ in \eqref{eqn:spacetruncated_error_pf3}, and taking the imaginary part of the resulting, we have 
\begin{align}\label{eqn:spacetruncated_error_pf4}
\frac{\|\theta_{h,\tau}^{T,n+1}\|_{L^2}^2-\|\theta_{h,\tau}^{T,n}\|_{L^2}^2}{2\tau}
= &-Re\left(D_\tau \rho_{h,\tau}^{T, n+\frac{1}{2}}, \hat \theta_{h,\tau}^{T, n+\frac{1}{2}}\right)	+Im\left(V\hat \rho_{h,\tau}^{T, n+\frac{1}{2}}, \hat \theta_{h,\tau}^{T, n+\frac{1}{2}}\right)-\Omega Im\left(L_z\hat \rho_{h,\tau}^{T, n+\frac{1}{2}}, \hat \theta_{h,\tau}^{T, n+\frac{1}{2}}\right)\nn\\
		&
		-\Omega Im\left(L_z\hat \theta_{h,\tau}^{T, n+\frac{1}{2}}, \hat \theta_{h,\tau}^{T, n+\frac{1}{2}}\right)
+Im\left(\mathcal N_{h,\tau}^{T, n+\frac{1}{2}}, \hat \theta_{h,\tau}^{T, n+\frac{1}{2}}\right)
\end{align}
From Lemma \ref{lem:time_error2} and using \eqref{eqn:ritz_error}, we obtain 
\begin{align}\label{eqn:spacetruncated_error_pf5}
	&Re\left(D_\tau \rho_{h,\tau}^{T, n+\frac{1}{2}}, \hat \theta_{h,\tau}^{T, n+\frac{1}{2}}\right)
	\leq \left\|D_\tau \rho_{h,\tau}^{T, n+\frac{1}{2}}\right\|_{L^2}\left\|\hat \theta_{h,\tau}^{T, n+\frac{1}{2}}\right\|_{L^2}
	\leq C_{R_h}h^2\left\|D_\tau u_{\tau}^{n+\frac{1}{2}}\right\|_{H^2}\left\|\hat \theta_{h,\tau}^{T, n+\frac{1}{2}}\right\|_{L^2}\nn\\
	&~~~~~\leq C_{R_h}h^2\left(\left\|D_\tau u^{n+\frac{1}{2}}\right\|_{H^2}+\left\|D_\tau e_\tau^{n+\frac{1}{2}}\right\|_{H^2}\right)\left\|\hat \theta_{h,\tau}^{T, n+\frac{1}{2}}\right\|_{L^2}
	\leq C\left(\|\theta_{h,\tau}^{T, n+1}\|_{L^2}^2+\|\theta_{h,\tau}^{T, n}\|_{L^2}^2\right)+C\left(h^4+h^4\tau^2\right).
\end{align}
For the second term on the right-hand side of 
\eqref{eqn:spacetruncated_error_pf4}, it follows from Lemma \ref{lem:time_error1} that  
\begin{align}\label{eqn:spacetruncated_error_pf6}
	Im\left(V\hat \rho_{h,\tau}^{T, n+\frac{1}{2}}, \hat \theta_{h,\tau}^{T, n+\frac{1}{2}}\right)\leq \|V\|_{L^\infty}\left\|\hat\rho_{h,\tau}^{T, n+\frac{1}{2}}\right\|_{L^2}\left\|\hat \theta_{h,\tau}^{T, n+\frac{1}{2}}\right\|_{L^2}\leq Ch^2\left\|\hat u_{\tau}^{n+\frac{1}{2}}\right\|_{H^2}\left\|\hat \theta_{h,\tau}^{T, n+\frac{1}{2}}\right\|_{L^2}\leq C\left(\|\theta_{h,\tau}^{n+1}\|_{L^2}^2+\|\theta_{h,\tau}^{n}\|_{L^2}^2\right)+Ch^4.
\end{align}
In addition, by virtue of \eqref{eqn:ritz_error} and Lemma \ref{lem:time_error2}, we get 
\begin{align}\label{eqn:spacetruncated_error_pf7}
	&-\Omega Im\left(L_z\hat \rho_{h,\tau}^{T, n+\frac{1}{2}}, \hat \theta_{h,\tau}^{T, n+\frac{1}{2}}\right)
	=-\Omega Im\left(L_z\left[\hat u_{\tau}^{n+\frac{1}{2}}-R_h\hat u_{\tau}^{n+\frac{1}{2}}\right], \hat \theta_{h,\tau}^{T, n+\frac{1}{2}}\right)\nn\\
	&~~~~~~~
	=\Omega Im\left(L_z\left[\hat e_\tau^{n+\frac{1}{2}}-R_h\hat e_\tau^{n+\frac{1}{2}}\right], \hat \theta_{h,\tau}^{T, n+\frac{1}{2}}\right)
	-\Omega Im\left(L_z\left[\hat u^{n+\frac{1}{2}}-R_h\hat u^{n+\frac{1}{2}}\right], \hat \theta_{h,\tau}^{T, n+\frac{1}{2}}\right)\nn\\
	&~~~~~~~\leq |\Omega|C_{R_h}h\left\|L_z\hat e_\tau^{n+\frac{1}{2}}\right\|_{H^1}\left\|\hat \theta_{h,\tau}^{T, n+\frac{1}{2}}\right\|_{L^2}
	+|\Omega|C_{R_h}h^2\left\|L_z\hat u^{n+\frac{1}{2}}\right\|_{H^2}\left\|\hat \theta_{h,\tau}^{T, n+\frac{1}{2}}\right\|_{L^2}\nn\\
	&~~~~~~~\leq C\left(h\tau^2+h^2\right)\left\|\hat \theta_{h,\tau}^{T, n+\frac{1}{2}}\right\|_{L^2}\leq C\left(\|\theta_{h,\tau}^{T,n+1}\|_{L^2}^2+\|\theta_{h,\tau}^{T,n}\|_{L^2}^2\right)+C\left(h^4+h^2\tau^4\right).
\end{align}
Moreover, it is obvious that 
\begin{align}\label{eqn:spacetruncated_error_pf8}
-\Omega Im\left(L_z\hat \theta_{h,\tau}^{T, n+\frac{1}{2}}, \hat \theta_{h,\tau}^{T, n+\frac{1}{2}}\right)
=0.
\end{align}
Thanks to the definition of the truncated function $\mu_A$, we have 
\begin{align}\label{eqn:spacetruncated_error_pf9}
	\mathcal N_{h,\tau}^{T, n+\frac{1}{2}}
	&=\beta\left[\frac{|u_\tau^{n}|^2+|u_\tau^{n+1}|^2}{2}	\hat u_\tau^{n+\frac{1}{2}}-\frac{\mu_A(|u_{h, \tau}^{T, n}|^2)+\mu_A(|u_{h, \tau}^{T, n+1}|^2)}{2}	\hat u_{h, \tau}^{T, n+\frac{1}{2}}\right]\nn\\
	&=\beta\left[\frac{\mu_A(|u_{\tau}^{n}|^2)+\mu_A(|u_{\tau}^{n+1}|^2)}{2}	\hat u_{\tau}^{n+\frac{1}{2}}-\frac{\mu_A(|u_{h, \tau}^{T, n}|^2)+\mu_A(|u_{h, \tau}^{T, n+1}|^2)}{2}	\hat u_{h, \tau}^{T, n+\frac{1}{2}}\right]\nn\\
	&=\beta\left[\frac{\mu_A(|u_{\tau}^{n}|^2)+\mu_A(|u_{\tau}^{n+1}|^2)}{2}	-\frac{\mu_A(|u_{h, \tau}^{T, n}|^2)
	+\mu_A(|u_{h, \tau}^{T, n+1}|^2)}{2}\right]\hat u_{\tau}^{n+\frac{1}{2}}+\beta\frac{\mu_A(|u_{h, \tau}^{T, n}|^2)+\mu_A(|u_{h, \tau}^{T, n+1}|^2)}{2}	\hat e_{h, \tau}^{T, n+\frac{1}{2}}.
	\end{align}
	Then, following the properties of the truncated function $\mu_A$, it is easy to obtain that 
	\begin{align}\label{eqn:spacetruncated_error_pf10}
		\|\mathcal N_{h,\tau}^{T, n+\frac{1}{2}}\|_{L^2}\leq C_{\mathcal N}\left(\|e_{h, \tau}^{T,n}\|_{L^2}+\|e_{h, \tau}^{T,n+1}\|_{L^2}\right).
	\end{align}
Therefore, for the last term on the right-hand side of \eqref{eqn:spacetruncated_error_pf4}, it follows from \eqref{eqn:spacetruncated_error_pf10} and \eqref{eqn:ritz_error} that 
\begin{align}\label{eqn:spacetruncated_error_pf11}
Im\left(\mathcal N_{h,\tau}^{T, n+\frac{1}{2}}, \hat \theta_{h,\tau}^{T, n+\frac{1}{2}}\right)\leq C\left(\|\theta_{h,\tau}^{T,n+1}\|_{L^2}^2+\|\theta_{h,\tau}^{T,n}\|_{L^2}^2\right)+Ch^4.
\end{align}

Now, using \eqref{eqn:spacetruncated_error_pf5}, \eqref{eqn:spacetruncated_error_pf6}, \eqref{eqn:spacetruncated_error_pf7}, 
\eqref{eqn:spacetruncated_error_pf8} and \eqref{eqn:spacetruncated_error_pf11} in 
\eqref{eqn:spacetruncated_error_pf4} obtains that 
\begin{align}\label{eqn:spacetruncated_error_pf12}
	\frac{\|\theta_{h,\tau}^{T,n+1}\|_{L^2}^2-\|\theta_{h,\tau}^{T,n}\|_{L^2}^2}{2\tau}
	\leq C\left(\|\theta_{h,\tau}^{T,n+1}\|_{L^2}^2+\|\theta_{h,\tau}^{T,n}\|_{L^2}^2\right)+C\left(h^4+h^2\tau^4+h^4\tau^2\right).
\end{align}
By using the discrete Gronwall's inequality in \eqref{eqn:spacetruncated_error_pf12}, there exists a positive constant $\tau_4$, such that when $\tau\leq \tau_4$, it holds  
\begin{align}\label{eqn:spacetruncated_error_pf13}
\|\theta_{h,\tau}^{T,n+1}\|_{L^2}\leq C_{E,h}\left(h^2+h\tau^2+h^2\tau\right).
\end{align}
Setting $\tau_0^{**}=\min\{\tau_0^*, \tau_4\}$, we have completed the proof of this lemma.

\end{proof}

\subsection{Proof of $L^2$-norm convergence}
Thanks to the inverse inequality, and using the results given in Lemma \ref{lem:time_error1} and Lemma \ref{lem:spacetruncated_error}, it holds
\begin{align}\label{eqn:spacetruncated_boundedness}
	\left\|u_{h, \tau}^{T, n}\right\|_{L^\infty}& \leq 
	\|u^n\|_{L^\infty}+
	\|u^n-u_{\tau}^{n}\|_{H^2}+
	\left\|u_{\tau}^{n}-R_hu_{\tau}^{n}\right\|_{H^2} 
	+ \left\|R_hu_{\tau}^{n}-u_{h, \tau}^{T, n}\right\|_{L^\infty}\nn\\
	&\leq \|u^n\|_{L^\infty}+C_{E,\tau}\tau^2+C_{E,h}h^2\|u_{\tau}^{n}\|_{H^2}+C_{E, h}h^{-1}(h^2+h\tau^2+h^2\tau)\leq K_0,
\end{align} 
provided $h\leq h_0$ and $\tau\leq \tau_0^{**}$ with positive constants $h_0$ and $\tau_0^{**}$. The $L^\infty$-norm boundedness of $u_{h, \tau}^{T, n}$ can guarantee $\mu_A(|u_{h, \tau}^{T, n}|^2)=|u_{h, \tau}^{T, n}|^2$, which further implies the equivalence of the schemes \eqref{eqn:fullydiscrete} and 
\eqref{eqn:conformingfullydiscrete_truncation}. Hence, $u_{h, \tau}^{T, n}=u_{h, \tau}^{n}$ holds. Therefore, from Lemma \ref{lem:spacetruncated_error}, we get
\begin{align}\label{eqn:L2pro_space_error}
	\sup_{0\leq n\leq N}\left\|R_hu_{\tau}^{n}-u_{h, \tau}^{n}\right\|	_{L^2}\leq C_{E, h}(h^2+h\tau^2+h^2\tau). 
	\end{align}
Then, using \eqref{eqn:ritz_error} and \eqref{eqn:L2pro_space_error}, we obtain 
\begin{align}\label{eqn:L2space_error}
\left\|u^{n}-u_{h, \tau}^{n}\right\|	_{L^2}
&\leq \left\|u^{n}-u_\tau^{n}\right\|_{L^2}+\left\|u_\tau^{n}-R_hu_\tau^{n}\right\|_{L^2}+\left\|R_hu_{\tau}^{n}-u_{h, \tau}^{n}\right\|_{L^2}\nonumber\\
&\leq C_{E, \tau}\tau^2+C_{R_h}C_{u,\tau}h^2+ C_{E, h}(h^2+h\tau^2+h^2\tau)
\leq C_{E}(h^2+\tau^2). 
	\end{align}
We have completed the proof of the optimal convergence in the sense of $L^2$-norm. 

\begin{rem}
It is noticed that the $L^2$-norm convergence rate is optimal without any time-space ratio restriction, and also without any requirements on the parameters. Furthermore, observed from the proof, the results provided in Lemma \ref{lem:time_error1} are sufficient to yield the convergence result \eqref{eqn:L2space_error}. Indeed, under the results of Lemma \ref{lem:time_error1}, \eqref{eqn:spacetruncated_error_pf5} can be modified by
	\begin{align}
		Re\left(D_\tau \rho_{h,\tau}^{T, n+\frac{1}{2}}, \hat \theta_{h,\tau}^{T, n+\frac{1}{2}}\right)\leq C\left(\|\theta_{h,\tau}^{n+1}\|_{L^2}^2+\|\theta_{h,\tau}^{n}\|_{L^2}^2\right)+Ch^2.
	\end{align}
Thus, \eqref{eqn:spacetruncated_error_pf13} will be updated by 
\begin{align}\label{eqn:spacetruncated_error_pf13_update}
\|\theta_{h,\tau}^{n+1}\|_{L^2}\leq C_{E,h}h.
\end{align}
Then the $L^\infty$-boundedness of $u_{h, \tau}^{T, n}$ can be derived similarly to \eqref{eqn:spacetruncated_boundedness}, but the upper bound will not be $K_0$. Additionally, if we consider the FE space in 3D, the $L^\infty$-boundedness of $u_{h, \tau}^{T, n}$ cannot be obtained by using the inverse inequality. Therefore, it is evident that the improved result presented in Lemma \ref{lem:time_error2} is a preferable choice.
\end{rem}

\subsection{Proof of $H^1$-norm high-order convergence}

In \cite{bao2013optimal}, the authors employed the finite difference method to tackle the model described by equation \eqref{eqn:model}. They conducted an analysis that yielded an error estimate in the $H^1$-norm, demonstrating a convergence rate of $O(h^{3/2}+\tau^{3/2})$. In Lemma \ref{lem:spacetruncated_error}, we established an error estimate in the $L^2$-norm with a optimal convergence rate of $O(h^{2})$. Consequently, this allowed us to achieve the optimal $H^1$-norm error estimate by using the inverse inequality, which exhibits a convergence rate of $O(h)$. In what follows, we will conduct a comprehensive analysis to establish high-order convergence in the $H^1$-norm.

Denoting $\omega_h=D_\tau \theta_{h,\tau}^{T, n+\frac{1}{2}}$ in \eqref{eqn:spacetruncated_error_pf3}, and selecting the real parts of the resulting, we get 
\begin{align}\label{eqn:super_pf1}
	&\frac{\|\nabla\theta_{h,\tau}^{T,n+1}\|_{L^2}^2-\|\nabla\theta_{h,\tau}^{T,n}\|_{L^2}^2}{4\tau}+\frac{\left(V\theta_{h,\tau}^{T,n+1}, \theta_{h,\tau}^{T,n+1}\right)-\left(V\theta_{h,\tau}^{T,n}, \theta_{h,\tau}^{T,n}\right)}{2\tau}
	=Im\left(D_\tau\rho_{h,\tau}^{T,n+\frac12}, D_\tau\theta_{h,\tau}^{T,n+\frac12}\right)
	-Re\left(V\hat{\rho}_{h,\tau}^{T,n+\frac12}, D_\tau \theta_{h,\tau}^{T, n+\frac{1}{2}}\right)\nn\\
	&~~~~~~~~~~~~~~~+\Omega Re\left(L_z\hat{\rho}_{h,\tau}^{T,n+\frac12}, D_\tau \theta_{h,\tau}^{T, n+\frac{1}{2}}\right)+\Omega Re\left(L_z\hat{\theta}_{h,\tau}^{T,n+\frac12}, D_\tau \theta_{h,\tau}^{T, n+\frac{1}{2}}\right)-Re\left(\mathcal N_{h,\tau}^{T, n+\frac{1}{2}}, D_\tau \theta_{h,\tau}^{T, n+\frac{1}{2}}\right).
\end{align}
Similar as \eqref{eqn:spacetruncated_error_pf5},  \eqref{eqn:spacetruncated_error_pf6} and \eqref{eqn:spacetruncated_error_pf7}, we have 
\begin{align}\label{eqn:super_pf2}
	&Im\left(D_\tau\rho_{h,\tau}^{T,n+\frac12}, D_\tau\theta_{h,\tau}^{T,n+\frac12}\right)
	-Re\left(V\hat{\rho}_{h,\tau}^{T,n+\frac12}, D_\tau \theta_{h,\tau}^{T, n+\frac{1}{2}}\right)+\Omega Re\left(L_z\hat{\rho}_{h,\tau}^{T,n+\frac12}, D_\tau \theta_{h,\tau}^{T, n+\frac{1}{2}}\right)\nonumber\\
	&\leq C\left\|D_\tau\theta_{h,\tau}^{T,n+\frac12}\right\|_{L^2}^2+C\left(h^4+h^4\tau^2+h^2\tau^4\right).
\end{align}
Meanwhile, similar as \eqref{eqn:spacetruncated_error_pf11}, it follows that 
\begin{align}\label{eqn:super_pf3}
	Re\left(\mathcal N_{h,\tau}^{T, n+\frac{1}{2}}, D_\tau \theta_{h,\tau}^{T, n+\frac{1}{2}}\right)
	\leq C\left\|D_\tau\theta_{h,\tau}^{T,n+\frac12}\right\|_{L^2}^2+Ch^4.
\end{align}
Substituting \eqref{eqn:super_pf2} and \eqref{eqn:super_pf3} into \eqref{eqn:super_pf1} obtains 
\begin{align}\label{eqn:super_pf4}
	&\frac{\|\nabla\theta_{h,\tau}^{T,n+1}\|_{L^2}^2-\|\nabla\theta_{h,\tau}^{T,n}\|_{L^2}^2}{4\tau}+\frac{\left(V\theta_{h,\tau}^{T,n+1}, \theta_{h,\tau}^{T,n+1}\right)-\left(V\theta_{h,\tau}^{T,n}, \theta_{h,\tau}^{T,n}\right)}{2\tau}\nn\\
	&\leq C\left(\|\nabla\theta_{h,\tau}^{T,n+1}\|_{L^2}^2+\|\nabla\theta_{h,\tau}^{T,n}\|_{L^2}^2+\|D_\tau\theta_{h,\tau}^{T,n+\frac12}\|_{L^2}^2\right)+C\left(h^4+h^4\tau^2+h^2\tau^4\right).
\end{align}

On the other hand, we take the difference of \eqref{eqn:spacetruncated_error_pf3} between two consecutive steps to arrive at 
\begin{align}\label{eqn:super_pf5}
	&i\left(\frac{\theta_{h,\tau}^{T,n+1}-2\theta_{h,\tau}^{T,n}+\theta_{h,\tau}^{T,n-1}}{\tau^2}, \omega_h\right)
	=-i\left(\frac{\rho_{h,\tau}^{T,n+1}-2\rho_{h,\tau}^{T,n}+\rho_{h,\tau}^{T,n-1}}{\tau^2}, \omega_h\right)
	+\left(\frac{\nabla\theta_{h,\tau}^{T,n+1}-\nabla\theta_{h,\tau}^{T,n-1}}{4\tau}, \nabla\omega_h\right)+\left(\frac{V\rho_{h,\tau}^{T,n+1}-V\rho_{h,\tau}^{T,n-1}}{2\tau}, \omega_h\right)\nn\\
	&+\left(\frac{V\theta_{h,\tau}^{T,n+1}-V\theta_{h,\tau}^{T,n-1}}{2\tau}, \omega_h\right)
	-\Omega\left(\frac{L_z\rho_{h,\tau}^{T,n+1}-L_z\rho_{h,\tau}^{T,n-1}}{2\tau}, \omega_h\right)
	-\Omega\left(\frac{L_z\theta_{h,\tau}^{T,n+1}-L_z\theta_{h,\tau}^{T,n-1}}{2\tau}, \omega_h\right)
	+\left(\frac{\mathcal N_{h,\tau}^{T, n+\frac{1}{2}}-\mathcal N_{h,\tau}^{T, n-\frac{1}{2}}}{\tau}, \omega_h\right).
\end{align}
For simplicity, we denote $\theta_{h,\tau}^{T,n+1}-\theta_{h,\tau}^{T,n}=\tau H_{h,\tau}^{T,n+1}$. Then we rewrite \eqref{eqn:super_pf5} as 
\begin{align}\label{eqn:super_pf6}
	&i\left(D_\tau H_{h,\tau}^{T,n+\frac12}, \omega_h\right)
	=-i\left(\frac{\rho_{h,\tau}^{T,n+1}-2\rho_{h,\tau}^{T,n}+\rho_{h,\tau}^{T,n-1}}{\tau^2}, \omega_h\right)
	+\frac12\left(\nabla\hat{H}_{h,\tau}^{T,n+\frac12}, \nabla\omega_h\right)+\left(\frac{V\rho_{h,\tau}^{T,n+1}-V\rho_{h,\tau}^{T,n-1}}{2\tau}, \omega_h\right)+\left(V\hat{H}_{h,\tau}^{T,n+\frac12}, \omega_h\right)\nn\\
	&~~~~~~~~~~~~
	-\Omega\left(\frac{L_z\rho_{h,\tau}^{T,n+1}-L_z\rho_{h,\tau}^{T,n-1}}{2\tau}, \omega_h\right)
	-\Omega\left(L_z\hat{H}_{h,\tau}^{T,n+\frac12}, \omega_h\right)
	+\left(\frac{\mathcal N_{h,\tau}^{T, n+\frac{1}{2}}-\mathcal N_{h,\tau}^{T, n-\frac{1}{2}}}{\tau}, \omega_h\right).
\end{align}
Let $\omega_h=\hat{H}_{h,\tau}^{T,n+\frac12}$ in \eqref{eqn:super_pf6}, and take the imaginary part of the resulting, then we get 
\begin{align}\label{eqn:super_pf7}
	\frac{\|H_{h,\tau}^{T,n+1}\|_{L^2}^2-\|H_{h,\tau}^{T,n}\|_{L^2}^2}{2\tau}=&-Re\left(\frac{\rho_{h,\tau}^{T,n+1}-2\rho_{h,\tau}^{T,n}+\rho_{h,\tau}^{T,n-1}}{\tau^2}, \hat{H}_{h,\tau}^{T,n+\frac12}\right)+Im\left(\frac{V\rho_{h,\tau}^{T,n+1}-V\rho_{h,\tau}^{T,n-1}}{2\tau}, \hat{H}_{h,\tau}^{T,n+\frac12}\right)\nn\\
	&-\Omega\left(\frac{L_z\rho_{h,\tau}^{T,n+1}-L_z\rho_{h,\tau}^{T,n-1}}{2\tau}, \hat{H}_{h,\tau}^{T,n+\frac12}\right)+\left(\frac{\mathcal N_{h,\tau}^{T, n+\frac{1}{2}}-\mathcal N_{h,\tau}^{T, n-\frac{1}{2}}}{\tau}, \hat{H}_{h,\tau}^{T,n+\frac12}\right).
\end{align}
By using Lemma \ref{lem:time_error2}, we can easily obtain 
\begin{align}\label{eqn:super_pf8}
	Re\left(\frac{\rho_{h,\tau}^{T,n+1}-2\rho_{h,\tau}^{T,n}+\rho_{h,\tau}^{T,n-1}}{\tau^2}, \hat{H}_{h,\tau}^{T,n+\frac12}\right)+Im\left(\frac{V\rho_{h,\tau}^{T,n+1}-V\rho_{h,\tau}^{T,n-1}}{2\tau}, \hat{H}_{h,\tau}^{T,n+\frac12}\right)
	\leq C\left(\|H_{h,\tau}^{T,n+1}\|_{L^2}^2+\|H_{h,\tau}^{T,n}\|_{L^2}^2\right)+C\left(h^4+h^4\tau^2\right). 
\end{align}
 Thanks to  
\begin{align}\label{eqn:super_pf9}
	\frac{L_z\rho_{h,\tau}^{T,n+1}-L_z\rho_{h,\tau}^{T,n-1}}{2\tau}
	&=\left[L_z\left(\frac{u^{n+1}-u^{n-1}}{2\tau}\right)
	-L_zR_h\left(\frac{u^{n+1}-u^{n-1}}{2\tau}\right)\right]
	-
	\left[L_z\left(\frac{e_\tau^{n+1}-e_\tau^{n-1}}{2\tau}\right)
	-L_zR_h\left(\frac{e_\tau^{n+1}-e_\tau^{n-1}}{2\tau}\right)\right],\end{align}
and by using \eqref{eqn:ritz_error} and Lemma \ref{lem:time_error2}, we can get 
\begin{align}\label{eqn:super_pf10}
	\left\|\frac{L_z\rho_{h,\tau}^{T,n+1}-L_z\rho_{h,\tau}^{T,n-1}}{2\tau}\right\|_{L^2}\leq C_{R_h}
h^2+C_{R_h}h\left\|\frac{e_\tau^{n+1}-e_\tau^{n-1}}{2\tau}\right\|_{H^2}\leq C\left(h^2+h\tau\right),
\end{align}
which further implies 
\begin{align}\label{eqn:super_pf11}
	-\Omega\left(\frac{L_z\rho_{h,\tau}^{T,n+1}-L_z\rho_{h,\tau}^{T,n-1}}{2\tau}, \hat{H}_{h,\tau}^{T,n+\frac12}\right)\leq C\left(\|H_{h,\tau}^{T,n+1}\|_{L^2}^2+\|H_{h,\tau}^{T,n}\|_{L^2}^2\right)+C\left(h^4+h^2\tau^2\right). 
\end{align}
Moreover, using similar methods as \eqref{eqn:time_error2_pf12}, we have 
\begin{align}\label{eqn:super_pf12}
	\left(\frac{\mathcal N_{h,\tau}^{T, n+\frac{1}{2}}-\mathcal N_{h,\tau}^{T, n-\frac{1}{2}}}{\tau}, \hat{H}_{h,\tau}^{T,n+\frac12}\right)
	\leq C\left(\|H_{h,\tau}^{T,n+1}\|_{L^2}^2+\|H_{h,\tau}^{T,n}\|_{L^2}^2\right)+C\left(h^4+h^2\tau^4+h^4\tau^2\right). 
\end{align}
Hence, substituting \eqref{eqn:super_pf8}, \eqref{eqn:super_pf11} and \eqref{eqn:super_pf12} into \eqref{eqn:super_pf7},  we have 
\begin{align}\label{eqn:super_pf13}
\frac{\|H_{h,\tau}^{T,n+1}\|_{L^2}^2-\|H_{h,\tau}^{T,n}\|_{L^2}^2}{2\tau}\leq C\left(\|H_{h,\tau}^{T,n+1}\|_{L^2}^2+\|H_{h,\tau}^{T,n}\|_{L^2}^2\right)+C\left(h^4+\tau^4\right).
\end{align}
By using discrete Gronwall's inequality, \eqref{eqn:super_pf13} further implies that there exists a positive constant $\tau_5$, such that when $\tau\leq \tau_5$, there holds 
\begin{align}\label{eqn:super_pf14}
	\|H_{h,\tau}^{T,n+1}\|_{L^2} = \|D_\tau\theta_{h,\tau}^{T,n+\frac12}\|_{L^2} \leq C(h^2+\tau^2). 
\end{align}
Therefore, from \eqref{eqn:super_pf14} and \eqref{eqn:super_pf4}, we have 
\begin{align}\label{eqn:super_pf15}
	&\frac{\|\nabla\theta_{h,\tau}^{T,n+1}\|_{L^2}^2-\|\nabla\theta_{h,\tau}^{T,n}\|_{L^2}^2}{4\tau}+\frac{\left(V\theta_{h,\tau}^{T,n+1}, \theta_{h,\tau}^{T,n+1}\right)-\left(V\theta_{h,\tau}^{T,n}, \theta_{h,\tau}^{T,n}\right)}{2\tau}\leq C\left(\|\nabla\theta_{h,\tau}^{T,n+1}\|_{L^2}^2+\|\nabla\theta_{h,\tau}^{T,n}\|_{L^2}^2\right)+C\left(h^4+\tau^4\right).
\end{align}
Then by using discrete Gronwall's inequality again, we can get 
\begin{align}\label{eqn:super_pf16}
	\|\nabla\theta_{h,\tau}^{n+1}\|_{L^2}=\|\nabla\theta_{h,\tau}^{T,n+1}\|_{L^2}\leq C_E(h^2+\tau^2),
\end{align}
provided that $\tau\leq \tau_6$ with a positive constant $\tau_6$ independent of $h$ and $\tau$.

Finally, we have 
\begin{align}
	\left\|\nabla(I_hu^{n+1}-u_{\tau,h}^n)\right\|_{L^2}
	\leq \left\|\nabla(I_hu^{n+1}-R_hu^{n+1})\right\|_{L^2}+\left\|\nabla(R_hu^{n+1}-R_hu_{\tau}^{n+1})\right\|_{L^2}+\left\|\nabla(R_hu_\tau^{n+1}-u_{\tau,h}^{n+1})\right\|_{L^2}\leq C_E(h^2+\tau^2),
\end{align}
and based on interpolated postprocessing technique, we can obtain 
\begin{align}
\left\|u^{n+1}-I_{2h}U_h^{n+1}\right\|_{H^1}\leq C_E(h^2+\tau^2),
\end{align}
where $I_{2h}$ is the interpolated postprocessing operator \cite{lin2007finite}.

Setting $\tau_0^{***}=\min\{\tau_0^{**}, \tau_5, \tau_6\}$, we have completed the proof.

\begin{rem}
When in the nonconforming FE space, we do not need to provide a complete error analysis, as most of it is similar to the conforming case. The brief proof is shown in Appendix. 
\end{rem}

\section{Numerical results}\label{sec5}
In this section, we first verify the accuracy of the conforming and nonconforming FEMs shown in Theorem \ref{thm:main}. 
Then, we check the discrete mass and energy conservations of the both schemes. Finally, we show some experiments about the 
dynamics of vortex lattice in rotating BEC. 

\subsection{Accuracy test}

To test the accuracy of our numerical methods in 2D, we add a source term at the right-hand side of the model \eqref{eqn:model} with the exact solution 
$u(x, y, t)=(t+1)^2\sin\pi x\sin\pi y$, and take $U=[0, 1]\times [0, 1]$, $\Omega=0.8$, $\beta=1$, $\gamma_x=1$ and $\gamma_y=2$. For convenience, we select $h_x=h_y=h=\tau$, with $h_x$ and $h_y$ denoting the spatial step sizes in $x$ and $y$ directions, respectively. We numerically verify the optimal error estimates in the $L^2$-norm and $H^1$-norm, as well as the high-order error estimates in $H^1$-norm, using both the conforming and nonconforming FEMs proposed in Definition \ref{Def:fully}. The corresponding results are illustrated in Figure \ref{figure:error}, and they align with the theoretical analysis presented in Theorem \ref{thm:main}.

\begin{figure}[!htp]
\centering
\includegraphics[width=0.45\textwidth]{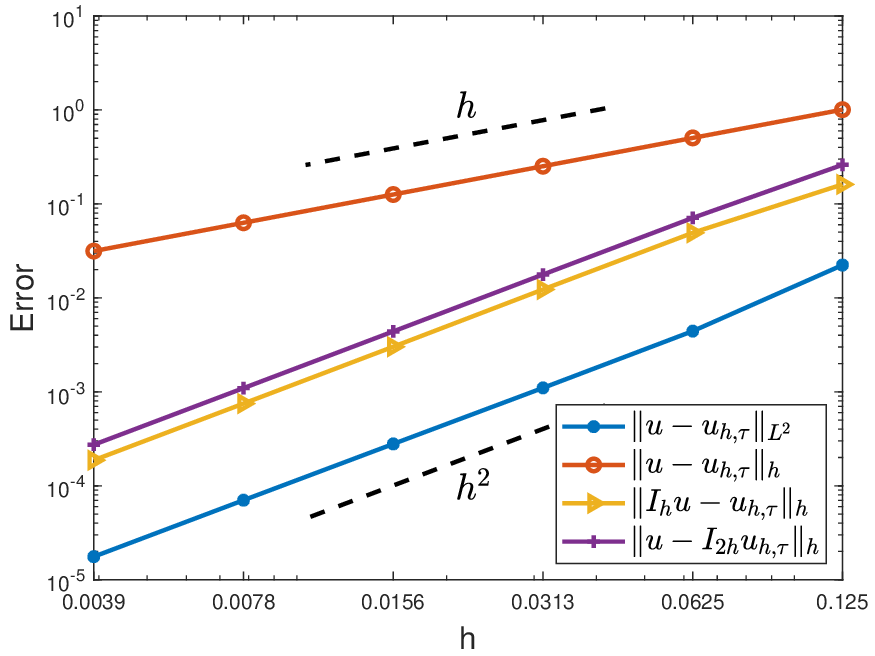}
\includegraphics[width=0.45\textwidth]{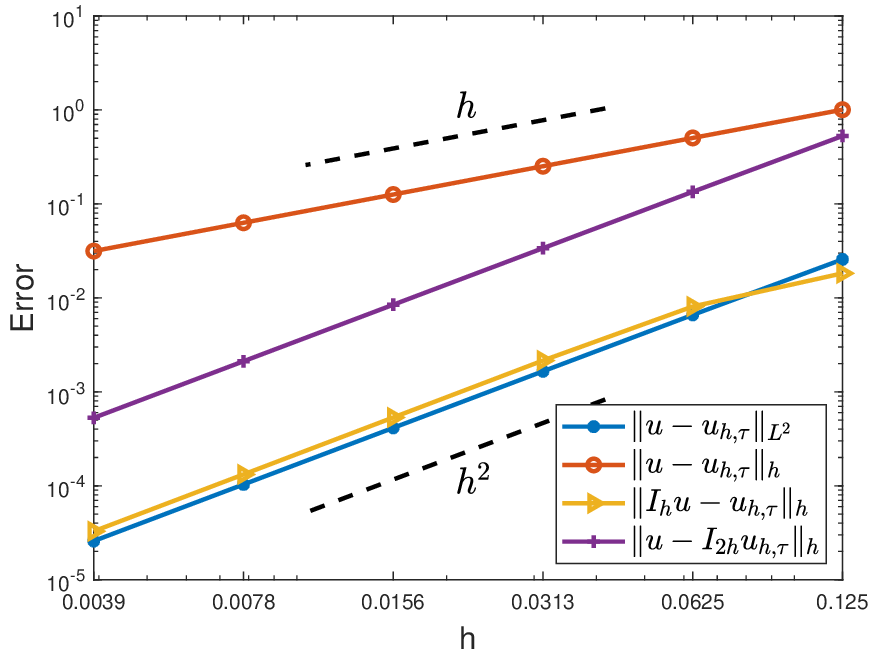}
\caption{The error estimates and convergence rates for the conforming (left) and nonconforming (right) FEMs.}
\label{figure:error}
\end{figure}

\subsection{Structure-preservation test}
In this subsection, we test the structure-preserving properties proposed in Theorem \ref{thm-fullydiscrete-conservation}, including the mass conservation and energy conservation in the discrete senses. We define the relative errors of the discrete mass and discrete energy by 
\begin{align}
	ER(N(u_{h,\tau}^n)) = \frac{N(u_{h,\tau}^n)-N(u_{h,\tau}^0)}{N(u_{h,\tau}^0)},\qquad 
	ER(E_h(u_{h,\tau}^n)) = \frac{E_h(u_{h,\tau}^n)-N(u_{h,\tau}^0)}{N(u_{h,\tau}^0)}.
\end{align}
The numerical results, as showcased in Figures \ref{figure:massQ1}-\ref{figure:energyEQ1}, are consistent with the theoretical conclusions that we have proved.

\begin{figure}[!htp]
\centering
\includegraphics[width=0.45\textwidth]{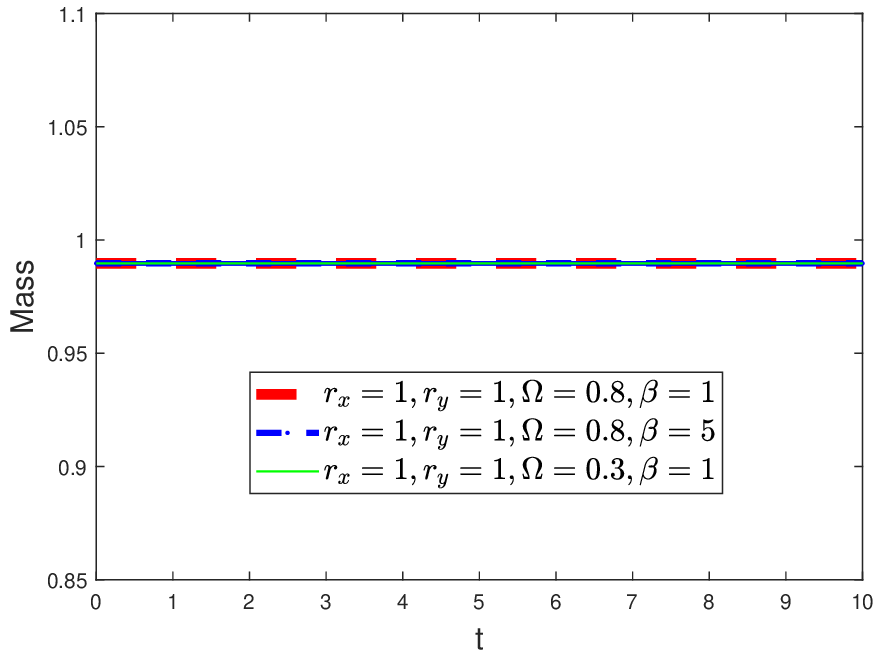}
\includegraphics[width=0.45\textwidth]{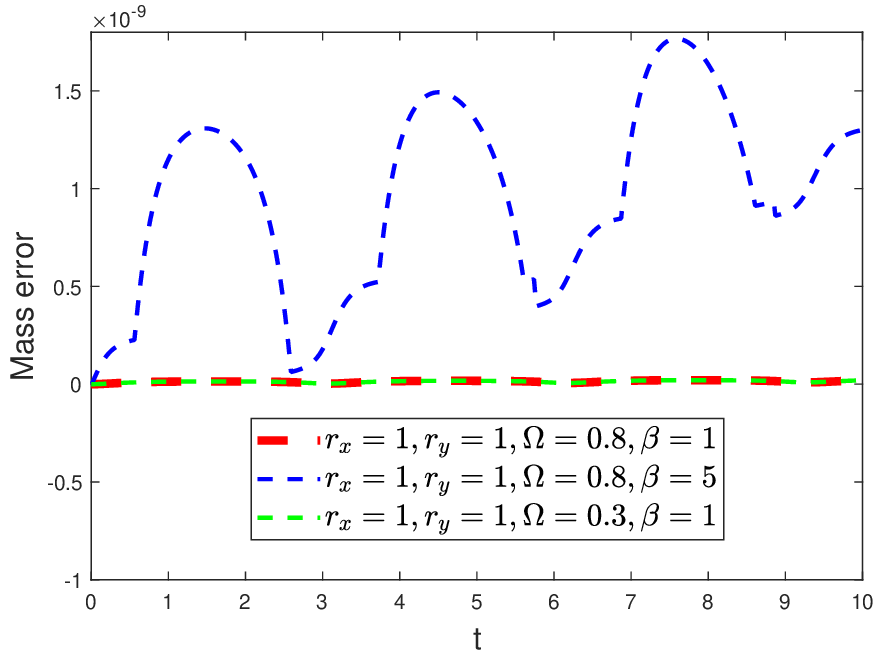}
\caption{The discrete mass and its relative error for the conforming FEM.}
\label{figure:massQ1}
\end{figure}

\begin{figure}[!htp]
\centering
\includegraphics[width=0.45\textwidth]{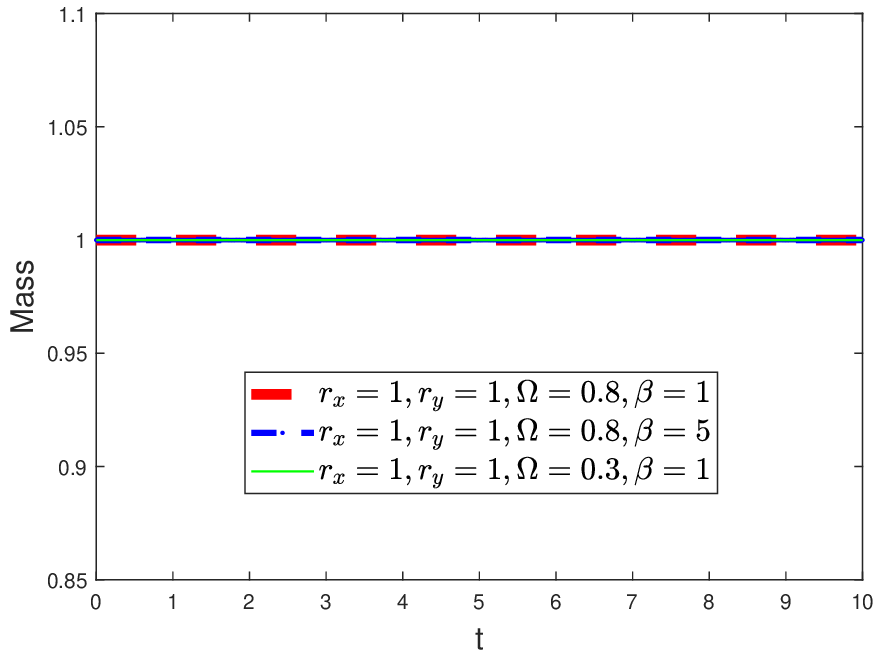}
\includegraphics[width=0.45\textwidth]{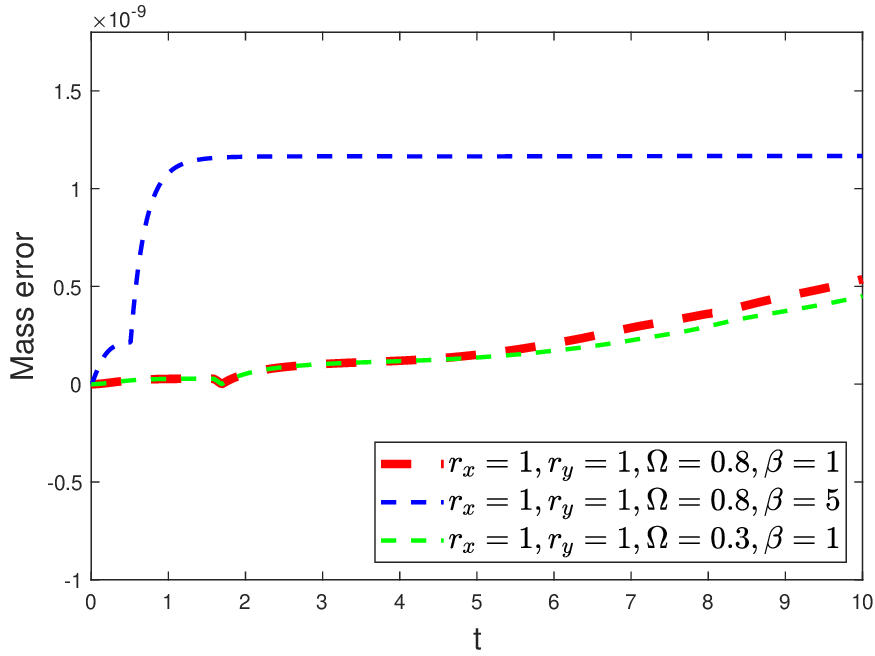}
\caption{The discrete mass and its relative error for the nonconforming FEM.}
\label{figure:massEQ1}
\end{figure}

\begin{figure}[!htp]
\centering
\includegraphics[width=0.45\textwidth]{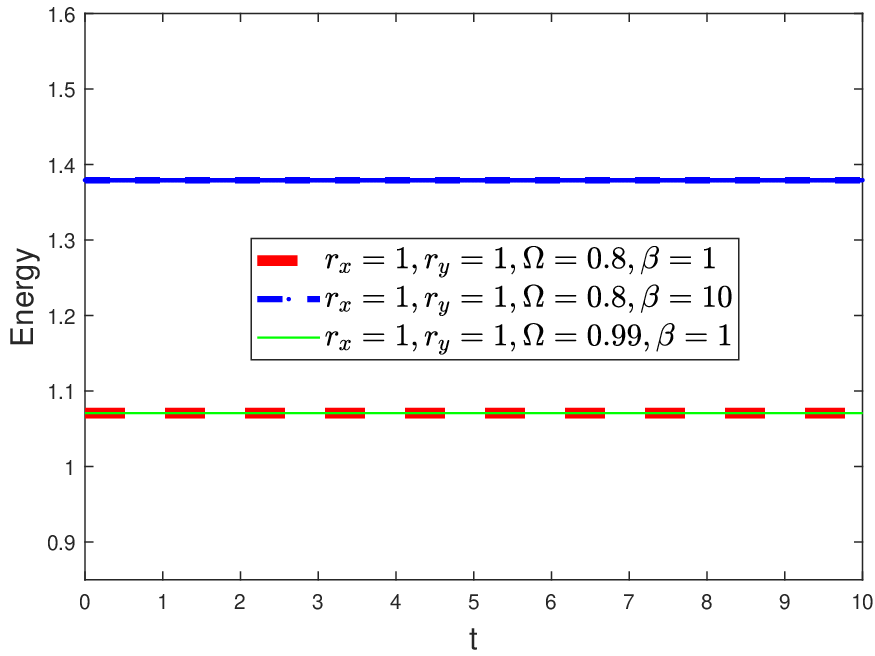}
\includegraphics[width=0.45\textwidth]{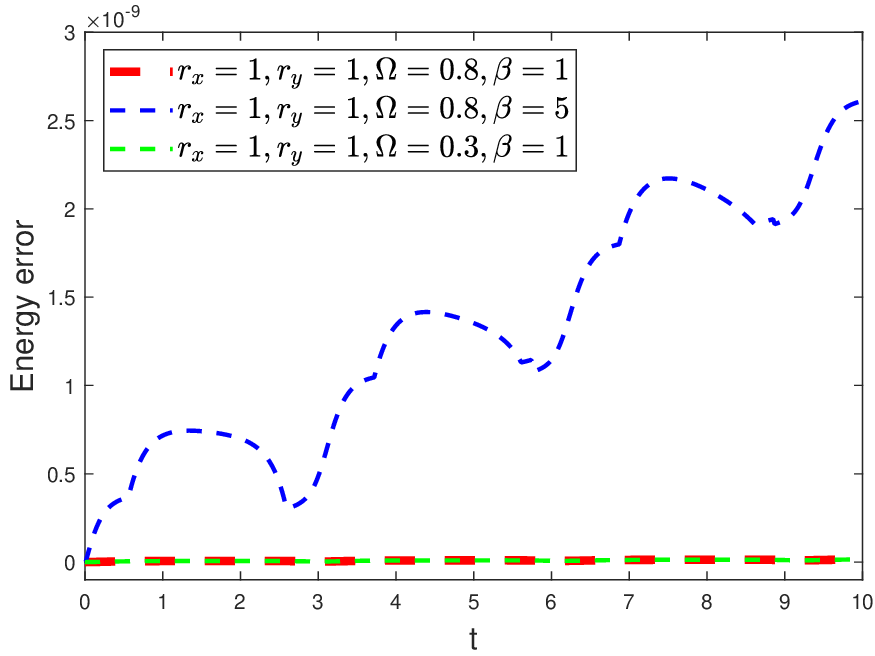}
\caption{The discrete energy and its relative error for the conforming FEM.}
\label{figure:energyQ1}
\end{figure}

\begin{figure}[!htp]
\centering
\includegraphics[width=0.45\textwidth]{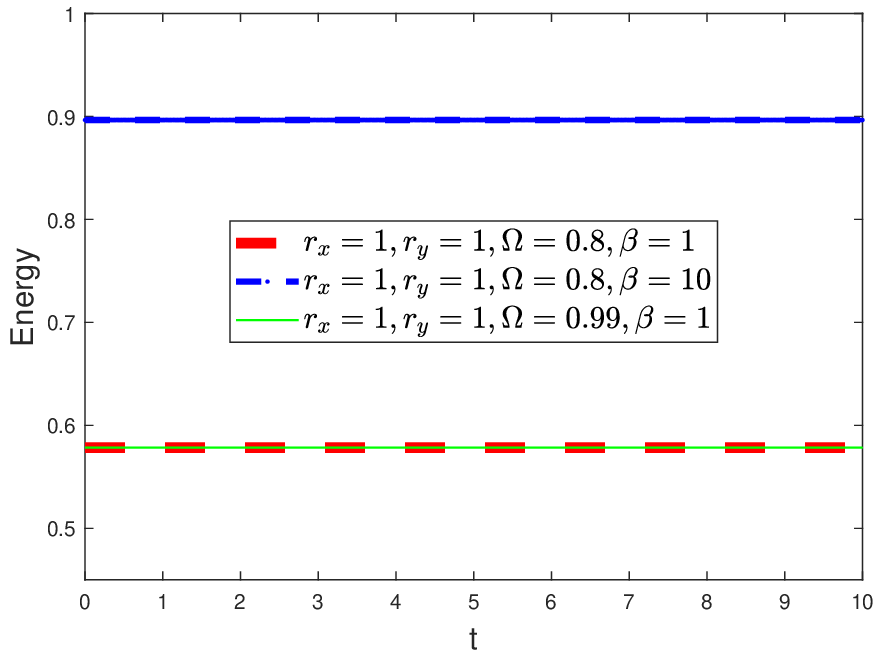}
\includegraphics[width=0.45\textwidth]{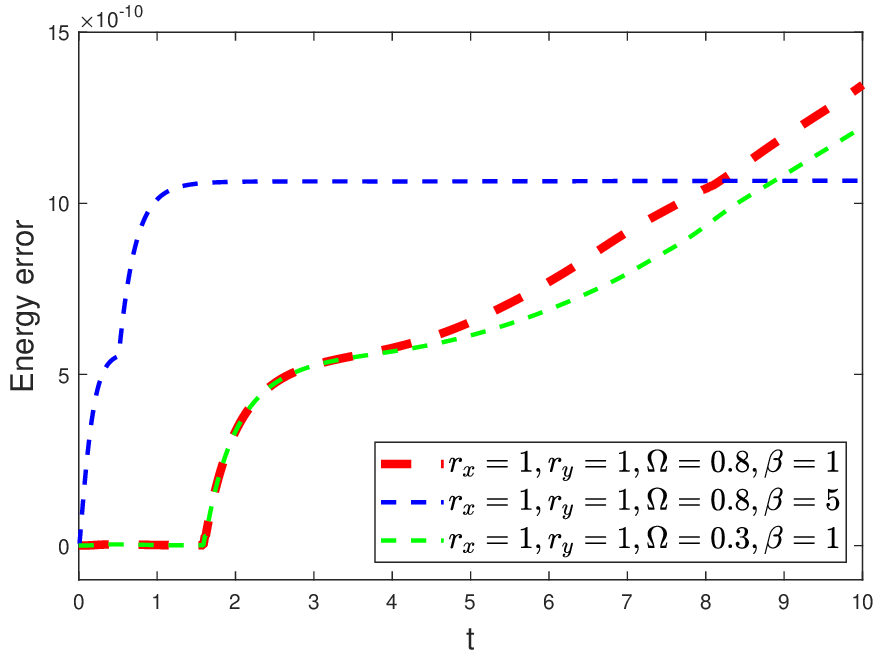}
\caption{The discrete energy and its relative error for the nonconforming FEM.}
\label{figure:energyEQ1}
\end{figure}

\subsection{Dynamics of a vortex lattice in rotating BEC}
We in this subsection numerically study the dynamics of a vortex lattice in rotating BEC influenced by the trap frequencies. We consider the model \eqref{eqn:model} with the parameters $\beta=100$ and $\Omega=0.99$, and the the domain $U=[-16, 16]\times[-16, 16]$. The initial data is selected as the ground state of the system with $\gamma_x=\gamma_y=1$, numerically computed by using the normalized gradient flow proposed in \cite{bao2005ground}. In the subsequent experiments, we will exclusively employ the conforming FEM. 
We mainly do the following tests: 
\begin{itemize}
\item [(I)]	The first test is to check the free expansion of the quantized vortex lattice when the trap is removed, i.e., $\gamma_x=\gamma_y=0$. From Figure \ref{figure:BEC1}, it is noticeable that the vortex lattice undergoes expansion over time in the absence of any trapping influence during the numerical test, but the vortex structure still keeps the rotational symmetry during the expansion. As a comparison, we also present the dynamics of vortex lattices for the case of $\gamma_x=\gamma_y=1$, see Figure \ref{figure:BEC2}.
\begin{figure}[!htp]
\centering
\includegraphics[width=0.22\textwidth]{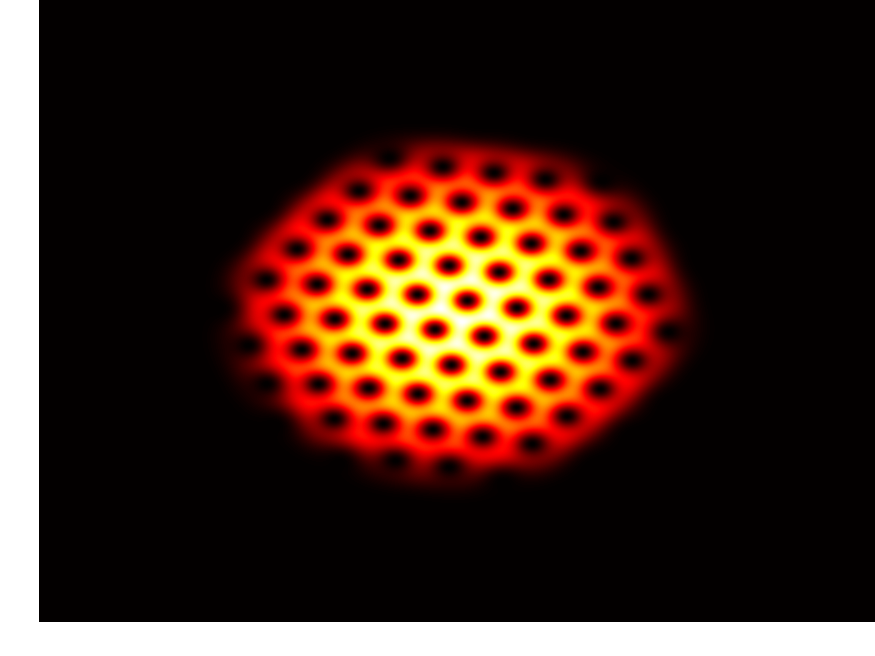}
\includegraphics[width=0.22\textwidth]{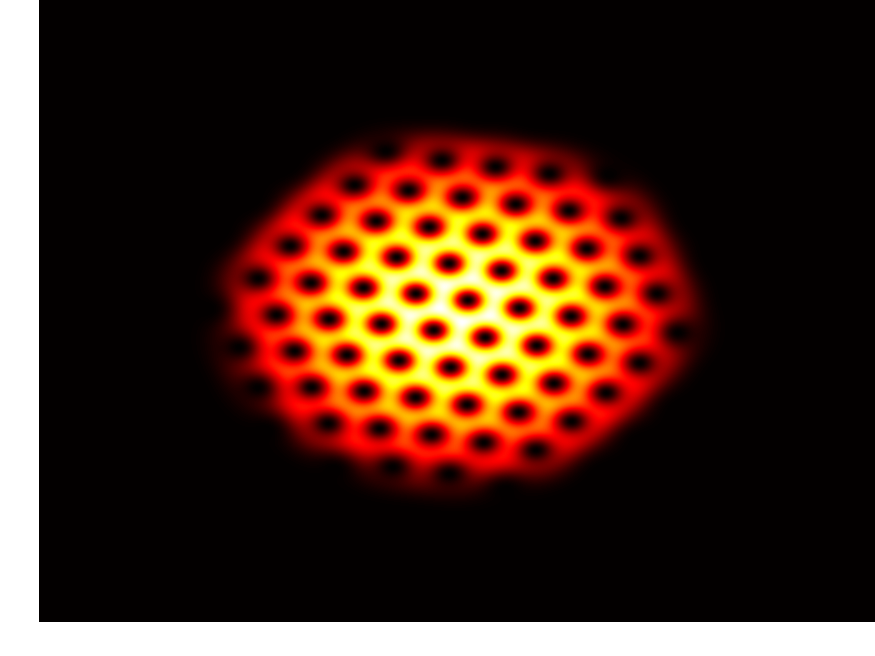}
\includegraphics[width=0.22\textwidth]{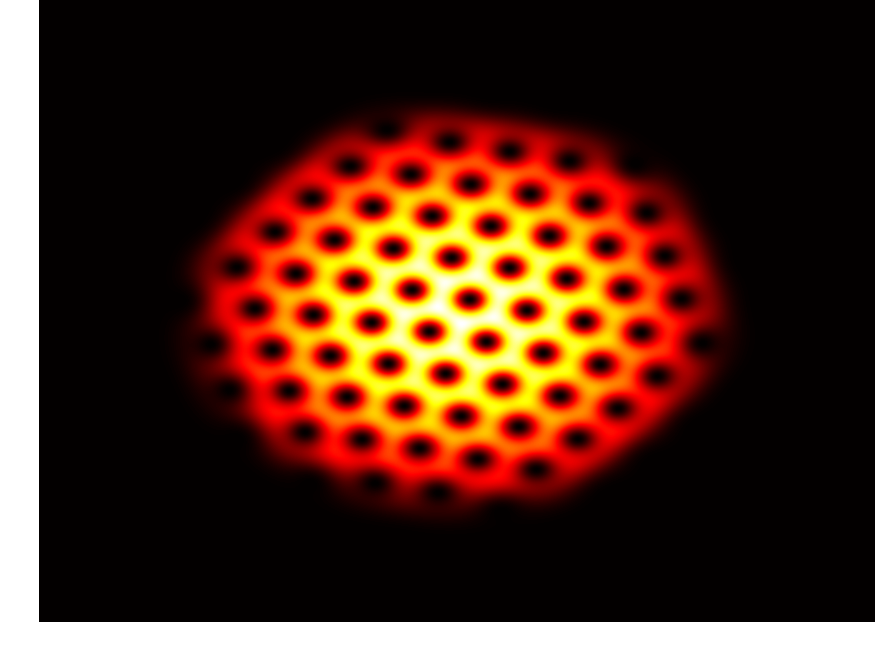}
\includegraphics[width=0.22\textwidth]{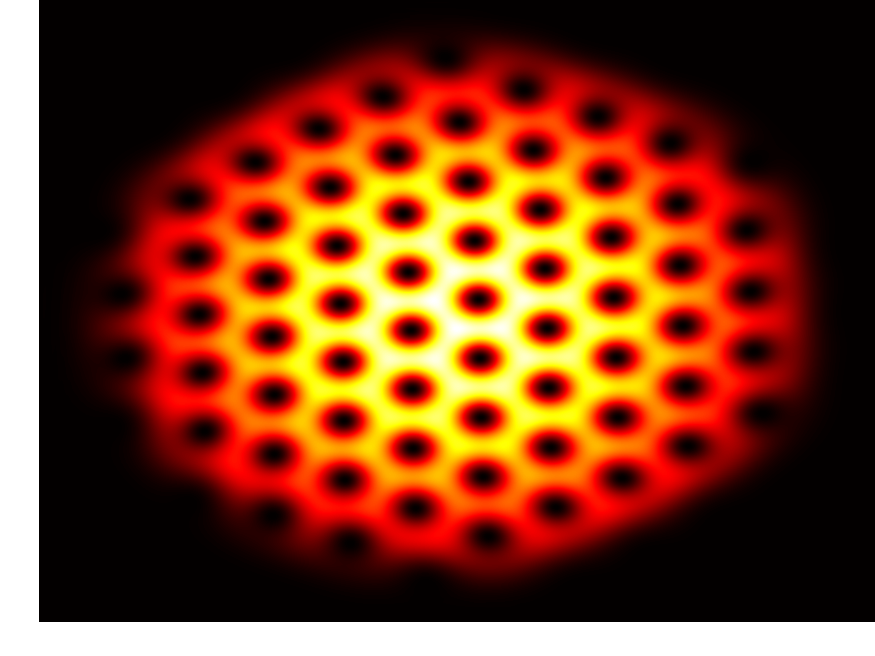}
\caption{Image plots of the density $|u|^2$ at the times $t=0, 0.3, 0.6$ and $1.2$. Here, we select $\gamma_x=\gamma_y=0$, $h=1/16$ and $\tau=0.01$.}
\label{figure:BEC1}
\end{figure}

\begin{figure}[!htp]
\centering
\includegraphics[width=0.22\textwidth]{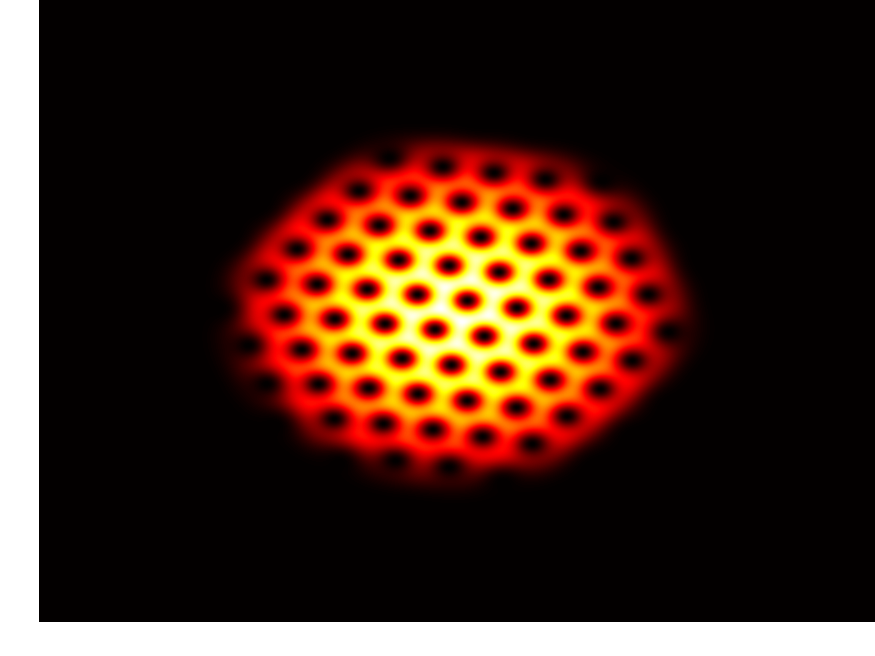}
\includegraphics[width=0.22\textwidth]{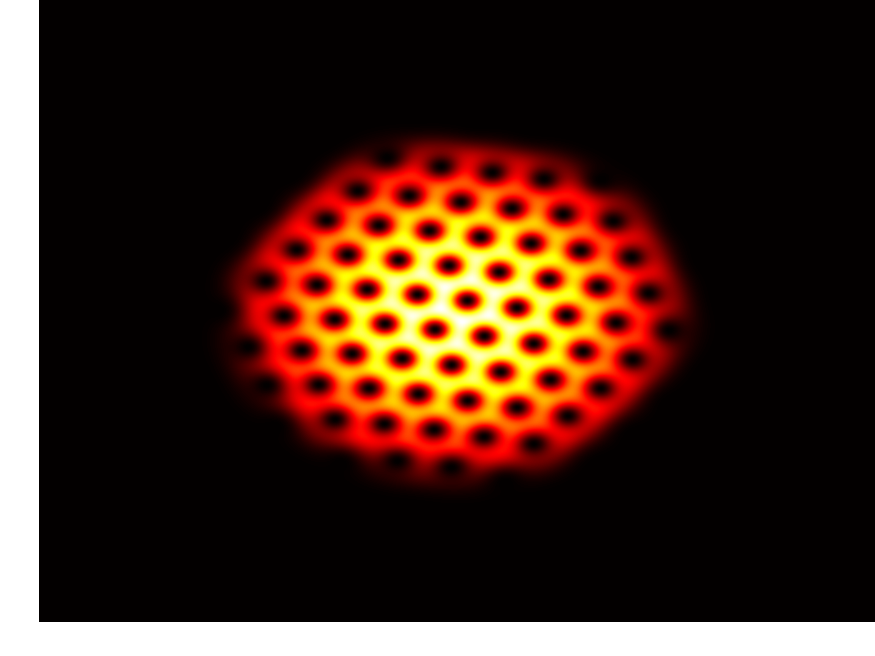}
\includegraphics[width=0.22\textwidth]{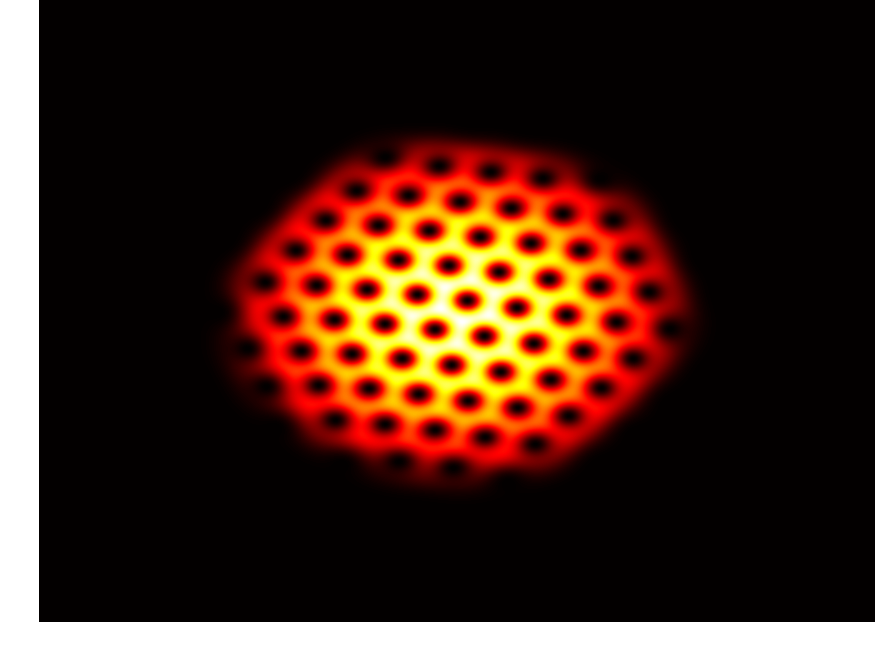}
\includegraphics[width=0.22\textwidth]{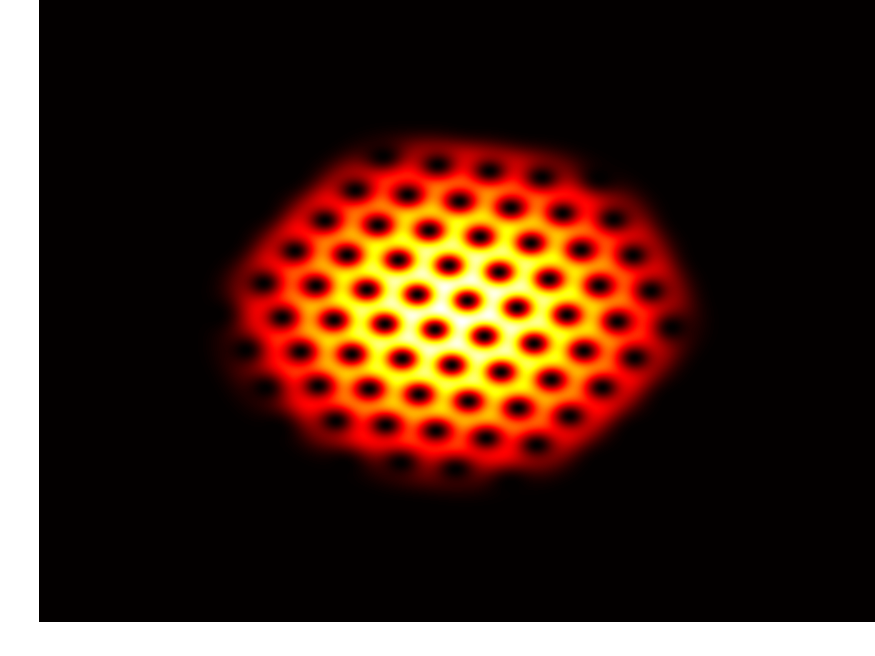}
\caption{Image plots of the density $|u|^2$ at the times $t=0, 0.75, 1.5$ and $3$. Here, we select $\gamma_x=\gamma_y=1$, $h=1/16$ and $\tau=0.01$.}
\label{figure:BEC2}
\end{figure}

\item [(II)]	Another test focuses on investigating the dynamics of the quantized vortex lattice as the trap frequencies undergo variations. We change the frequencies in $y$-direction only, in $x$-direction only, or in both $x$- and $y$-directions, see the image plots of the density $|u|^2$ with different cases in Figure \ref{figure:BEC3}. 

 \begin{figure}[!htp]
\centering
\includegraphics[width=0.22\textwidth]{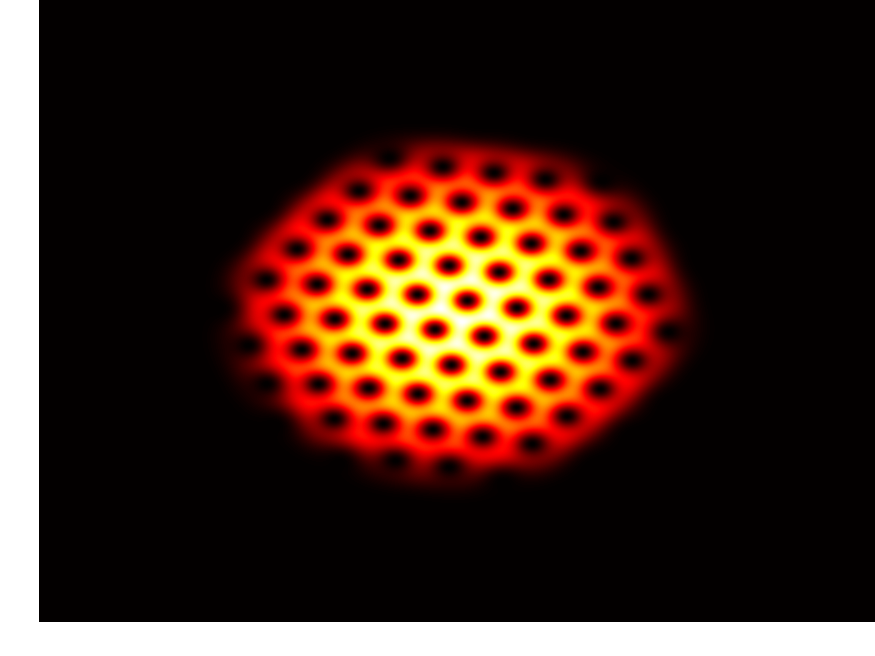}
\includegraphics[width=0.22\textwidth]{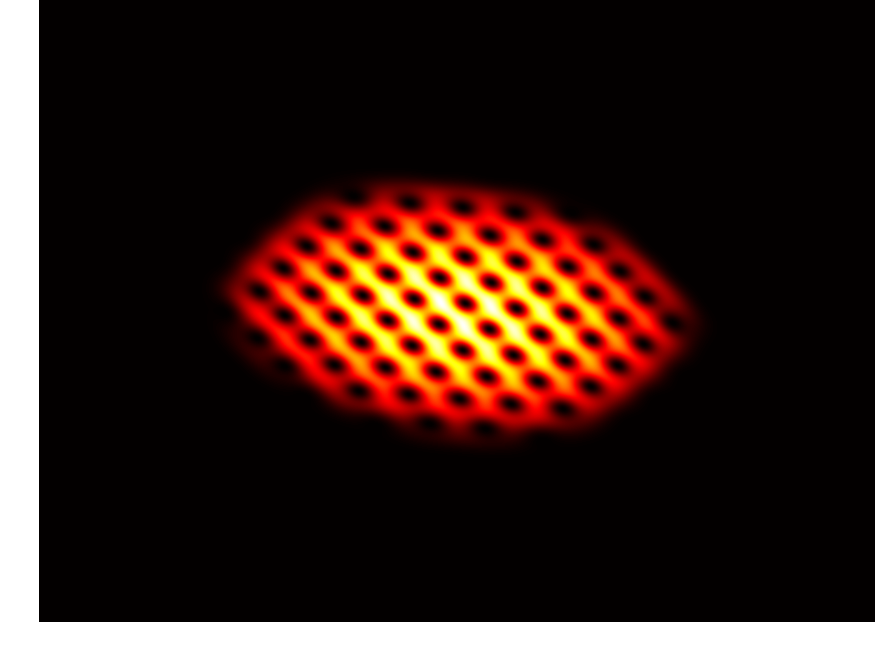}
\includegraphics[width=0.22\textwidth]{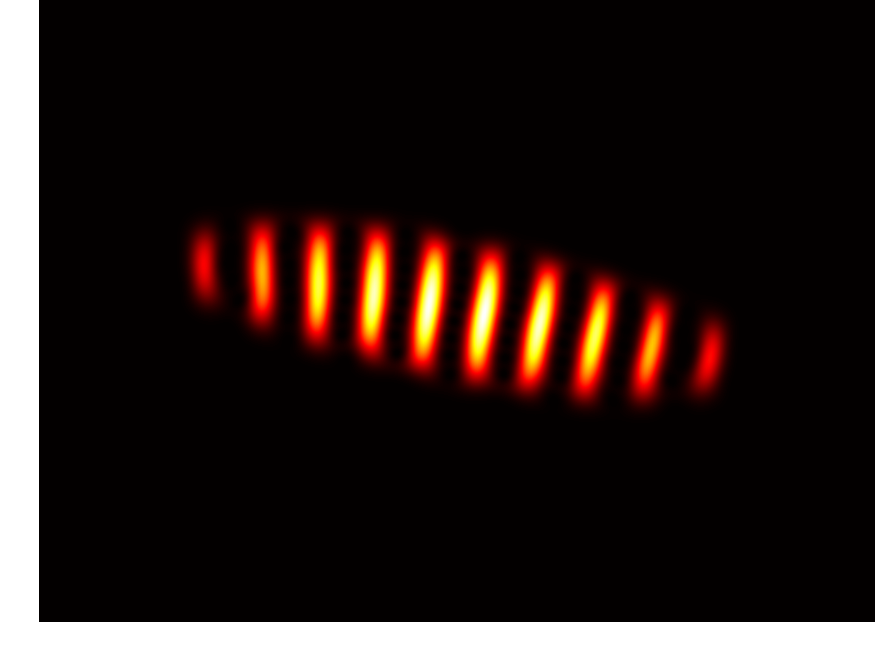}
\includegraphics[width=0.22\textwidth]{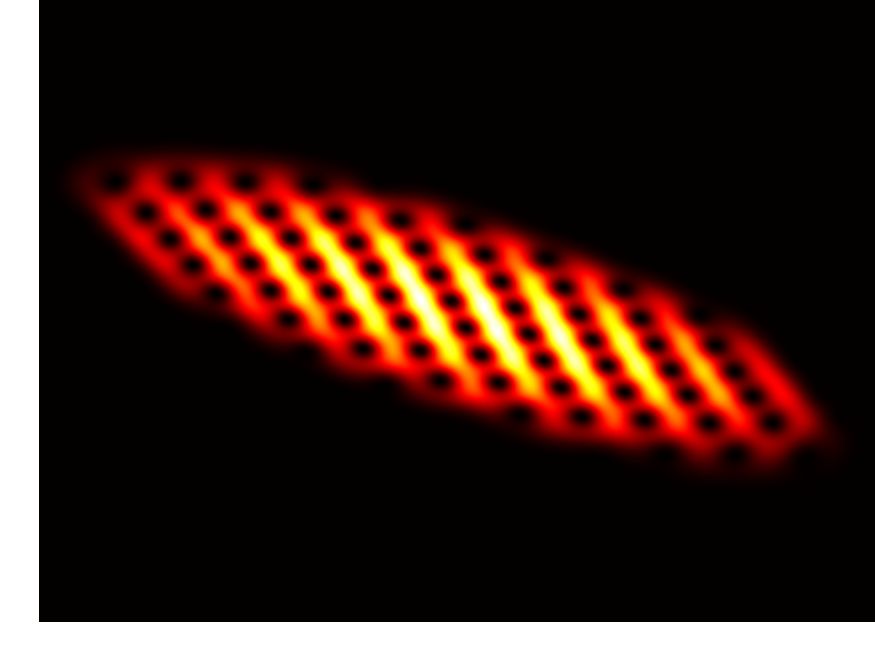}
\includegraphics[width=0.22\textwidth]{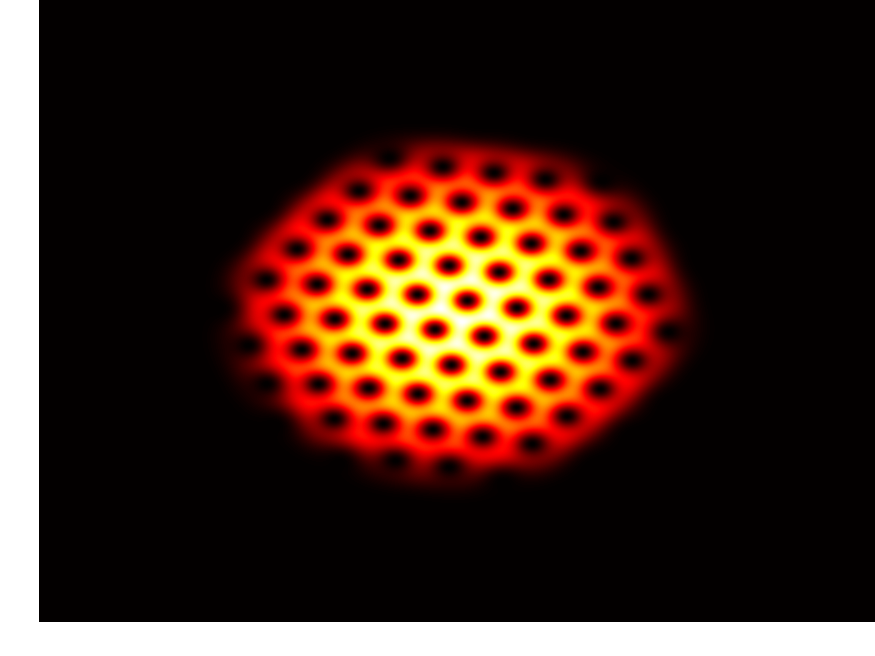}
\includegraphics[width=0.22\textwidth]{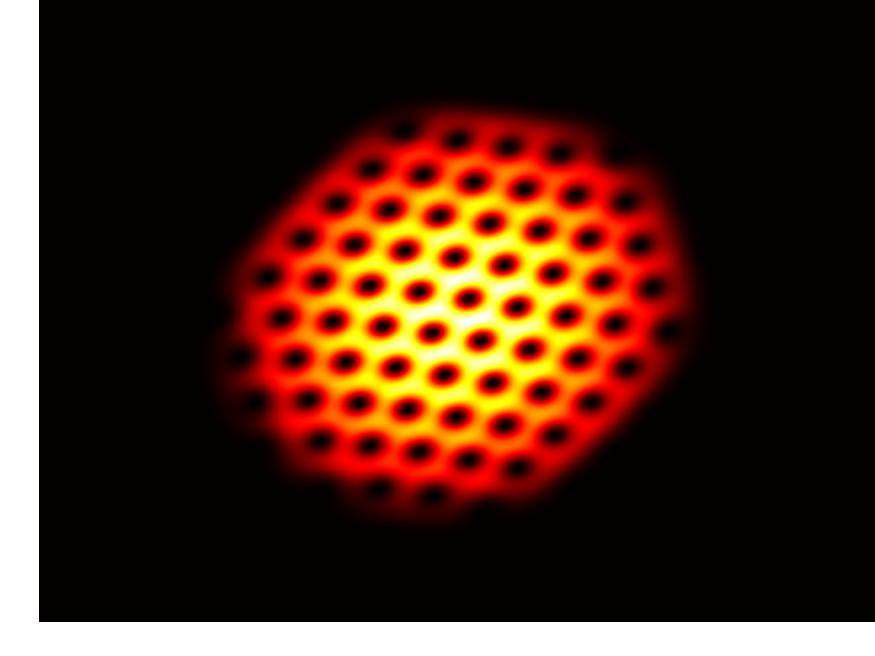}
\includegraphics[width=0.22\textwidth]{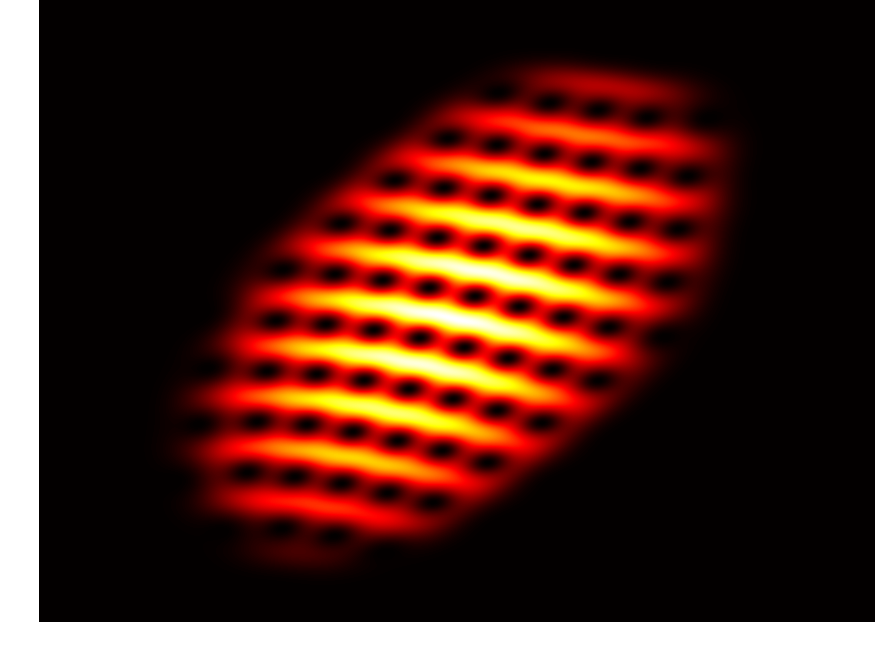}
\includegraphics[width=0.22\textwidth]{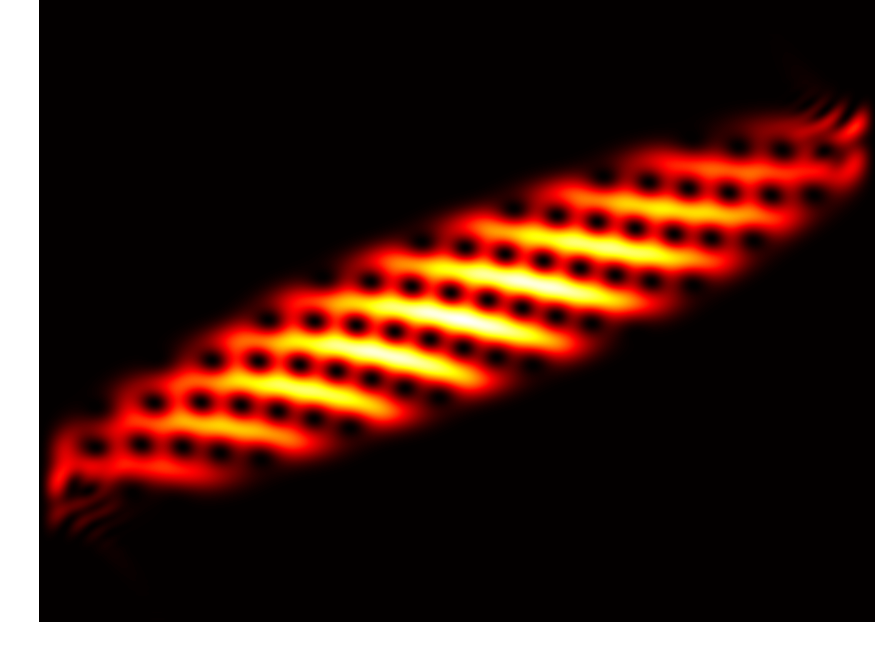}
\includegraphics[width=0.22\textwidth]{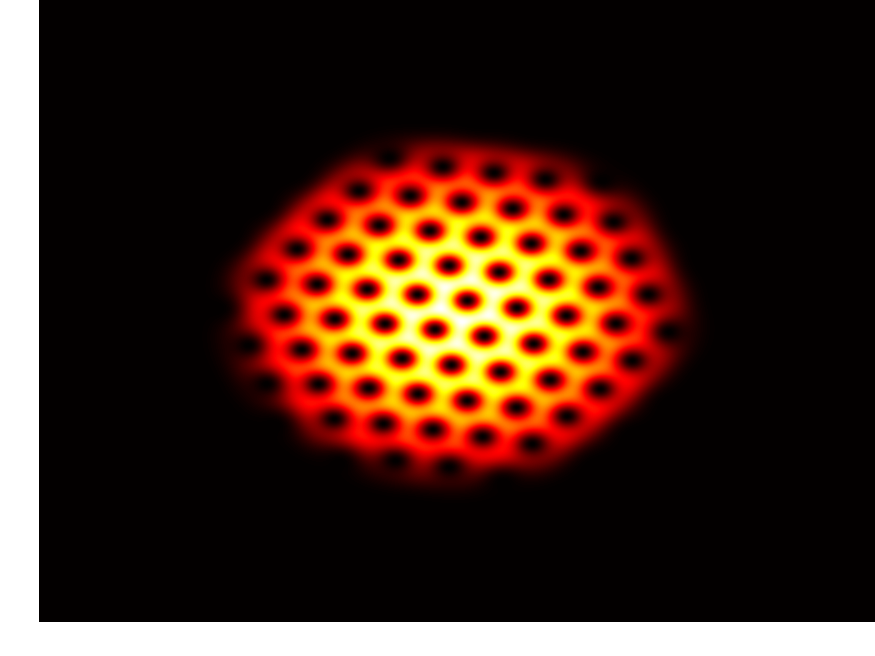}
\includegraphics[width=0.22\textwidth]{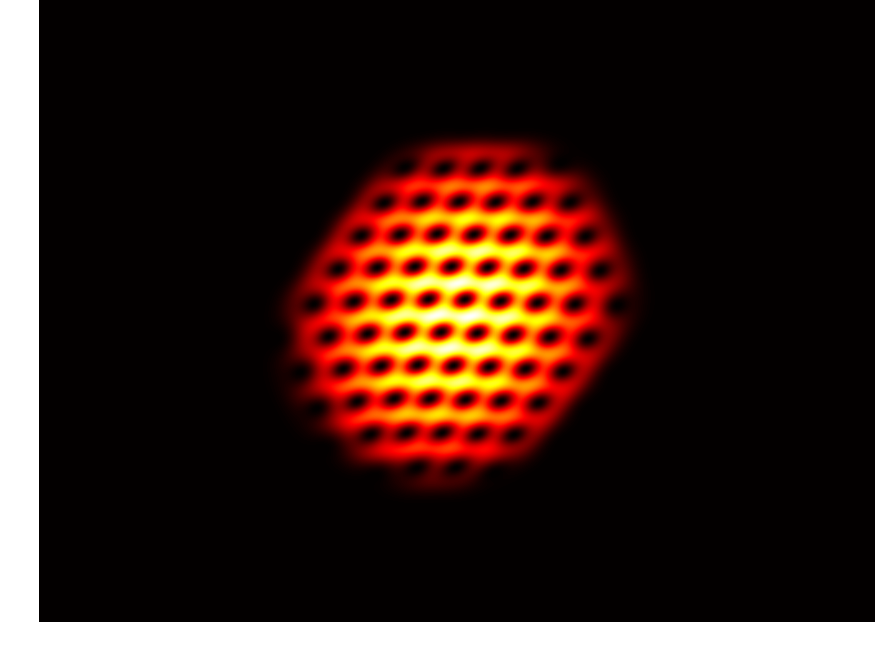}
\includegraphics[width=0.22\textwidth]{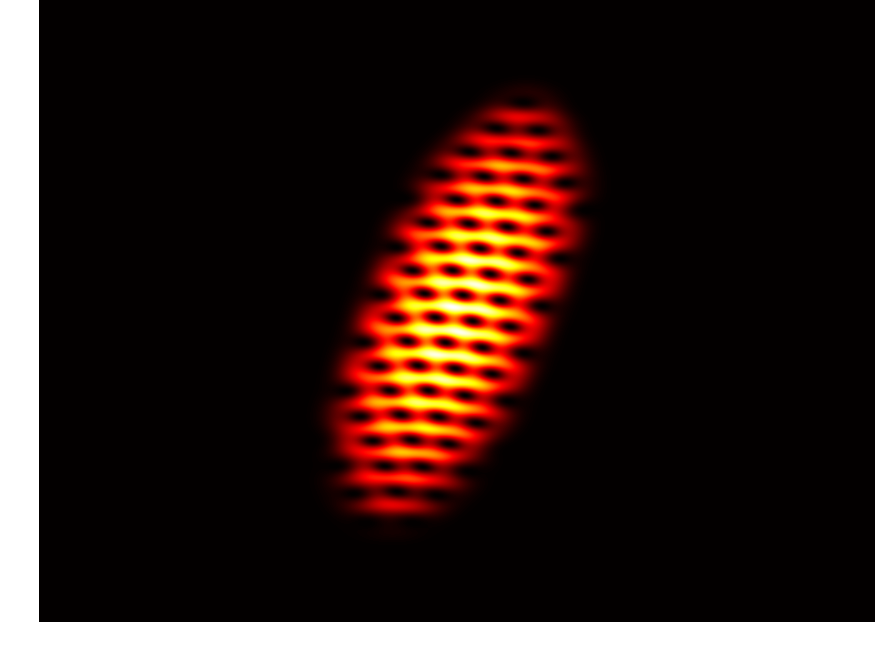}
\includegraphics[width=0.22\textwidth]{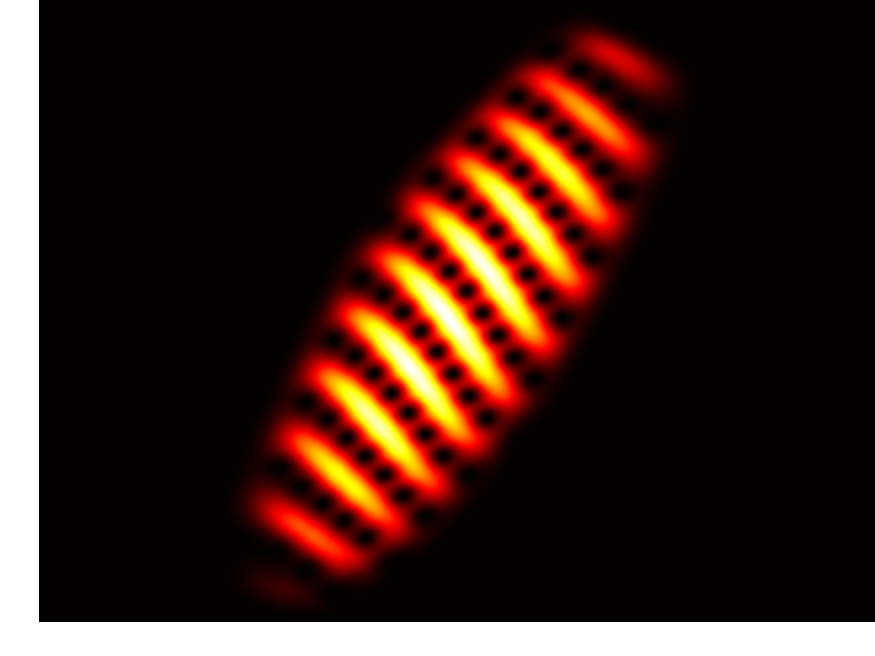}
\includegraphics[width=0.22\textwidth]{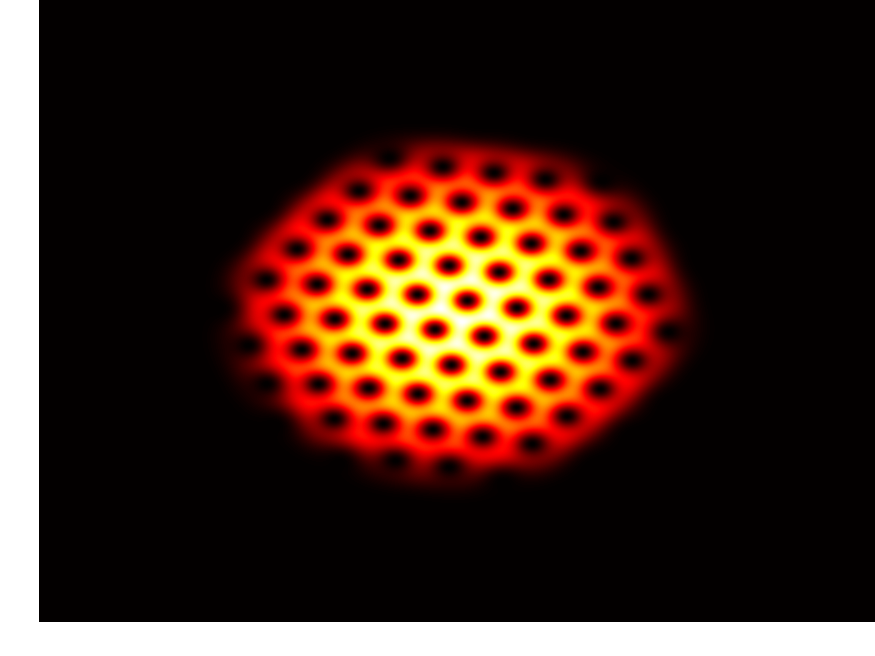}
\includegraphics[width=0.22\textwidth]{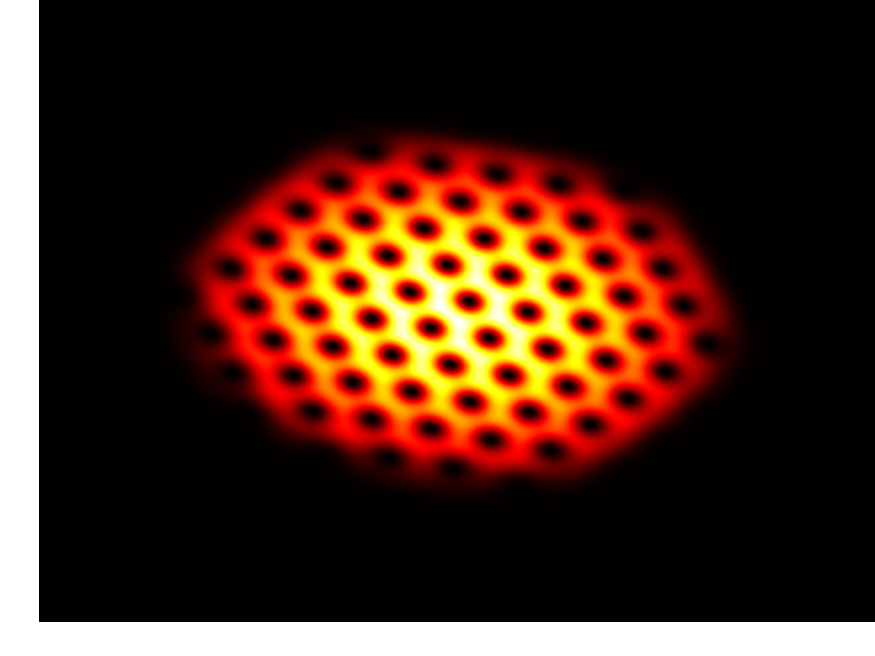}
\includegraphics[width=0.22\textwidth]{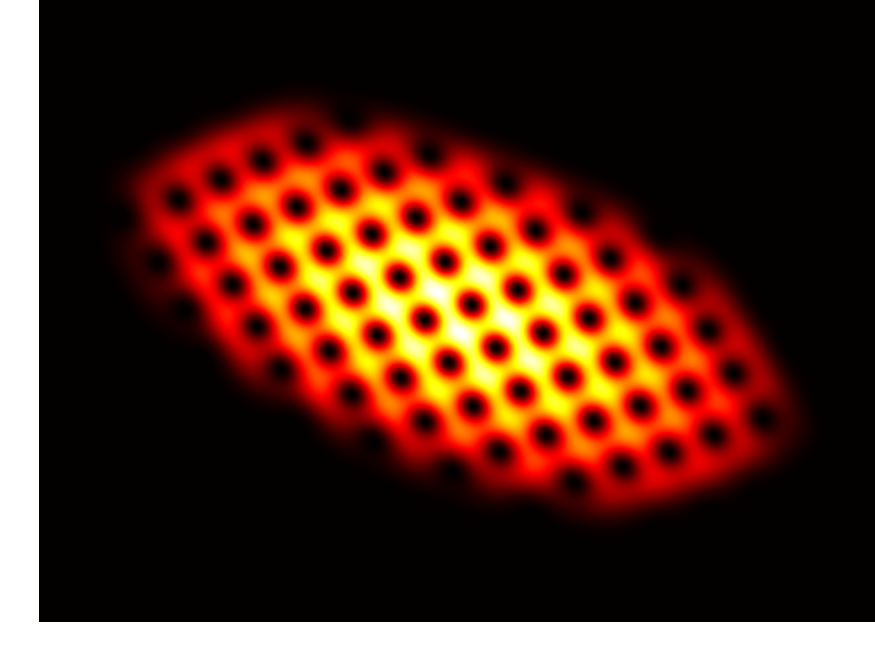}
\includegraphics[width=0.22\textwidth]{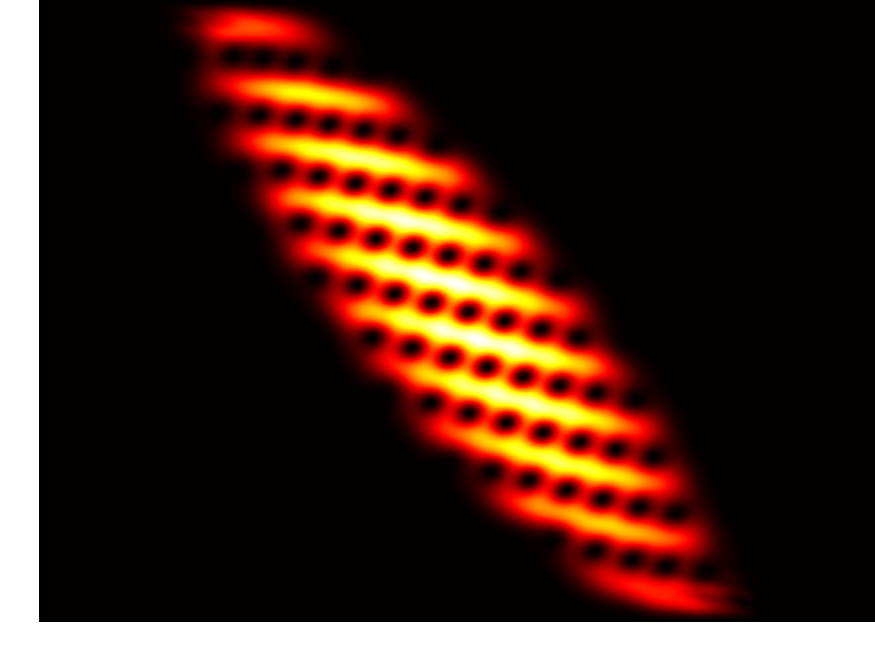}

\includegraphics[width=0.22\textwidth]{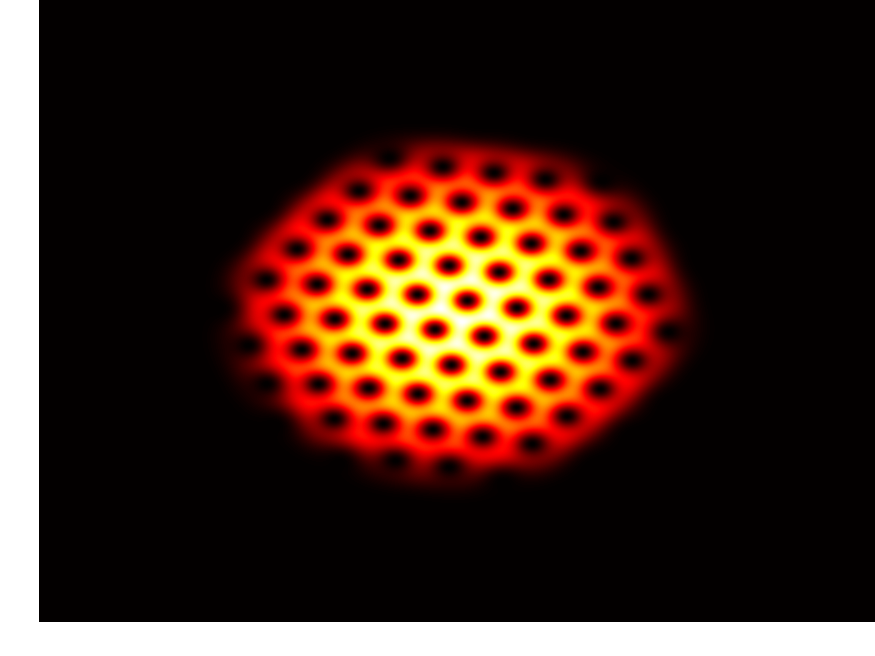}
\includegraphics[width=0.22\textwidth]{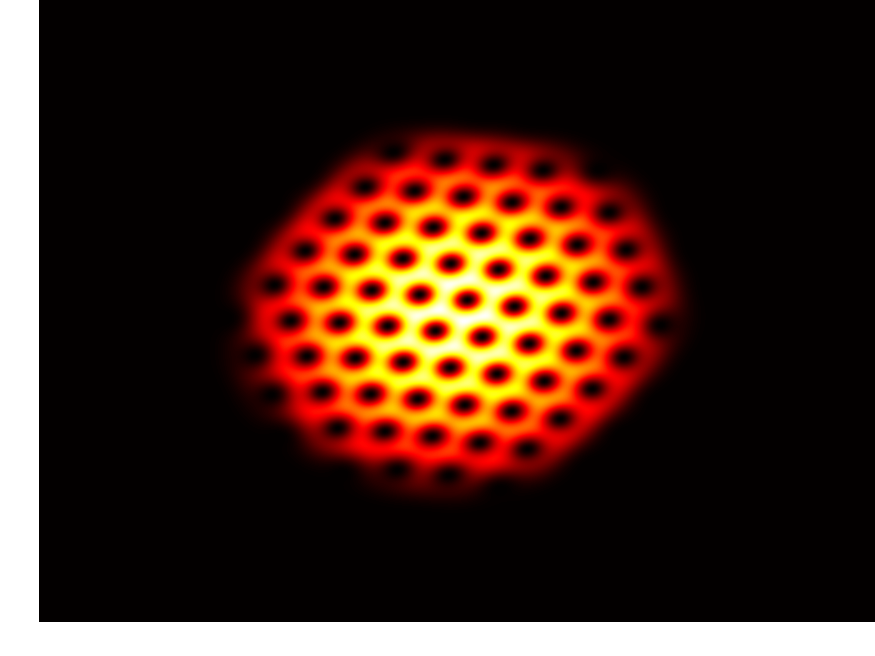}
\includegraphics[width=0.22\textwidth]{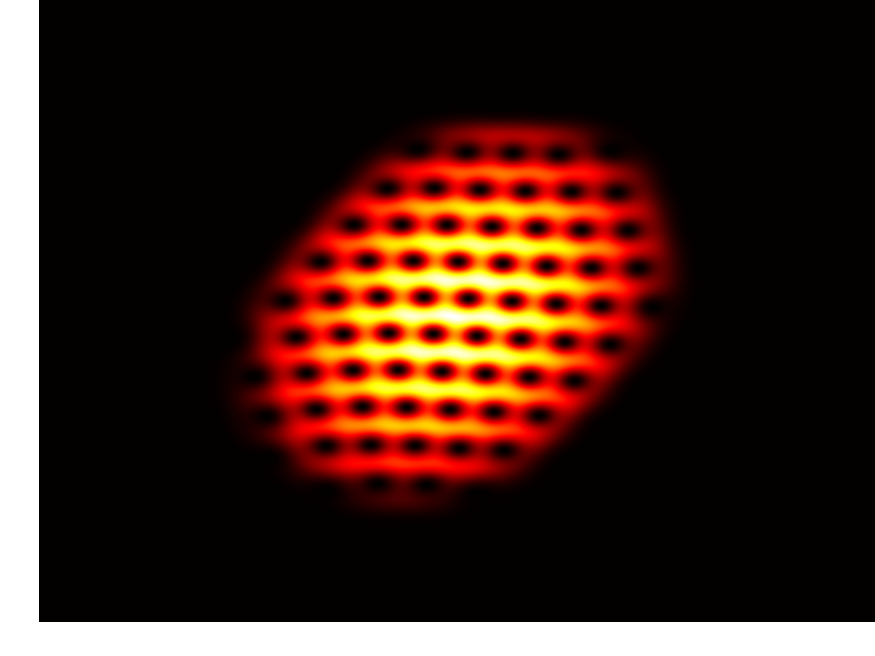}
\includegraphics[width=0.22\textwidth]{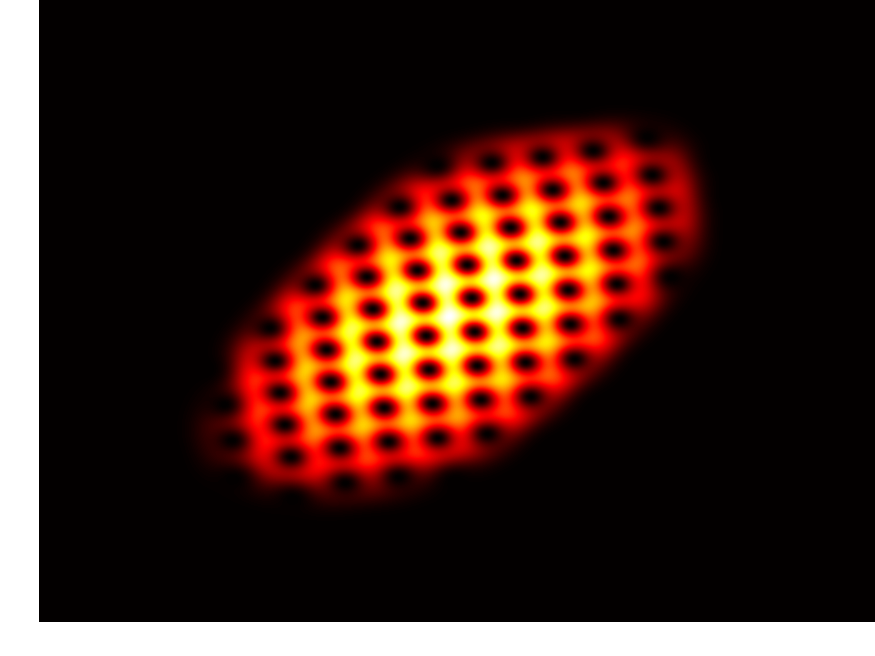}

\includegraphics[width=0.22\textwidth]{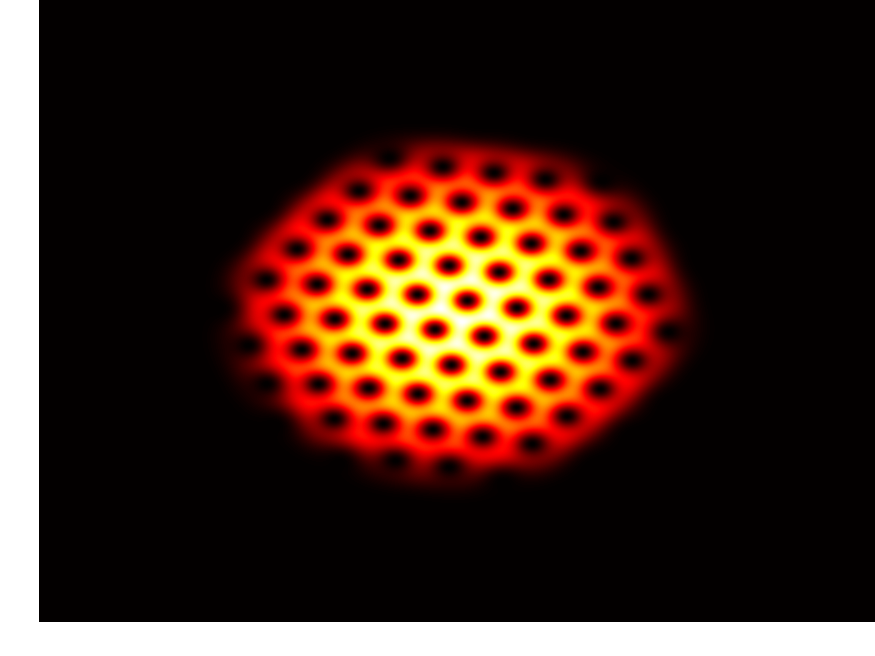}
\includegraphics[width=0.22\textwidth]{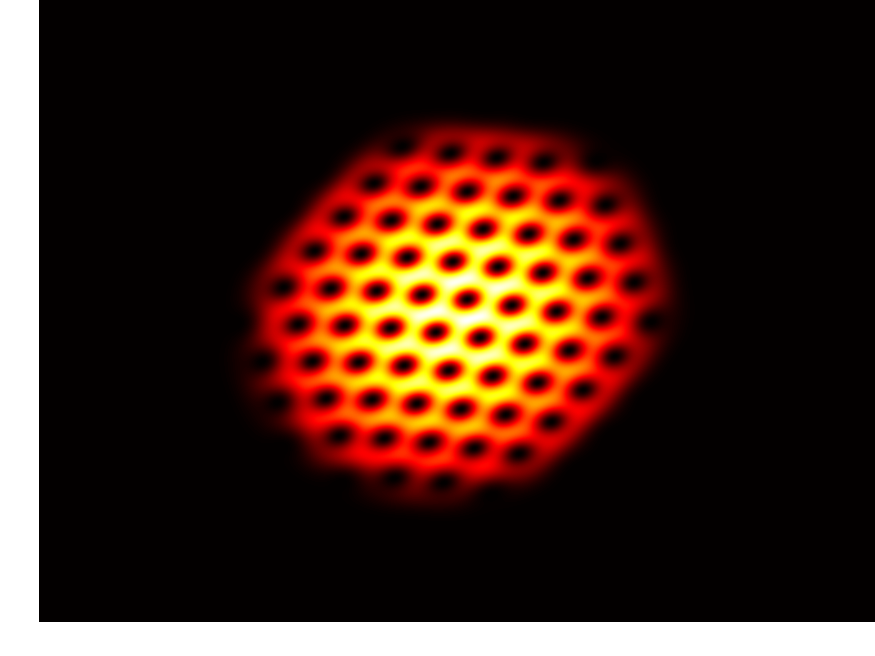}
\includegraphics[width=0.22\textwidth]{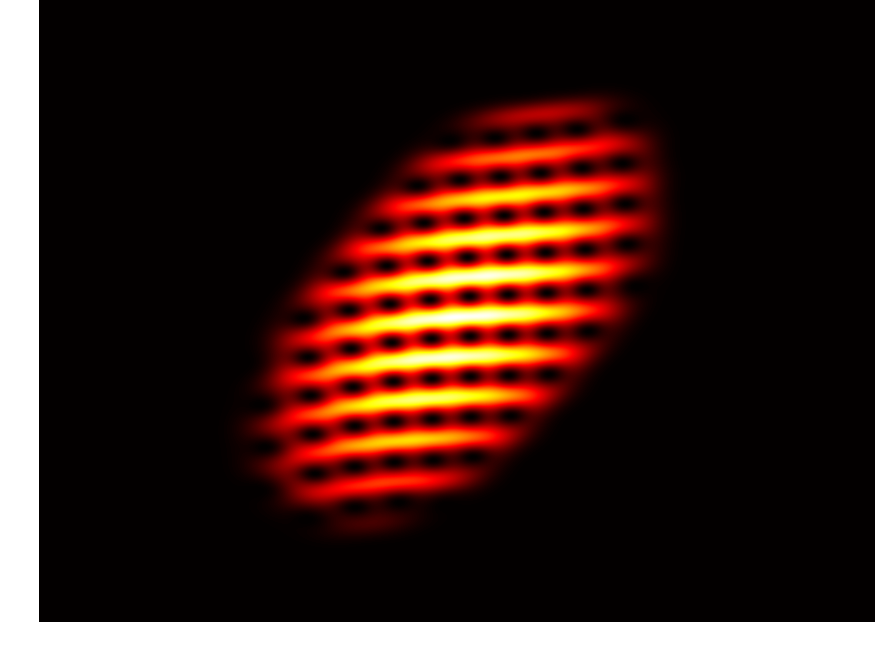}
\includegraphics[width=0.22\textwidth]{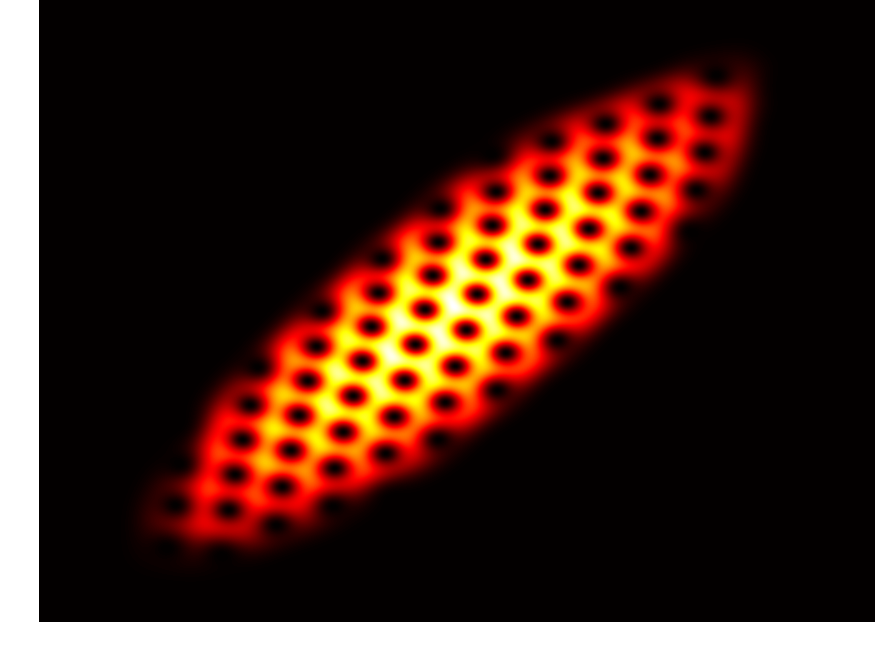}
\caption{Image plots of the density $|u|^2$ at the times $t=0, 0.75, 1.5$ and $3$. Here, we select $h=1/16$, $\tau=0.01$, and $\gamma_x=1, \gamma_y=1.5$ (first row); $\gamma_x=1, \gamma_y=0.5$ (second row); $\gamma_x=1.5, \gamma_y=1$ (third row); $\gamma_x=0.5, \gamma_y=1$ (fourth row); $\gamma_x=\sqrt{1.2}, \gamma_y=\sqrt{0.8}$ (fifth row); $\gamma_x=\sqrt{1.4}, \gamma_y=\sqrt{0.6}$ (sixth row).}
\label{figure:BEC3}
\end{figure}
Similar numerical tests were carried out in \cite{adhikari2002effect,bao2006efficient} by different methods, with much less number of vortices in the lattice.

\end{itemize}

\section{Conclusion} \label{sec6}

This paper focuses on developing and analyzing structure-preserving Galerkin methods for simulating the dynamics of rotating BEC based on the GPE with angular momentum rotation. The challenge lies in constructing FEMs that preserve both mass and energy, particularly in the context of nonconforming FEMs, due to the presence of the rotation term. Furthermore, we provide a comprehensive unconditional error analysis for the structure-preserving FEMs, demonstrating significant improvements compared to existing references, such as Bao and Cai \cite{bao2013optimal}, Henning and Peterseim \cite{henning2017crank}, Henning and M{\aa}lqvist \cite{henning2017finite}, and Henning and W{\"a}rneg{\aa}rd \cite{henning2022superconvergence}. 
To validate the theoretical analysis of the structure-preserving numerical method for rotating BEC, extensive numerical results are presented. The behavior of the quantized vortex lattice is thoroughly examined through a series of numerical tests.

\section*{References}
\bibliographystyle{elsarticle-num}
\bibliography{thebib}

\begin{thebibliography}{10}
\expandafter\ifx\csname url\endcsname\relax
  \def\url#1{\texttt{#1}}\fi
\expandafter\ifx\csname urlprefix\endcsname\relax\def\urlprefix{URL }\fi
\expandafter\ifx\csname href\endcsname\relax
  \def\href#1#2{#2} \def\path#1{#1}\fi

\bibitem{matthews1999vortices}
M.~R. Matthews, B.~P. Anderson, P.~Haljan, D.~Hall, C.~Wieman, E.~A. Cornell,
  Vortices in a {Bose-Einstein} condensate, Physical Review Letters 83~(13)
  (1999) 2498.

\bibitem{madison2000vortex}
K.~W. Madison, F.~Chevy, W.~Wohlleben, J.~Dalibard, Vortex formation in a
  stirred {Bose-Einstein} condensate, Physical Review Letters 84~(5) (2000)
  806.

\bibitem{abo2001observation}
J.~R. Abo-Shaeer, C.~Raman, J.~M. Vogels, W.~Ketterle, Observation of vortex
  lattices in {Bose-Einstein} condensates, Science 292~(5516) (2001) 476--479.

\bibitem{raman2001vortex}
C.~Raman, J.~Abo-Shaeer, J.~Vogels, K.~Xu, W.~Ketterle, Vortex nucleation in a
  stirred {Bose-Einstein} condensate, Physical Review Letters 87~(21) (2001)
  210402.

\bibitem{aftalion2001vortices}
A.~Aftalion, Q.~Du, Vortices in a rotating {Bose-Einstein} condensate: Critical
  angular velocities and energy diagrams in the thomas-fermi regime, Physical
  Review A 64~(6) (2001) 063603.

\bibitem{penckwitt2002nucleation}
A.~Penckwitt, R.~Ballagh, C.~Gardiner, Nucleation, growth, and stabilization of
  {Bose-Einstein} condensate vortex lattices, Physical Review Letters 89~(26)
  (2002) 260402.

\bibitem{adhikari2002effect}
S.~K. Adhikari, P.~Muruganandam, Effect of an impulsive force on vortices in a
  rotating {Bose--Einstein} condensate, Physics Letters A 301~(3-4) (2002)
  333--339.

\bibitem{coddington2004experimental}
I.~Coddington, P.~Haljan, P.~Engels, V.~Schweikhard, S.~Tung, E.~A. Cornell,
  Experimental studies of equilibrium vortex properties in a {Bose-condensed}
  gas, Physical Review A 70~(6) (2004) 063607.

\bibitem{bao2005dynamics}
W.~Bao, Y.~Zhang, Dynamics of the ground state and central vortex states in
  {Bose--Einstein} condensation, Mathematical Models and Methods in Applied
  Sciences 15~(12) (2005) 1863--1896.

\bibitem{bao2005ground}
W.~Bao, P.~A. Markowich, H.~Wang, Ground, symmetric and central vortex states
  in rotating {Bose-Einstein} condensates, Communications in Mathematical
  Sciences 3~(1) (2005) 57--88.

\bibitem{bao2006dynamics}
W.~Bao, Q.~Du, Y.~Zhang, Dynamics of rotating {Bose--Einstein} condensates and
  its efficient and accurate numerical computation, SIAM Journal on Applied
  Mathematics 66~(3) (2006) 758--786.

\bibitem{zhang2007dynamics}
Y.~Zhang, W.~Bao, H.~Li, Dynamics of rotating two-component {Bose--Einstein}
  condensates and its efficient computation, Physica D: Nonlinear Phenomena
  234~(1) (2007) 49--69.

\bibitem{bao2009generalized}
W.~Bao, H.~Li, J.~Shen, {A generalized-Laguerre--Fourier--Hermite
  pseudospectral method for computing the dynamics of rotating Bose--Einstein
  condensates}, SIAM Journal on Scientific Computing 31~(5) (2009) 3685--3711.

\bibitem{Bao2013Kinetic}
W.~Bao, Y.~Cai, Mathematical theory and numerical methods for {Bose-Einstein}
  condensation, Kinetic and Related Models 6~(1) (2013) 1--135.

\bibitem{besse2017high}
C.~Besse, G.~Dujardin, I.~Lacroix-Violet, High order exponential integrators
  for nonlinear {S}chr\"odinger equations with application to rotating
  {B}ose--{E}instein condensates, SIAM Journal on Numerical Analysis 55~(3)
  (2017) 1387--1411.

\bibitem{henning2017finite}
P.~Henning, A.~M{\aa}lqvist, The finite element method for the time-dependent
  {Gross--Pitaevskii} equation with angular momentum rotation, SIAM Journal on
  Numerical Analysis 55~(2) (2017) 923--952.

\bibitem{bao2019ground}
W.~Bao, Y.~Cai, X.~Ruan, Ground states of {Bose--Einstein} condensates with
  higher order interaction, Physica D: Nonlinear Phenomena 386 (2019) 38--48.

\bibitem{da2023vortex}
A.~N. da~Silva, R.~K. Kumar, A.~S. Bradley, L.~Tomio, Vortex generation in
  stirred binary {Bose-Einstein} condensates, Physical Review A 107~(3) (2023)
  033314.

\bibitem{chen2023second}
H.~Chen, G.~Dong, W.~Liu, Z.~Xie, Second-order flows for computing the ground
  states of rotating {Bose-Einstein} condensates, Journal of Computational
  Physics 475 (2023) 111872.

\bibitem{seiringer2002gross}
R.~Seiringer, Gross-pitaevskii theory of the rotating {B}ose gas,
  Communications in Mathematical Physics 229 (2002) 491--509.

\bibitem{aftalion2011non}
A.~Aftalion, R.~L. Jerrard, J.~Royo-Letelier, Non-existence of vortices in the
  small density region of a condensate, Journal of Functional Analysis 260~(8)
  (2011) 2387--2406.

\bibitem{correggi2013inhomogeneous}
M.~Correggi, N.~Rougerie, Inhomogeneous vortex patterns in rotating
  {Bose-Einstein} condensates, Communications in Mathematical Physics 321
  (2013) 817--860.

\bibitem{pitaevskii2016bose}
L.~Pitaevskii, S.~Stringari, Bose-{E}instein condensation and superfluidity,
  Vol. 164, Oxford University Press, 2016.

\bibitem{cazenave2003semilinear}
T.~Cazenave, Semilinear {S}chr{\"o}dinger Equations, Vol.~10, American
  Mathematical Soc., 2003.

\bibitem{hao2007global}
C.~Hao, L.~Hsiao, H.-L. Li, Global well posedness for the gross-pitaevskii
  equation with an angular momentum rotational term in three dimensions,
  Journal of Mathematical Physics 48 (2007) 102105.

\bibitem{lieb2006derivation}
E.~H. Lieb, R.~Seiringer, Derivation of the {Gross-Pitaevskii} equation for
  rotating {B}ose gases, Communications in Mathematical Physics 264 (2006)
  505--537.

\bibitem{williams1998achieving}
J.~Williams, R.~Walser, C.~Wieman, J.~Cooper, M.~Holland, Achieving
  steady-state {Bose-Einstein} condensation, Physical Review A 57~(3) (1998)
  2030.

\bibitem{zapata1998josephson}
I.~Zapata, F.~Sols, A.~J. Leggett, Josephson effect between trapped
  {Bose-Einstein} condensates, Physical Review A 57~(1) (1998) R28.

\bibitem{nikolic2013dipolar}
B.~Nikoli{\'c}, A.~Bala{\v{z}}, A.~Pelster, Dipolar {Bose-Einstein} condensates
  in weak anisotropic disorder, Physical Review A 88~(1) (2013) 013624.

\bibitem{henning2017crank}
P.~Henning, D.~Peterseim, Crank--nicolson {G}alerkin approximations to
  nonlinear {S}chr{\"o}dinger equations with rough potentials, Mathematical
  Models and Methods in Applied Sciences 27~(11) (2017) 2147--2184.

\bibitem{sanz1984methods}
J.~Sanz-Serna, Methods for the numerical solution of the nonlinear
  {S}chr{\"o}dinger equation, Mathematics of Computation 43~(167) (1984)
  21--27.

\bibitem{akrivis1991fully}
G.~D. Akrivis, V.~A. Dougalis, O.~A. Karakashian, On fully discrete galerkin
  methods of second-order temporal accuracy for the nonlinear {S}chr{\"o}dinger
  equation, Numerische Mathematik 59~(1) (1991) 31--53.

\bibitem{tourigny1991optimal}
Y.~Tourigny, Optimal ${H}^1$ estimates for two time-discrete {G}alerkin
  approximations of a nonlinear {S}chr{\"o}dinger equation, IMA Journal of
  Numerical Analysis 11~(4) (1991) 509--523.

\bibitem{bao2013optimal}
W.~Bao, Y.~Cai, Optimal error estimates of finite difference methods for the
  {Gross-Pitaevskii} equation with angular momentum rotation, Mathematics of
  Computation 82~(281) (2013) 99--128.

\bibitem{gagliardo1961proprieta}
E.~Gagliardo, Proprieta di alcune classi di funzioni in piu variabili,
  Matematika 5~(4) (1961) 87--116.

\bibitem{fila2019gagliardo}
M.~Fila, M.~Winkler, A {Gagliardo-Nirenberg-type} inequality and its
  applications to decay estimates for solutions of a degenerate parabolic
  equation, Advances in Mathematics 357 (2019) 106823.

\bibitem{leoni2017first}
G.~Leoni, A first course in {S}obolev spaces, 2nd edition, American
  Mathematical Soc., 2017.

\bibitem{lin2007finite}
Q.~Lin, J.~Lin, Finite element methods: accuracy and improvement, Science
  Press, Beijing, 2006.

\bibitem{bao2006efficient}
W.~Bao, H.~Wang, An efficient and spectrally accurate numerical method for
  computing dynamics of rotating {Bose--Einstein} condensates, Journal of
  Computational Physics 217~(2) (2006) 612--626.

\bibitem{henning2022superconvergence}
P.~Henning, J.~W{\"a}rneg{\aa}rd, Superconvergence of time invariants for the
  {Gross--Pitaevskii} equation, Mathematics of Computation 91~(334) (2022)
  509--555.

\end{thebibliography}

\section*{Appendix: convergence analysis for the nonconforming FEM}
The fully discrete method with truncation for the nonconforming case is given as follows.
\begin{myDef}
	(Fully discrete method with truncation) Let $u_{h,\tau}^{T, 0}$ be a suitable interpolation of $u^0$. Then for $n\geq 1$, we define the following truncated fully discrete system, which is to find $u_{h,\tau}^{T, n+1}\in V_h^{NC}$, $0\leq n\leq N-1$ such that 
	\begin{align}\label{eqn:nonconformingfullydiscrete_truncation}
		i\left(D_\tau u_{h,\tau}^{T, n+\frac{1}{2}}, \omega_h\right)=&\frac{1}{2}\left(\nabla \hat u_{h,\tau}^{T, n+\frac{1}{2}}, \nabla\omega_h\right)_h+\left(V\hat u_{h,\tau}^{T, n+\frac{1}{2}}, \omega_h\right)-\Omega \left(L_z\hat u_{h,\tau}^{T, n+\frac{1}{2}}, \omega_h\right)_h
+\beta\left(\frac{\mu_A(|u_{h, \tau}^{T, n}|^2)+\mu_A(|u_{h, \tau}^{T, n+1}|^2)}{2}	\hat u_{h, \tau}^{T, n+\frac{1}{2}}, \omega_h\right)\nn\\
& +\left\langle S^{T, n+1}, \omega_h\right\rangle,
	\end{align}
for any $\omega_h\in V_h^{NC}$,
where 
\[
\left\langle S^{T, n+1}, \omega_h\right\rangle:=-i\frac{\Omega}{2}Re\left\langle (x-y)\hat u_{h,\tau}^{n+\frac{1}{2}}, \omega_h\right\rangle+\frac{\Omega}{2}Im\left\langle (x-y)u_{h,\tau}^{n+1},  \omega_h\right\rangle.
\]
\end{myDef}

We will use the interpolation operator to split the error. For convenience, we still adopt the same denotations:  
\begin{align*}
	e_{h,\tau}^{T, n}=u_\tau^n-u_{h, \tau}^{T, n}=\left(u_\tau^n-I_hu_\tau^n\right)+\left(I_hu_\tau^n-u_{h, \tau}^{T, n}\right)=:\rho_{h,\tau}^{T,n}+\theta_{h,\tau}^{T,n}.
\end{align*}
Subtracting \eqref{eqn:nonconformingfullydiscrete_truncation} from the variational  formulation of the time-discrete method \eqref{eqn:timediscrete} gives 
	\begin{align}\label{eqn:NC_spacetruncated_error_pf1}
		i\left(D_\tau e_{h,\tau}^{T, n+\frac{1}{2}}, \omega_h\right)=&\frac{1}{2}\left(\nabla \hat e_{h,\tau}^{T, n+\frac{1}{2}}, \nabla\omega_h\right)_h+\left(V\hat e_{h,\tau}^{T, n+\frac{1}{2}}, \omega_h\right)-\Omega \left(L_z\hat e_{h,\tau}^{T, n+\frac{1}{2}}, \omega_h\right)_h
+\left(\mathcal N_{h,\tau}^{T, n+\frac{1}{2}}, \omega_h\right)\nn\\
&-\left\langle S^{T, n+1}, \omega_h\right\rangle- \frac12\sum_K\int_{\partial K}\left(\nabla\hat u_\tau^{n+\frac12}\cdot \mathbf n\right)\omega_hds\qquad \forall \omega_h\in V_h^{NC}.
	\end{align}
Noting the property \eqref{eqn:Ih_error}, \eqref{eqn:NC_spacetruncated_error_pf1} is indeed equivalent to 
\begin{align}\label{eqn:NC_spacetruncated_error_pf2}
		i\left(D_\tau \theta_{h,\tau}^{T, n+\frac{1}{2}}, \omega_h\right)=&-i\left(D_\tau \rho_{h,\tau}^{T, n+\frac{1}{2}}, \omega_h\right)+\frac{1}{2}\left(\nabla \hat \theta_{h,\tau}^{T, n+\frac{1}{2}}, \nabla\omega_h\right)_h+\left(V\hat e_{h,\tau}^{T, n+\frac{1}{2}}, \omega_h\right)-\Omega \left(L_z\hat \rho_{h,\tau}^{T, n+\frac{1}{2}}, \omega_h\right)_h
+\left(\mathcal N_{h,\tau}^{T, n+\frac{1}{2}}, \omega_h\right)\nn\\
&
-\Omega \left(L_z\hat \theta_{h,\tau}^{T, n+\frac{1}{2}}, \omega_h\right)_h
-\left\langle S^{T, n+1}, \omega_h\right\rangle- \frac12\sum_K\int_{\partial K}\left(\nabla\hat u_\tau^{n+\frac12}\cdot \mathbf n\right)\omega_hds.
	\end{align}
Denoting $\omega_h=\hat \theta_{h,\tau}^{T, n+\frac{1}{2}}$ in \eqref{eqn:NC_spacetruncated_error_pf2}, and taking the imaginary part of the equation, we obtain 
\begin{align}\label{eqn:NC_spacetruncated_error_pf3}
	\frac{\|\theta_{h,\tau}^{T,n+1}\|_{L^2}^2-\|\theta_{h,\tau}^{T,n}\|_{L^2}^2}{2\tau}
= &-Re\left(D_\tau \rho_{h,\tau}^{T, n+\frac{1}{2}}, \hat \theta_{h,\tau}^{T, n+\frac{1}{2}}\right)	+Im\left(V\hat \rho_{h,\tau}^{T, n+\frac{1}{2}}, \hat \theta_{h,\tau}^{T, n+\frac{1}{2}}\right)-\Omega Im\left(L_z\hat \rho_{h,\tau}^{T, n+\frac{1}{2}}, \hat \theta_{h,\tau}^{T, n+\frac{1}{2}}\right)_h+Im\left(\mathcal N_{h,\tau}^{T, n+\frac{1}{2}}, \hat \theta_{h,\tau}^{T, n+\frac{1}{2}}\right)
\nn\\
		&
		-\Omega Im\left(L_z\hat \theta_{h,\tau}^{T, n+\frac{1}{2}}, \hat \theta_{h,\tau}^{T, n+\frac{1}{2}}\right)_h
		-Im\left\langle S^{T, n+1}, \hat \theta_{h,\tau}^{T, n+\frac{1}{2}}\right\rangle
		- 
		Im\left\{\frac12\sum_K\int_{\partial K}\left(\nabla\hat u_\tau^{n+\frac12}\cdot \mathbf n\right)\hat \theta_{h,\tau}^{T, n+\frac{1}{2}}ds\right\}.
	\end{align}
Compared with the conforming case, we only need to consider the last three terms of \eqref{eqn:NC_spacetruncated_error_pf3}. For the last term of \eqref{eqn:NC_spacetruncated_error_pf4}, we have 
\begin{align}\label{eqn:NC_spacetruncated_error_pf5}
	&- 	Im\left\{\frac12\sum_K\int_{\partial K}\left(\nabla\hat u_\tau^{n+\frac12}\cdot \mathbf n\right)\hat \theta_{h,\tau}^{T, n+\frac{1}{2}}ds\right\}
	=Im\left\{\frac12\sum_K\int_{\partial K}\left(\nabla\hat e_\tau^{n+\frac12}\cdot \mathbf n\right)\hat \theta_{h,\tau}^{T, n+\frac{1}{2}}ds\right\}
	-Im\left\{\frac12\sum_K\int_{\partial K}\left(\nabla\hat u_\tau^{n+\frac12}\cdot \mathbf n\right)\hat \theta_{h,\tau}^{T, n+\frac{1}{2}}ds\right\}\nn\\
	&\qquad \leq Ch\|\hat e_\tau^{n+\frac12}\|_{H^2}\|\hat \theta_{h,\tau}^{T, n+\frac{1}{2}}\|_{1,h}
	+Ch^2\|\hat u^{n+\frac12}\|_{H^3}
	\|\hat \theta_{h,\tau}^{T, n+\frac{1}{2}}\|_{1,h}
	\leq C\left(h\tau^2+h^2\right)\|\hat \theta_{h,\tau}^{T, n+\frac{1}{2}}\|_{1,h}.
\end{align}
By the definition of $\langle S^{T, n+1}, \cdot\rangle$, using integration by parts, by virtue of  
$\left\langle\cdot, \omega_h\right\rangle=O(h^2)\|\cdot\|_{H^2}\|\omega_h\|_{1,h}$, thanks to the boundedness of $\hat \rho_{h,\tau}^{T, n+\frac{1}{2}}$ in $H^2$-norm, 
and similar as \eqref{eqn:NC_spacetruncated_error_pf5}, we can get  
\begin{align}\label{eqn:NC_spacetruncated_error_pf4}
	&-\Omega Im\left(L_z\hat \theta_{h,\tau}^{T, n+\frac{1}{2}}, \hat \theta_{h,\tau}^{T, n+\frac{1}{2}}\right)_h
		-Im\left\langle S^{T, n+1}, \hat \theta_{h,\tau}^{T, n+\frac{1}{2}}\right\rangle
		\leq C\left(h\tau^2+h^2\right)\|\hat \theta_{h,\tau}^{T, n+\frac{1}{2}}\|_{1,h},
\end{align}

Hence, from \eqref{eqn:NC_spacetruncated_error_pf3} and similar as the conforming case, we can obtain  
\begin{align}\label{eqn:NC_spacetruncated_error_pf6}
	\|\theta_{h,\tau}^{T,n+1}\|_{L^2}^2
	\leq C_{E,h}\left(h^2\tau^4+h^4\right)+C\tau\sum_{k=1}^n\|\theta_{h,\tau}^{T,k}\|_{1,h}^2.
\end{align}

In addition, similar as the conforming case \eqref{eqn:super_pf4} and using above analytical approach, we can further obtain 
\begin{align}\label{eqn:NC_spacetruncated_error_pf7}
	\|\theta_{h,\tau}^{T,n+1}\|_{h}^2
	\leq C_{E,h}\left(h^2\tau^4+h^4\right)+C\tau\sum_{k=1}^n\left(\|\theta_{h,\tau}^{T,k}\|_{1,h}^2+\|D_\tau\theta_{h,\tau}^{T,k+\frac12}\|_{1,h}^2\right).
\end{align}
Moreover, like \eqref{eqn:super_pf5}, we can also take the difference between two consecutive steps, and then adopt similar steps as conforming case to finally obtain that 
\begin{align}\label{eqn:NC_spacetruncated_error_pf8}
	\|D_\tau\theta_{h,\tau}^{T,k+\frac12}\|_{h}^2
	\leq C_{E,h}\left(h^2\tau^4+h^4\right)+C\tau\sum_{k=1}^n\|\theta_{h,\tau}^{T,k}\|_{1,h}^2.
\end{align}
From \eqref{eqn:NC_spacetruncated_error_pf6}, \eqref{eqn:NC_spacetruncated_error_pf7} and \eqref{eqn:NC_spacetruncated_error_pf8}, we get 
\begin{align}
	\|\theta_{h,\tau}^{T,n+1}\|_{L^2}^2+\|\theta_{h,\tau}^{T,n+1}\|_{h}^2+\|D_\tau\theta_{h,\tau}^{T,n+\frac12}\|_{h}^2
	\leq C_{E,h}\left(h^2\tau^4+h^4\right)
	+C\tau\sum_{k=1}^n\left[\|\theta_{h,\tau}^{T,k}\|_{L^2}^2+\|\theta_{h,\tau}^{T,k}\|_{h}^2+\|D_\tau\theta_{h,\tau}^{T,k+\frac12}\|_{h}^2\right]. 
\end{align}
In conclusion, by applying the discrete Gronwall's inequality, we derive the desired convergence results that are identical to those in the conforming case. 
\end{document}